\documentclass[12pt]{article}
\usepackage{amsmath,latexsym,amssymb,amsthm,enumerate,amsmath,amscd}
\usepackage{amsthm,amsmath,amsfonts,amssymb,amscd,latexsym,stmaryrd}
\usepackage[mathscr]{eucal}
\usepackage{color}
\usepackage[usenames,dvipsnames]{xcolor}
\usepackage{colortbl}
\usepackage[all,cmtip]{xy}
\usepackage{graphicx}
\usepackage{ytableau}
\usepackage{comment}
\definecolor{light-gray}{gray}{0.9}
\definecolor{med-gray}{gray}{0.5}
\definecolor{gray1}{gray}{0.87}
\definecolor{gray2}{gray}{0.74}
\definecolor{gray3}{gray}{0.64}
\definecolor{gray4}{gray}{0.55}
\definecolor{verylight-yellow}{rgb}{1,1,0.7}
\definecolor{beaublue}{rgb}{0.74, 0.83, 0.9}
\definecolor{yellow}{rgb}{1,1,0.2}
\definecolor{vivid-blue}{rgb}{0.2,0,1}
\definecolor{light-pink}{rgb}{1,0.8,1}
\definecolor{med-pink}{rgb}{1,0.6,1}
\definecolor{aqua}{rgb}{0.0, 1.0, 1.0}
\definecolor{light-gray}{rgb}{0.5, 0.9, 0.5}
\definecolor{cadmiumgreen}{rgb}{0.0, 0.42, 0.24}

\usepackage{comment}
\usepackage{hyperref}
\usepackage{tikz}
\usetikzlibrary{calc,shadings,patterns}
\theoremstyle{plain}
\newtheorem{theorem}{Theorem}[section]
\newtheorem{proposition}[theorem]{Proposition}
\newtheorem{corollary}[theorem]{Corollary}
\newtheorem{lemma}[theorem]{Lemma}
\theoremstyle{definition}
\newtheorem{definition}[theorem]{Definition}

\newtheorem{remark}[theorem]{Remark}
\newtheorem{example}[theorem]{Example}

\newtheorem{thm}{Theorem}

\newtheorem{rem}[thm]{Remark}

\newtheorem*{remark*}{Remark}

\newtheorem*{ack}{Acknowledgment}

\numberwithin{equation}{section}
\numberwithin{table}{section}
\definecolor{purple}{rgb}{0.4,0.2,0.4}

\def\rk{\mathrm{rk}}

\def\cha{\mathrm{char}\ }

\DeclareMathOperator{\Ann}{Ann}

\def\<{\left<}
\def\>{\right>}

\def\G{\mathrm{G}}
\def\Hess{\mathrm{Hess}}

\def\ns{\footnotesize \it}

\def\max{\mathrm{max}}

\definecolor{med-gray}{gray}{0.5}
\definecolor{gray1}{gray}{0.87}
\definecolor{gray2}{gray}{0.74}
\definecolor{gray3}{gray}{0.64}
\definecolor{gray4}{gray}{0.48}
\definecolor{verylight-yellow}{rgb}{1,1,0.7}
\definecolor{yellow}{rgb}{1,1,0.2}
\definecolor{vivid-blue}{rgb}{0.2,0,1}
\definecolor{light-pink}{rgb}{1,0.8,1}
\definecolor{med-pink}{rgb}{1,0.6,1}
\definecolor{aqua}{rgb}{0.0, 1.0, 1.0}
\definecolor{light-gray}{rgb}{0.5, 0.9, 0.5}

\def\cha{\mathrm{char}\ }

\DeclareMathOperator{\hess}{Hess}
\author{Nasrin Altafi\\[.05in]
{\ns Department of Mathematics, KTH Royal Institute of Technology, S-100 44 Stockholm, Sweden.
}\\[.2in] Anthony Iarrobino\\[.05in]
{\ns Department of Mathematics, Northeastern University, Boston, MA 02115,
 USA.
}\\[.2in] 
Leila Khatami\\[0.05in]
{\ns Department of Mathematics, Union College, Schenectady, New York, 
 12308, USA.}
}

\date{April 24, 2020}

\begin{document}
\title{Complete intersection Jordan types in height two
}

\maketitle

\begin{abstract} We determine every Jordan type partition that occurs as the Jordan block decomposition for the multiplication map by a linear form in a height two homogeneous complete intersection (CI) Artinian algebra $A$ over an algebraically closed field $\sf k$ of characteristic zero or large enough. We show that these CI Jordan type partitions are those satisfying specific numerical conditions; also, given the Hilbert function $H(A)$, they are completely determined by which higher Hessians of $A$ vanish at the point corresponding to the linear form. We also show new combinatorial results about such partitions, and in particular we give ways to construct them from a branch label or hook code, showing how branches are attached to a fundamental triangle to form the Ferrers diagram. \end{abstract}\medskip
\tableofcontents
\section{Introduction.}\label{Introsec}
Let $A$ be a graded Artinian algebra over an infinite field $\sf k$. We assume that $A$ has a single maximal ideal $\mathfrak m$ and that $A/\mathfrak m=\sf k$. The Jordan type $P_\ell=P_{\ell,A}$ of a linear form $\ell$ of $A_1$ is the partition determining the Jordan block decomposition for the (nilpotent) multiplication map $m_\ell$ by $\ell$ on $A$. Such a partition must have diagonal lengths the Hilbert function $H(A)$ (Lemma \ref{diaglem}A). The diagonal lengths of a partition refer to the lengths of diagonals of slope 1 in the Ferrers diagram of the partition (see Figure ~\ref{newfigfordiagonal}). 

\begin{figure}[hb]
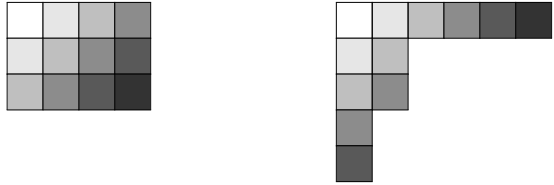

$$\begin{array}{ccc}
\scalebox{0.75}{
\begin{ytableau}
\, &*(black! 10)&*(black! 25)&*(black! 45)\cr
*(black! 10)&*(black! 25)&*(black! 45)&*(black! 65)\cr
*(black! 25)&*(black! 45)&*(black! 65)&*(black! 80)\cr
\end{ytableau}}
&\quad \quad \quad \quad \quad&
\scalebox{0.75}{
\begin{ytableau}
\, &*(black! 10)&*(black! 25)&*(black! 45)&*(black! 65)&*(black! 80)\cr
*(black! 10)&*(black! 25)\cr
*(black! 25)&*(black! 45)\cr
*(black! 45)\cr
*(black! 65)\cr
\end{ytableau}}

\end{array}$$
\caption{Ferrers diagrams of partitions  $P=(4^3)$  and $Q=(6, 2^2, 1^2)$ with diagonal lengths $(1,2,3,3,2,1)$.}\label{newfigfordiagonal}

\end{figure}
We denote the set of all partitions having diagonal lengths $T$ by $\mathcal P(T)$. The Sperner number Sp$(T)$ of a Hilbert function sequence $T$ is its height Sp$(T)=\max_i(T_i)$. We say that a pair $(A,\ell)$ with $A$ a graded Artinian algebra and $\ell \in A_1$ a linear element of $A$, has the weak Lefschetz (WL) property if the partition $P_\ell$ has  Sp$(T)$ parts, where $T=H(A)$. This is the smallest possible number of parts possible given $T$ (\cite[Proposition 3.64]{H-W}). We will term such a partition $P$ of diagonal lengths $T$ having Sp$(T)$ parts a \emph{weak Lefschetz} partition for $T$. The pair $(A,\ell)$ is termed \emph{strong Lefschetz} (SL) if $P_\ell=T^\vee$ where $ T=H(A)$. Here $T^\vee=[T]^\vee$ the conjugate partition (switch rows and columns in the Ferrers graph) to $[T]$, the set of values of $T$.\par

% W\footnote{Since $\sf k$ is infinite, a generic element of $A_1$ is one not on a certain proper subvariety of $A_1\cong {\sf k}^2$, that depends on $A$.}  $\ell \in A_1$: $A$ is strong Lefschetz if the Jordan type $P_\ell$ of $(\ell,A)$ is the conjugate partition $T^\vee$ to that determined by the Hilbert function $T=H(A)$ \cite[Proposition 3.64]{H-W}; and $A$ is weak Lefschetz if $P_\ell$ has the smallest possible number of parts, namely the height of $H(A)$ (sometimes called the Sperner number of $A$) \cite[Remark 3.63]{H-W}. 
 There have been many studies of graded Artinian algebras satisfying the strong or weak Lefschetz property for a generic element $\ell\in A_1$ (see \cite{H-W} and the references cited there).  Recently, there have been studies of more general questions about the Jordan type of pairs $(A,\ell)$ (see \cite{Go,GoZ,H-W,IMM,MW} and references cited.)  
 %All height two Artinian algebras $A$ satisfy $(A,\ell)$ is strong Lefschetz for a generic element of $A_1$, provided $\cha \sf k=0$ or is large enough. 
 We study in this paper which Jordan types $P_\ell$ can occur for arbitrary, usually non-generic elements of $A_1$, when $A$ is a graded Gorenstein quotient of the polynomial ring $R={\sf k}[x,y]$.   By a result of F.H.S. Macaulay \cite[\S 14]{Mac} the Artinian Gorenstein algebras of height two (codimension two) are complete intersections (CI).\footnote{F.H.S. Macaulay refers to this also in
  \cite[\S 71]{Mac2}; a homological version of this result was given by F. P. Serre \cite[Proposition 3]{Se}, see also C. Huneke's survey \cite[\S 4]{Hu}.}  
 \begin{definition}\label{CIJTdef}  We say that a partition $P$ of diagonal lengths $T$ is a \emph{complete intersection Jordan type} (CIJT) if it can occur as a partition $P_{\ell,A}$ for a graded complete intersection quotient $A=R/I, R={\sf k}[x,y]$ and some linear form  $\ell\in A_1$.
 \end{definition}
 It is well known (see \cite[\S 58]{Mac2}) that the Hilbert function of a standard-graded CI quotient $A$ of $R$ satisfies $H(A)=T$, a symmetric sequence of the form
\begin{equation}\label{CITeq}
T=(1_0,2_1,\ldots, (d-1)_{d-2},d_{d-1},\ldots, d_{d+k-2},(d-1)_{d+k-1},\ldots , 2_{j-1},1_j).
\end{equation}
Here $k$ is the multiplicity in $T$ of the height $d$, and the subscripts indicate degree.  In order to simplify results we will assume throughout that $\cha \sf k$ is zero, or is greater than $j$, the socle degree of $A$ -- except that in discussing Hessians, we will assume $\cha \sf k$ is zero. It is well known that for these characteristics, for a fixed codimension two algebra $A=R/I$ and a generic linear form $\ell\in A_1$ the pair $(A,\ell)$ is strong Lefschetz.\footnote{See \cite{Bri}, and discussion in \cite[Lemma 2.14]{IMM}]; this result depends on a standard basis argument of J. Brian\c{c}on and has been reproved many times.} 
Here $\ell\in A_1={\sf k}^2$ is \emph{generic} for $A$ if it is not -- up to constant multiple -- one of a finite number of exceptional linear forms. In this paper, we are interested primarily in the Jordan types possible for the exceptional linear forms. We first determine all CIJT partitions of diagonal lengths $T$ (Theorems~\ref{CIpartsthm} and \ref{CIpartsthm2}). As a consequence we give, surprisingly, a criterion for $P$ to be CIJT using just the number of parts of $P$ (Theorem \ref{simplegoodprop}): 
  \begin{thm} A partition $P$ having diagonal lengths $T$ satisfying Equation \eqref{CITeq} is CIJT if and only if its number of parts is $d$ (weak Lefschetz case) or $d+k-1$.
\end{thm}
By another result of F.H.S. Macaulay the graded Artinian CI quotients $A=R/I$ of socle degree $j$ satisfy $I=\Ann F$ where $F$, the Macaulay dual generator of $A$, is a degree-$j$ element of $ \mathcal E={\sf k}[X,Y]$, the dual to $R$, and 
\begin{equation}\label{dualeq}
I=\{f\in R\mid  f\circ F=0\},
\end{equation} where $R$ acts on $\mathcal E$ by differentiation: $x^iy^j\circ X^uY^v=u_i\cdot v_jX^{u-i}Y^{v-j}$ if $u\ge i$ and $v\ge j$, where $u_i=u(u-1)\cdots (u+1-i)$; otherwise $x^iy^j\circ X^uY^v=0$, and this is extended bilinearly.\par
T.~Maeno and J.~Watanabe in 2009 introduced a method of using higher Hessians of the Macaulay dual generator $F$ to determine the strong Lefschetz property of a graded Artinian algebra \cite{MW}; this was further developed and used by T. Maeno and Y.~Numata \cite{MN} and by R.~Gondim and colleagues \cite{Go,GoZ,CGo}. In particular, R.~Gondim and G. Zappal\`{a}  developed mixed Hessians to study the weak Lefschetz property \cite{GoZ}, and B. Costa and R. Gondim showed that the ranks of the mixed Hessian matrices, evaluated at a point $p_\ell$ corresponding to the linear form $\ell\in A_1$, determine the Jordan type $P_{\ell,A}$ of a graded Artinian Gorenstein algebra $A$ \cite[Theorem 4.10]{CGo}. 
Given a graded Gorenstein algebra $A=R/I$ of Hilbert function $T$ and a linear form $\ell\in R_1$, the Hessian $h^i_\ell(F)$ is the determinant of the homomorphism $m_{\ell^{j-2i}}: A_i\to A_{j-i}$ given by the multiplication map by $\ell^{j-2i}$ (Definition~\ref{hessdef}). When $T$ satisfies Equation \eqref{CITeq} and $k\ge 2$ there are $d$ active Hessians; when $k\ge 1$ there are $d-1$. We show concerning Hessians (see Theorem \ref{2Hessthm})
\begin{thm}\label{thm1} Let $T$ satisfy \eqref{CITeq} for an integer $d\geq 2$ and assume that the characteristic $\cha {\sf k}=0$. Then there is a 1-1 correspondence between the CIJT partitions $P_\ell $ having diagonal lengths $T$, and the $2^d$ (when $k>1$), or $2^{d-1}$ (when $k=1$) subsets of the active Hessians for $T$ that vanish at $\ell$ in $R_1$.
\end{thm}
\begin{example}\label{1221ex} Let $T=(1,2,2,1)$. The active Hessians are $h^0$, the determinant of multiplication by $m_{\ell^3}: A_0\to A_3$, and $h^1$, from multiplication by $m_{\ell}: A_1\to A_2$.
Thus, there are four subsets of these Hessians. When no Hessian is zero, the Jordan type  $P_\ell$ is $(4,2)$, the conjugate of $T$, and $\ell$ is strong Lefschetz. For the CI algebra $R/(x^2,y^3)$, the multiplication $m_y$ has partition $P_y=(3,3)$, and only $h^0$ is zero; the multiplication $m_x$ has partition $P_x=(2,2,2)$ and both 
$h^0,h^1$ are zero, while $m_{x+y}$ has partition $(4,2)$. For the CI algebra $R/(xy,x^3+y^3)$ the multiplication $m_x$ (or $m_y$) has partition $(4,1,1)$, and only $h^1$ is zero. There are two more partitions of diagonal lengths $T$, namely $(3,1,1,1)$, which occurs for $m_x$ in the non-CI algebra $R/(xy,x^3,y^4)$ and $(2,2,1,1)$ which occurs for $m_x$ in $R/(x^2, xy^2,y^4).$ Note that $m_y$ in the latter, non-CI algebra has partition $(4,2)$: that is, certain CIJT partitions may occur also for a non-CI algebra.  See Example \ref{hook1ex} and Figure \ref{fig2a} for further detail when $T=(1,2,2,1)$.
\end{example}
We also determine the CIJT partitions in other, combinatorial ways, involving the attaching of branches to a basic triangle of a partition (Section \ref{JTCIsec}). Our proofs involve a careful combinatorial study of this process of attaching branches, in the spirit of \cite{IY}, and we adapt results from \cite{IY2}. The paper is self-contained.\vskip 0.2cm\noindent
{\bf Outline of results.}
In Lemmas \ref{criterionlem} and Lemma \ref{criterioncor} of Section \ref{JTCIsec}, we state and prove a necessary criterion for a partition having diagonal lengths $T$ satisfying Equation \eqref{CITeq} to occur as the Jordan type of a linear form of some Artinian \emph{complete intersection} algebra. This is very restrictive: there are only $2^d$ partitions satisfying the criterion when $k>1$ and $2^{d-1}$ when $k=1$. We adapt from \cite{IY} the method of adding branches to a basic triangle of $T$ to determine the partition $P$; we introduce a branch label (Definition~\ref{branchdef}) to describe this. We first exhibit in Lemma \ref{partlabel} (for multiplicity $k\ge 2$) and Lemma \ref{partlabel2} (for multiplicity $k=1$), \emph{all} possible partitions that have diagonal lengths $T$ satisfying Equation \eqref{CITeq} above: each of these may occur as a Jordan type for an algebra quotient $A$ of $R$ having Hilbert function $T$, that is not necessarily a complete intersection. This is less restrictive than CIJT and we show there are $2\cdot 3^{d-1}$ such partitions when $k>1$ and $3^{d-1}$ when $k=1$ (Corollary~\ref{PdiagTcor}): these numbers agree with the more general formulas for all $T$ of \cite[Theorem 3.30, (3.35)]{IY}.  In the main results of Section~\ref{JTCIsec}, Theorem~\ref{CIpartsthm} and 
Theorem~\ref{CIpartsthm2}, we determine via their branch labels, all the partitions that can occur as the Jordan type partitions of a linear form for some Artinian complete intersection algebra with the Hilbert function $T$, that is, all the CIJT partitions having diagonal lengths $T$.  Theorem~\ref{equalityCIthm} and Corollary~\ref{goodcountcor} confirm that all the partitions $P$ satisfying the criterion of Lemmas~\ref{criterionlem} and \ref{criterioncor} actually do occur as CIJT partitions. We show Theorem 1, the CIJT criterion using just the number of parts of $P$ and the diagonal lengths $T$ in Theorem \ref{simplegoodprop}.\par
 The first author with M.~Boij had previously determined the Jordan types for complete intersection algebras $A=R/\Ann F$, for a form $F\in \mathcal{E}_j$, when just a single Hessian could vanish (Theorem~\ref{genericthm} in Section \ref{genericFsec}). In Section \ref{arbitraryFsec}, Theorem~\ref{sumOfPartitions}, we determine the numerical condition on a  CIJT partition $P_\ell$ of diagonal lengths $T$ for a specific higher Hessian $h^i_\ell$ to be non-vanishing.  In
Theorem \ref{2Hessthm}, we specify and show the 1-1 correspondence between the sets of vanishing higher Hessians  of $F$ at $p_\ell$  with the CIJT partitions $P_\ell$ for $A=R/\Ann F$; this is a more precise version of Theorem \ref{thm1} above. 
We also provide the possible rank sequences for Hessian matrices at $p_\ell$ of a CIJT partition in Proposition~\ref{rkseq}.\par We report geometric consequences in Section \ref{latticesec}. First, our results imply that there is a lattice structure on the set of CIJT's having given diagonal lengths $T$. In Theorem~\ref{posetprop} we show that this structure coincides with the usual dominance order on these partitions (Definition \ref{dominancedef}). In Theorem \ref{mainHessthm} we show that the Zariski closure in the family $CI_T$ (all CI ideals of diagonal lengths $T$) of the cell $\mathbb V(E_P)$ of algebras having a given CIJT in a direction $\ell$ is the union of smaller or equal CIJT cells in the dominance order. The analogous frontier property is not shared by non-CIJT cells, by an example of J. Yam\'eogo (Remark \ref{frontierrem}). In Section \ref{patternsec} we show a result characterizing simply the CIJT partitions having $d$ parts, and their relation to those having $d+k-1$ parts (Theorem \ref{tablethm}). We end this section with tables of the CIJT partitions for $d\le 5$.\par
In Section \ref{CIJTsolsec} we make the connection with the hook codes of \cite{IY}. In Section \ref{hookcodesec}, we explain the hook code as well as illustrate it in some examples, and in Section \ref{CIJThookcodesec} Proposition~\ref{HCprop} we relate the branch labels and the hook codes for all partitions of diagonal lengths $T$, and in Corollary \ref{goodhookcor}. we apply this to CIJT partitions.  Then in Section \ref{vanishHess&hookcodesec} we prove the correspondence of vanishing Hessians of a complete intersection Jordan type  with its hook code (Proposition  \ref{HookandHess}).
\par
We include throughout diagrams and examples to illustrate the results.

\section{Jordan type for complete intersection Artinian algebras in two variables.}\label{JTCIsec}
This section contains our main results concerning the characterization of partitions of diagonal lengths $T$ satisfying Equation \eqref{CITeq}, and on characterizing the partitions that are CIJT -- that occur as the Jordan type $P_\ell$ of multiplication by a linear form $\ell$ in a graded Artinian complete intersection algebra $A={\sf k}[x,y]/I$.  We first in Section \ref{cell1sec} study ideals $I$ having a given initial monomial ideal determined by a partition $Q$, that is, the ideals $I\in \mathbb V(E_Q)$. By Lemma~\ref{diaglem}, the diagonal lengths of a CIJT partition is a Hilbert function satisfying Equation \eqref{CITeq}, and the algebras $A=R/I$ have Jordan type 
$P_{x,A}=Q$.\par
The Ferrers diagram of a partition $P$ of diagonal lengths $T=(1,2,\ldots, d-1,d, t_d,\ldots)$ with $t_d\le d$ has a filled {\it basic triangle} $\Delta(P)=\Delta_d$ consisting  of all monomials of degrees less than or equal to $d-1$. We regard $P$ as having {\emph{branches}}
glued to the basic triangle: some branches are horizontal, some may be vertical. We associate to each such partition a {\it{branch label}}, a sequence of integers corresponding to the lengths of these glued branches (Definition \ref{branchdef}). Beginning in Section \ref{whichJTCIsec}, we characterize labels associated with partitions having diagonal lengths $T$ satisfying \eqref{CITeq}. In Section \ref{attachbranch1sec} we characterize \emph{all} the partitions having diagonal lengths $T$. In addition to a complete numerical description of all such partitions, we also use the labels to count the number of such partitions having given diagonal lengths $T$. In Section \ref{goodsec} we characterize the CIJT partitions having diagonal lengths $T$.

\subsection{The cell $\mathbb V(E_Q)$ and Jordan type.}\label{cell1sec}
Recall that $R={\sf k}[x,y]$ is the polynomial ring over an infinite field of characteristic zero or characteristic $p>j$ where $j$ is the socle degree of the Artinian algebras we consider. 
By \cite[\S 58]{Mac2} the sequence $T$ occurs as the Hilbert function $T=T(d,k)$ of a graded complete intersection (CI) quotient $A=R/I$ of height $d$, where $d$ occurs $k$ times in $T$.  The ideal generator degrees are $(d,d+k-1)$ and we have (subscripts indicate degree; this is also Equation~\eqref{CITeq})
 \begin{equation}\label{CIT2eq}
T(d,k)=(1_0,2_1,\ldots,(d-1)_{d-2},d_{d-1},\ldots, d_{d+k-2},(d-1)_{d+k-1},\ldots , 2_{j-1},1_j). 
\end{equation}
Here the socle degree $j=2d+k-3$; the sequence $T(d,k)$ is symmetric about $j/2$ and is the Hilbert function of the monomial complete intersection  $R/(x^d,y^{d+k-1})$, of vector space dimension $\sum_i T(d,k)_i=d(d+k-1)=d(j+2-d)$.\par
More generally, a sequence $T$ occurs as the Hilbert function $T=H(A)$ of some graded Artinian quotient $A=R/I$ of $R={\sf k}[x,y]$ having order $d$ ($I\subset \mathfrak m^d, I\nsubseteq \mathfrak m^{d-1}$) and maximal socle degree $j$ ($I\nsupseteq \mathfrak m^j, I \supset \mathfrak m^{j+1}$) if and only if 
\begin{equation}\label{HFeq}
T=(1_0,2_1,\ldots,  d_{d-1},t_d,t_{d+1},\ldots, t_j,0) \text { where } d\ge t_d\ge \cdots \ge t_j>0.
\end{equation}
We will initially consider such general Hilbert functions $T$ for our definitions and Lemma \ref{diaglem} just below, which concerns the cell $\mathbb V(E_Q)$ parametrizing ideals with initial monomial ideal $E_Q$ for partitions $Q$ of diagonal lengths $T$.  Then beginning in Section~\ref{whichJTCIsec} we will restrict to the graded complete intersection sequences $T$ of \eqref{CIT2eq}. We will denote by $G_T$ the smooth projective variety parametrizing graded quotients $A=R/I$ of Hilbert function $T$ \cite{Got,I0,IY}.
\begin{definition}\label{basiccelldef}[The cell $\mathbb V(E_Q)$ of the family $G_T$.]
The \emph{Ferrers diagram} of the partition $Q=(q_1,q_2,\ldots,q_s), q_1\ge q_2\ge \ldots \ge q_s$ is an array of length $q_i$ in the $i$-th row from the top. We denote by $C_Q$ the filling of the Ferrers diagram by monomials, with $i$-th row $\{y^{i-1},y^{i-1}x,\ldots ,y^{i-1}x^{q_i-1}\}$ --  see Figure~\ref{hook1fig}). We denote by $(C_Q)_i$ the degree-$i$ subset of $C_Q$. We denote by $E_Q$ the monomials in $x,y$ not in $C_Q$, and by $(E_Q)$ the ideal they generate. The \emph{diagonal lengths} $T=T(Q)$ are the Hilbert function $T=H(R/E_Q)$, and are the lengths of the lower-left to upper-right diagonals of the Ferrers graph of $Q$:  that is, $t_i=\dim_{\sf k} (C_Q)_i$.\label{diagdef} The \emph{cell $\mathbb V(E_Q)$ determined by $Q$} is all ideals of $R$ having $(E_Q)$ as initial ideal in reverse degree-lex order, using $(y,x)$ as ordered basis for $R_1$; the cell $\mathbb V(E_{Q,\ell})$ is the analogous cell using $(\ell,x)$ as ordered basis.\par
A finite-length simple ${\sf k}[x]$ module $M$ must satisfy $M\cong k[x]/(x^k)$. When such a module occurs as a direct summand in a decomposition of an $A-$ module $M$ under multiplication by $m_a, a\in \mathfrak m_A$ we term the simple module a length-$k$ \emph{string} of $M$.\end{definition}
Evidently, for $\ell=x$ the ideal $(E_Q)=(E_{Q,x})$ has generating set,
\begin{equation}\label{EPeqn}
(E_Q)=(x^{q_1},yx^{q_2},\ldots, y^{i-1}x^{q_i},\ldots, y^{s-1}x^{q_s},y^s),  \qquad i\in [1, s].
\end{equation}
We have
\begin{align} I\in \mathbb V(E_Q)&\Leftrightarrow I=(f_1,\ldots , f_i,\ldots, f_s,f_{s+1}) \text { with }\notag\\
f_i=\,&x^{q_i}g_i, \text { where } g_i=y^{i-1}+h_i, h_i\in (x)\cap R_{i-1}\notag\\
\equiv\,&y^{i-1}x^{q_i}\mod (x^{q_i+1})\cap R_i.\label{mueqn}
\end{align}
We show a key preparatory result.\footnote{Although \cite[Proposition 3.6]{IY} does not use the language of Jordan type, the Lemma \ref{diaglem} here may be regarded as a consequence of the discussion there.  Theorem 3.12 of \cite{IY} determines the dimension of $\mathbb V(E_P)$ in terms of the hook code (our Theorem \ref{celldimthminapp}). The Equation \eqref{mueqn} follows from the standard basis results of either \cite{Bri} or \cite{I0}: see the historical note following Theorem 3.12 of \cite{IY}.}
\begin{lemma}\cite{IY}\label{diaglem} A. Let $A={\sf k}[x,y]/I$  be a graded Artinian quotient of $R={\sf k}[x,y]$, and let $\ell\in A_1$ be a linear form.  The Jordan type partition $P_{\ell,A}$ has diagonal lengths the Hilbert function $T=H(A)$ of $A$, which satisfies \eqref{HFeq}.\par
B. Let $\ell=x$, and $I\in \mathbb V(E_Q)$, define $g_i$ as in Equation \eqref{mueqn} and denote by $\overline{g}_i$ the class of $g_i$ in $A=R/I$. Then we have the following decomposition of $A$ as a direct sum of simple ${\sf k}[x]$-modules (strings):
\begin{equation}\label{Alisteq}
A=\oplus \{\langle \overline{g}_i, x\overline{g}_i \ldots ,x^{q_i-1}\overline{g}_i\rangle,  i\in [1,s]\}.
\end{equation}
The Jordan type $P_{x,A}=Q$.
\end{lemma}\par
\begin{proof}  Recall the rev-lex ordering $y^i>y^{i-1}x>\cdots >x^i$ on degree-$i$ monomials of $R$. For $I\in \mathbb V(E_Q)$ the degree-$i$ component $I_i$ has initial monomials $E_Q(i)$, highest in the order; they  span a vector space complementary in $R_i$ to the span of $C(Q)_i$:  $R_i=\langle (C_Q)_i\rangle\oplus (E_Q)_i$. Thus, the total number of elements in the putative basis for $A$ given in Equation \eqref{Alisteq} is
the dimension $n=\dim_{\sf k}A$, as the number of elements of degree $i$ is just $H(A)_i$. Note that it follows from last statement $f_i\equiv y^{i-1}x^{q_i}\mod (x^{q_i+1})\cap R_i$ of Equation~\eqref{Alisteq} that 
\begin{equation}\label{stdidealeq}
I={\sf k}[x]\langle f_1,\ldots, f_s\rangle+(y^s).
\end{equation}
Suppose by way of contradiction that there is a relation among these elements
\begin{equation}
\sum \alpha_{i,k}x^k\overline{g}_i=0 \text { in } A, \text { with }\alpha_{i,k}\in {\sf k}.
\end{equation}
Then 
\begin{equation*}\sum \alpha_{i,k}x^k{g}_i\in (f_1,\ldots, f_s) \text { where $i\in [1,s]$ and $k$ in $x^kg_i$ satisfies }0\le k\le q_i-1,
\end{equation*}
and by Equation \eqref{stdidealeq} we have
\begin{equation}
\sum \alpha_{i,k}x^k{g}_i\in {\sf k}[x]\langle f_1,\ldots, f_s\rangle+(y^s).
\end{equation}
This implies that the sum on the left is in the ${\sf k}[x]$ submodule of $R$ generated by $f_1,\ldots, f_s$. Collecting by $y$-degree, we have for each $i$, $\sum_{i,k}\alpha_{i,k}x^kg_i\in {\sf k}[x]f_i$: since $ f_i=x^{q_i}g_i$, and each $k$ is less than $q_i$, each such summand is zero,  and each $\alpha_{i,k}=0$. We have shown that  Equation \eqref{Alisteq} gives a basis of $A$. Noting that $x^{q_i}g_i=f_i\in I$, so $x^{q_i}\overline{g}_{i-1}=0$ in $A$,  we may conclude from \eqref{Alisteq} that $A$ is the direct sum of strings of lengths $q_1,q_2,\ldots, q_s$; hence, the Jordan partition determined by multiplication by $x$ is indeed $P_{x,A}=Q$.  This completes the proof of the Lemma.
\end{proof}\par
\begin{remark}\label{QPrem} Note that while for algebras $A$ in the cell $\mathbb V(E_Q)$, the Jordan partition $P_{x,A}$ is $Q$, we will also consider the Jordan partition $P_\ell$ in other directions $\ell$: these may be different from $Q$.  For example, when $Q=(3,1)$, the monomial ideal $E_Q=(y^2,xy,x^3)$: for $A=R/E_Q$ the Jordan type $P_x=Q$, but $P_{y,A}=(2,1,1)$.  This occurs more generally: for $A=R/(x\ell,y\ell,x^3)$ with $\ell=y+ax$ we have $A\in \mathbb V(E_Q)$ but $P_{\ell,A}=(2,1,1)$. This is why we have used $Q$ in place of $P$ here in defining the cell $\mathbb V(E_Q)$. Of course, for an open dense set of $A=R/I\in \mathbb V(E_Q)$ the homogeneous component does not have a common factor: there is no $\ell\in R_1$ such that $I_2=(\ell)\cap R_2$: for such $A$ the Jordan type $P_{\ell',A}=Q$ for all $\ell'\in R_1$.
\end{remark}

 \subsection{Complete intersection Jordan type criterion.}\label{whichJTCIsec}
 In this section, we provide necessary and sufficient numerical conditions for a partition $P$ of diagonal lengths $T$ satisfying \eqref{CIT2eq} (same as \eqref{CITeq}) to have CIJT. We will henceforth write the partition $P$ in power form
 \begin{equation}\label{partpoweqn}  P=(p_1^{n_1},p_2^{n_2}, \ldots p_t^{n_t}) \text { with }p_1>p_2>\cdots >p_t,
 \end{equation}
 as it will be useful in determining the minimal generators of an ideal $I\in \mathbb V(E_P)$.

\begin{lemma}\label{criterionlem}  
Let $T$ be a sequence satisfying \eqref{CITeq} and assume that $P$ as in \eqref{partpoweqn} is a partition of diagonal lengths $T$. If $P$ has CIJT then for each $i\in [ 2,t]$, we have
\begin{equation}\label{good2eq}
p_{i-1}\geq  n_{i-1}+n_i+p_i.
\end{equation}
\end{lemma}

\begin{proof}

Let $E=(E_P)$ be the monomial ideal corresponding to a partition $P$. It is evident from Lemma \ref{diaglem} and Equation \eqref{partpoweqn} that we may write the minimal generators of the monomial ideal $E$ (a basis $B_E$ for $E/\mathfrak mE$)
as  
\begin{equation}\label{generatorsEeq}
B_E=\langle  x^{p_1}, x^{p_2}y^{a_1}, \ldots, x^{p_t}y^{a_{t-1}}, y^{a_t}\rangle \text { where for each } i\in[1, t], a_i=\displaystyle{\sum_{\ell=1}^{i}n_\ell}.
\end{equation}
All but the first and last generators are the monomials corresponding to the inside corners of the Ferrers diagram $F_P$ corresponding to $P$; we include the highest $y$ power and highest $x$-power, the outside corners.

By assumption $P$ occurs as the Jordan type of multiplication by a linear element in some graded Artinian complete intersection quotient  $A=R/I$. Thus for $i\in [2, t]$, the unique elements  $f_{i+1},f_{i}$ of $I$ having initial monomials $\mu_{i+1}=x^{p_{i+1}}y^{a_{i}}$ and $\mu_{i}=x^{p_{i}}y^{a_{i-1}}$, respectively,  must generate $f_{i-1}$ (with initial monomial $\mu_{i-1}=x^{p_{i-1}}y^{a_{i-2}}$). They can do so only if the relation
 
  \begin{equation}\label{relationeq1}
  x^{(p_i-p_{i+1})}\mu_{i+1}-y^{(a_i-a_{i-1})}\mu_i
  \end{equation}
  between the initial forms $\mu_{i+1}$ and $\mu_{i}$, when applied to the two generators kicks out $f_{i-1}$ (after further reduction by multiples of $\mu_{i+1}$). This is only possible if the degree of $f_{i-1}$ is at least that of the terms in \eqref{relationeq1}. This implies that  $p_{i-1}+a_{i-2}\geq p_i+a_i$. By definition of $a_i$'s, this implies the desired inequality $p_{i-1}\geq n_{i-1}+n_i+p_i.$
\end{proof}

We will later show that there must be equality in \eqref{good2eq} for $P$ a CIJT partition  (Corollary~\ref{equalityCIthm}, Equation \eqref{equalityCIeq}). Thus, using this later result, the hypothesis of the following Lemma can be weakened to Equation \eqref{good2eq}.

\par In the proof of the following Lemma, we construct a particular complete intersection ideal such that multiplication $m_x$ by the element $x$ on $A=R/I$ has a given partition $P$ satisfying Equation \eqref{good2eq}. Such an ideal $I$ is in the cell $\mathbb V(E_P)$ (Definition \ref{basiccelldef}). The dimension of the family of all such CI ideals is the dimension of the cell $\mathbb V(E_P)$, which we give in Proposition \ref{HCprop} and Corollary \ref{goodhookcor} below. 
\begin{lemma}[CIJT Criterion]\label{criterioncor}
Let $T$ satisfy Equation \eqref{CITeq} and let $P=(p_1^{n_1}, \ldots p_t^{n_t})$ with $p_1>\cdots >p_t$ be a partition of diagonal lengths $T$. If for each $i\in[2,t]$, the following equality holds, then $P$ can occur as the Jordan type of a linear form for some graded complete intersection algebra $A=R/I$, of Hilbert function $H(A)=T$.
\begin{equation}\label{good3eq}
p_{i-1}= n_{i-1}+n_i+p_{i}
\end{equation}
Furthermore $P$ occurs as the Jordan type of multiplication by $\ell$ on $A=R/I$ if and only if $E_P$ is the monomial initial ideal of $I$ in the $\ell $ direction (that is with $(y,\ell)$ as distinguished coordinates for $R$).
\end{lemma}
\medskip

\begin{proof}

We inductively define $t+1$ polynomials $f_1, \ldots, f_{t+1}$ in $R$ such that 
\begin{itemize}
\item[(1)] For $i\in [1, t+1]$, $f_i$ is a homogeneous polynomial with leading term $x^{p_i}y^{a_{i-1}}$, where $a_i=\displaystyle{\sum_{j=1}^in_i}.$
\item[(2)] For $i\in[2,t]$, $f_{i-1}=\left[x^{(p_{i}-p_{i+1})}\, f_{i+1}\right]-\left[f_{i}\, y^{n_i}\right].$
\end{itemize}
\medskip

Let 
$$
f_1=x^{p_1} \mbox{ and }
f_2=x^{p_2} y^{a_1}+\displaystyle{\sum_{\ell=1}^{a_1}\lambda_{2,\ell} \, x^{(p_2+\ell)}\, y^{(a_1-\ell)}},
$$ where $\lambda_{2,1}, \ldots, \lambda_{2,a_1}\in {\sf k}$ are arbitrary parameters. Then $f_1$ and $f_2$ satisfy condition (1) above. Now assume that $i\in[2,t]$ and that $f_1, \ldots, f_i$ are defined in a way that they satisfy conditions (1) and (2). Suppose that 
$$\begin{array}{ll}
f_{i-1}&=x^{p_{i-1}}y^{a_{i-2}}+\displaystyle{\sum_{\ell=1}^{a_{i-2}}\lambda_{i-1,\ell} \, x^{(p_{i-1}+\ell)}\, y^{(a_{i-2}-\ell)}} \mbox{ and}\\ \\
f_{i}&=x^{p_{i}}y^{a_{i-1}}+\displaystyle{\sum_{\ell=1}^{a_{i-1}}\lambda_{i,\ell} \, x^{(p_{i}+\ell)}\, y^{(a_{i-1}-\ell)}}.
\end{array}$$

In order for $f_{i+1}$ to satisfy (1) and (2) we find $\lambda_{i+1,1}, \ldots, \lambda_{i+1,a_i}\in {\sf k}$ such that 

$$\begin{array}{ll}
f_{i+1}&=x^{p_{i+1}}y^{a_{i}}+\displaystyle{\sum_{\ell=1}^{a_{i}}\lambda_{i+1,\ell} \, x^{(p_{i+1}+\ell)}\, y^{(a_{i}-\ell)}} \mbox{ and}\\ \\
f_{i-1}&=\left[x^{(p_{i}-p_{i+1})}\, f_{i+1}\right]-\left[f_{i}\, y^{n_i}\right].
\end{array}$$

We have 
\begin{small}
$$\begin{array}{ll}
\left[x^{(p_{i}-p_{i+1})}\, f_{i+1}\right]-\left[f_{i} \, y^{n_i}\right]&=\left[ x^{p_i} y^{a_i}+\displaystyle{\sum_{\ell=1}^{a_i}\lambda_{i+1,\ell} \, x^{(p_i+\ell)}\, y^{(a_i-\ell)}}\right]-\left[x^{p_i} y^{a_i}+\displaystyle{\sum_{\ell=1}^{a_{i-1}}\lambda_{i,\ell} \, x^{(p_i+\ell)}\, y^{(a_i-\ell)}} \right]\\ \\ 
&=\displaystyle{\sum_{\ell=1}^{a_{i-1}}(\lambda_{i+1,\ell}-\lambda_{i,\ell}) \, x^{(p_i+\ell)}\, y^{(a_i-\ell)}}+\displaystyle{\sum_{\ell=a_{i-1}+1}^{a_i}\lambda_{i+1,\ell} \, x^{(p_i+\ell)}\, y^{(a_i-\ell)}}.
\end{array}$$
\end{small}
By construction, the degree of the polynomial $\left(\left[x^{(p_{i}-p_{i+1})}\, f_{i+1}\right]-\left[f_{i}\, y^{n_i}\right]\right)$ is $(p_i+a_i)$. On the other hand, by assumption, we also have $p_{i-1}=p_i+n_i+n_{i-1}$. This implies that $p_{i-1}+a_{i-1}=p_i+a_i$ and therefore the degree of the polynomial $\left(\left[x^{(p_{i}-p_{i+1})}\, f_{i+1}\right]-\left[f_{i}\, y^{n_i}\right]\right)$ is the same as the degree of $f_{i-1}$. Finally, setting $\left(\left[x^{(p_{i}-p_{i+1})}\, f_{i+1}\right]-\left[f_{i}\, y^{n_i}\right]\right)=f_{i-1}$, we uniquely determine the coefficients $\lambda_{i+1,1}, \ldots, \lambda_{i+1,a_i}$ of $f_{i+1}$ in terms of $\lambda_{i-1,1}, \ldots, \lambda_{i-1,a_{i-2}}$ and  $\lambda_{i,1}, \ldots, \lambda_{i,a_{i-1}}$. In fact, if we let $\Lambda_i=(\lambda_{i,1}, \ldots, \lambda_{i,a_{i-1}})$, then for $i\in[2, t]$, we have $$\Lambda_{i+1}=(\Lambda_i, \underbrace{0, \ldots, 0}_{n_i})+(\underbrace{0, \ldots, 0}_{n_{i-1}+n_i-1}, 1, \Lambda_{i-1}).$$

Therefore, for $i=1, \ldots, t+1$, we have constructed polynomials $$f_i=x^{p_{i}}y^{a_{i-1}}+\displaystyle{\sum_{\ell=1}^{a_{i-1}}\lambda_{i,\ell} \, x^{(p_{i}+\ell)}\, y^{(a_{i-1}-\ell)}}$$ satisfying conditions (1) and (2) above.

Now consider the ideal $I$ of $R$ generated by polynomials $f_1, \ldots, f_{t+1}$ constructed above. Condition (2) implies that $I$ is in fact generated by $f_t$ and $f_{t+1}$. Thus $A=R/I$ is complete intersection. Furthermore, by construction of $I$, multiplication by $x$ in $A$ has Jordan type $P_x=P$, and  the Hilbert function of $H(A)$ is the diagonal lengths of $P$ (see Definition \ref{basiccelldef}, also \cite[Definition 3.3  and Lemma 3.4]{IY}, concerning the cell $\mathbb V(E)$). 
\end{proof}

\begin{example}
Consider the partition $P=(6, 2,2,2)$ satisfying Equation \eqref{good2eq} from Lemma \ref{criterioncor} above. The diagonal lengths of $P$ are $T=(1,2,3,3,2,1)$ which is of the form in Equation \eqref{CITeq}. Following the proof of the Lemma, we construct a complete intersection Artinian algebra $A={\sf k}[x,y]/I$ in which multiplication by $x$ has Jordan type $P$. 

We have $p_1=6$, $n_1=1$, $p_2=2$, $n_2=3$. Thus $a_1=1$ and $a_2=1+3=4.$ We set 
$$\begin{array}{ll}f_1&=x^{p_1}=x^6, \\ \\ f_2&=x^{p_2}y^{a_1}+\alpha x^{(p_2+1)}\, y^{(a_1-1)}=x^2y+\alpha x^3 \mbox{ and}\\ \\f_3&=y^{a_2}+\displaystyle{\sum_{\ell=1}^{a_2} \beta_{\ell} \, x^{\ell}\, y^{a_2-\ell}}=y^4+\beta_1 xy^3+\beta_2 x^2y^2+\beta_3x^3y+\beta_4x^4. \end{array}$$ Here $\alpha\in {\sf k}$ is an arbitrary parameter and  $(\beta_1, \beta_2, \beta_3, \beta_4)=(\alpha, 0, 0, 0)+(0, 0, 0, 1).$
Thus $$f_3=y^4+\alpha xy^3+x^4.$$

Then for each $\alpha\in {\sf k}$, $$A=\frac{{\sf k}[x,y]}{\langle f_2,f_3\rangle}=\frac{{\sf k}[x,y]}{\langle x^2(y+\alpha x),y^4+x(\alpha y^3+x^3)\rangle}$$ is the desired complete intersection Artinian algebra. 

Looking at one such algebra with $\alpha=0$, namely $A=\frac{{\sf k}[x,y]}{\langle x^2y,y^4+x^4\rangle},$ we can easily see that multiplication by $x$ in the basis $\{1, x, \ldots, x^5, y, xy, y^2, xy^2, y^3, xy^3\}$ for $A$ is in Jordan form with Jordan type $P=(6, 2, 2,2).$
 
\end{example}

\subsection{Partitions of diagonal lengths $T$, a combinatorial characterization.}\label{attachbranch1sec}
In this section, we provide a complete combinatorial characterization of partitions of diagonal lengths $T$, where $T$ is a Hilbert function satisfying Equation \eqref{CIT2eq} (and \eqref{CITeq}).
\bigskip

\noindent{\bf Labeling.}  
Let $T=(1_0,2_1,\ldots,(d-1)_{d-2},d_{d-1},\ldots, d_{d+k-2},(d-1)_{d+k-1},\ldots , 2_{j-1},1_j)$.  Let $P$ be a partition having diagonal lengths $T$. In \cite[\S 3.1]{IY2}, A. Iarrobino and J.~Yam\'{e}ogo show that the Ferrers diagram of $P$ is obtained from $\Delta_d$ by attaching $d+1$ ``branches" of lengths $0, k-1, k, \ldots, d+k-2$.\footnote{The result in \cite{IY2} is rather more general, for $T$ satisfying \eqref{HFeq}; the special case for $T$ satisfying Equation \eqref{CITeq} can be readily shown.} We note that when $k=1$, this sequence contains two 0's. Attaching a branch of length zero at a position in $\Delta_d$ represents leaving a gap at the corresponding position of $\Delta_d$. If $k>1$ then in the Ferrers diagram of a partition having diagonal lengths $T$ there is only one gap, a position with no new branch attachment, while for $k=1$ there are two gaps. 
\medskip

\noindent{\bf Convention.} We count the columns of a Ferrers diagram from left to right and its rows from top to bottom. Its boxes correspond to the monomials in $x,y$ (see Example~\ref{431ex} and Figure \ref{hook1fig}).
\medskip

We next define the {\it branch label} associated to a partition $P$ of diagonal lengths $T$ satisfying Equation \eqref{CITeq}. Although related to the concepts of \cite{IY2} this label, an ordered sequence of non-negative integers giving the branch lengths, is new here.

\begin{definition}[Branch label]\label{branchdef}
Let $T=(1, 2, \ldots,(d-1), d^k, (d-1),\ldots, 2,1)$ as in Equation \eqref{CIT2eq}. We label a partition $P$ of diagonal lengths $T$ by a $(d+1)$-tuple $\mathfrak{b}$, as follows. Recall that the Ferrers diagram of $P$ is formed by attaching branches to the $d+1$ attachment places of the basic triangle $\Delta_d$. If $k\geq 2$ then there is only one ``gap" in the Ferrers diagram of $P$, while there are two gaps when $k=1$. The branches below the highest gap are attached to $\Delta_d$ vertically, while the branches above the gap are attached horizontally. The branch label $\mathfrak{b}$ keeps track of the lengths of these attachments by listing the lengths of the vertical attachments listed from left to right, followed by a 0 indicating the gap, and then the lengths of the horizontal attachments listed from top to bottom. When $k\geq 2$, each branch in in fact ``thickened" by an extra $k-2$ boxes which we do not count in measuring each branch in the label. See Figure~\ref{simplebranchlabel} for an illustration of the correspondence between a partition and its branch label.

Let $s=\max\{0, k-2\}.$ Then the branch label $\mathfrak{b}$ for a partition $P$ of diagonal lengths $T$ is obtained by reordering $(0, k-1-s, k-s, \ldots, k+d-2-s)$ described below. If $k=1$ then the branch label $\mathfrak{b}$ is a reordering of the sequence $0,0, 1, \dots, d-1$, while for $k\geq 2$, the branch label $\mathfrak{b}$ is a reordering of $0, 1, \dots, d$. Let $\mathfrak{b}=(\mathfrak{b}(0), \mathfrak{b}(1), \dots, \mathfrak{b}(d))$. We first locate the entry of $\mathfrak{b}$ that corresponds to the gap where the switch between vertical and horizontal attachments occurs, as follows. 
$$e=\max\{i\, |\, \mbox{there is a gap at the position corresponding to } x^iy^{d-i}\}.$$ Then 
\begin{equation}\label{Beqn}
\begin{array}{l}

\mathfrak{b}(i)=\left\{
\begin{array}{lcl} 0 && i=e\\ 
\mbox{[Length of the }(i+1)\mbox{-st column of } P] -(d-i)-s&&i<e\\
\mbox{[Length of the }(i-e)\mbox{-th row of }  P] -(d-i+e+1)-s&&i>e.
 \end{array}\right.
 \end{array}
 \end{equation}
 
% We note that if $k=1$, then $\b$ is a reordering of $(0,0,1, \ldots, d-1)$ while for $k>1$, $\b$ is a permutation of  $(0,1, \ldots, d).$
  \medskip
  
Conversely, assume that $\mathfrak{b}=(b_0, \ldots, b_d)$ is a reordering of  $(0,0, 1, \ldots, d-1)$ when $k=1$, and a permutation of $\{0, 1, 2, \ldots, d\}$ for $k>1$. The Ferrers diagram assigned to $\mathfrak{b}$ is obtained from $\Delta_d$ through the following attachment process.  
\begin{align}
&\text{Let $e=\max\{i\, |\, b_i=0\}$. For $0\leq i < e$ a vertical branch of length $b_i+s$ is attached}\notag\\
&\text{ at the end of the $(i+1)$-st column of $\Delta_d$;}\notag\\
&\text{while for $e<i \leq d$ a horizontal branch of length $b_i+s$ is attached}\notag\\
&\text{ at the end of the $(i-e)$-th row of $\Delta_d$.}\label{PfromBeqn} 
\end{align}

%\begin{definition}[Branch label]\label{branchdef}
%Let $T=(1, 2, \ldots,(d-1), d^k, (d-1),\ldots, 2,1)$, and set $$s=\max\{0, k-2\}.$$ We label a partition $P$ of diagonal lengths $T$ by a $(d+1)$-tuple $\mathfrak{b}$, obtained by reordering $(0, k-1-s, k-s, \ldots, d+k-2-s)$ as follows. 
%
%Consider the Ferrers diagram of $P$ and let $$e=\max\{i\mbox{ such that there is a gap at the position corresponding to } x^iy^{d-i}\}.$$ We define 
%\begin{equation}\label{Beqn}
%\begin{array}{l}
%
%\mathfrak{b}(i)=\left\{
%\begin{array}{lcl} 0 && i=e\\ 
%\mbox{[Length of the }(i+1)\mbox{-st column of } P] -(d-i)-s&&i<e\\
%\mbox{[Length of the }(i-e)\mbox{-th row of }  P] -(d-i+e+1)-s&&i>e.
% \end{array}\right.
% \end{array}
% \end{equation}
% 
% We note that if $k=1$, then $\b$ is a reordering of $(0,0,1, \ldots, d-1)$ while for $k>1$, $\b$ is a permutation of  $(0,1, \ldots, d).$
%  \medskip
%  
%Conversely, assume that $\mathfrak{b}=(b_0, \ldots, b_d)$ is a reordering of  $(0,0, 1, \ldots, d-1)$ when $k=1$, and a permutation of $\{0, 1, 2, \ldots, d\}$ for $k>1$. The Ferrers diagram assigned to $\mathfrak{b}$ is obtained from $\Delta_d$ through the following attachment process.  
%\begin{align}
%&\text{Let $e=\max\{i\, |\, b_i=0\}$. For $0\leq i < e$ a vertical branch of length $b_i+s$ is attached}\notag\\
%&\text{ at the end of the $(i+1)$-st column of $\Delta_d$;}\notag\\
%&\text{while for $e<i \leq d$ a horizontal branch of length $b_i+s$ is attached}\notag\\
%&\text{ at the end of the $(i-e)$-th row of $\Delta_d$.}\label{PfromBeqn} 
%\end{align}
\end{definition}
\begin{figure}
$$\begin{array}{|c|c|c|c|}
\hline
\mbox{Partition:}&
\scalebox{0.7}{
\begin{ytableau}
\,&&&&*(black! 50)&*(black! 50)\cr
\,&&\cr
\,&&*(black! 50)\cr
\,\cr
*(black! 50)\cr
*(black! 50)\cr
*(black! 50)\cr
\end{ytableau}}
&
\scalebox{0.7}{\begin{ytableau}
\,&&&&*(black! 50)&*(black! 50)&*(black! 50)\cr
\,&&&*(black! 50)&*(black! 50)&*(black! 50)&*(black! 50)\cr
\,&\cr
\,&*(black! 50)\cr
*(black! 50)&*(black! 50)\cr
\end{ytableau}}
&
\scalebox{0.7}{
\begin{ytableau}
\,&&&&*(black! 20)&*(black! 20)&*(black! 50)&*(black! 50)&*(black! 50)\cr
\,&&&*(black! 20)&*(black! 20)&*(black! 50)&*(black! 50)&*(black! 50)&*(black! 50)\cr
\,&\cr
\,&*(black! 20)\cr
*(black! 20)&*(black! 20)\cr
*(black! 20)&*(black! 50)\cr
*(black! 50)&*(black! 50)\cr
\end{ytableau}}
\\
\hline
\mbox{Diagonal Lengths:}&(1,2,3,4,3,2,1)&(1,2,3,4^2,3,2,1)&(1,2,3,4^3,3,2,1)
\\
\hline
\mbox{Attachment Lengths:}&
\begin{array}{lllll}
&&&&2\\
&&&\boxed{\boxed{0}}\\
&&1&\\
&0&&\\
3&&
\end{array}
&
\begin{array}{llll}
&&&3\\
&&&4\\
&&\boxed{\boxed{0}}&\\
1&2&
\end{array}
&
\begin{array}{cccc}
&&&\textcolor{gray}{2\, +}\,  \, 3\\
&&&\textcolor{gray}{2\, +}\, \, 4\\
&&\boxed{\boxed{0}}&\\
\textcolor{gray}{2}&\textcolor{gray}{2}\\
\textcolor{gray}{+}&\textcolor{gray}{+}\\
1&2&
\end{array}
\\
\hline
\mbox{Branch Label:}&
(3,0,1, 0,2)&(1,2,0,3,4)&(1,2,0,3,4)
\\
\hline
\end{array}$$
\caption{An illustration of the correspondence between a partition of given diagonal lengths satisfying Equation \ref{CIT2eq} and its branch label. The branches below the highest gap which is indicated by the boxed 0's, are attached vertically and are listed in the label from left to right, while the branches above the gap are attached horizontally and are listed in the label from top to bottom. }\label{simplebranchlabel}
\end{figure}

\begin{example}\label{labeldefnexample}

Consider the partition $P_1=(6,3^2,1^4)$ illustrated in Figure~\ref{simplebranchlabel} on the left. Then the  diagonal lengths of $P_1$ are given by $T_1=(1,2,3,4,3,2,1)$. Using the same notations as in definition \ref{branchdef}, for $P_1$ we have $d=4$ $k=1$. Thus the branch label of $P_1$ will be a $5$-tuple, say $(b_0, b_1,b_2, b_3, b_4)$, with entries $0, 0, 1, 2, 3$. As we see in the figure, the Ferrers diagram of $P_1$ consists of the basic triangle $\Delta_4$, shown in white, and three branch attachments, shown in dark gray. There are two ``gaps" in the Ferrers diagram, one corresponds to the monomial $xy^3$ and the other one corresponds to $x^3y$. Thus $e=3$, representing the gap at $x^3y$. As the diagram illustrates, in the Ferrers diagram of $P_1$, the branches to the left of this gap are attached to $\Delta_4$ vertically while the ones above the gap are attached horizontally. Since $e=0$, by definition we have $b_3=0$. Moreover, the entries of the branch label to the left of $b_3$, list the length of the branch (vertical) attachments from left to right, giving us the sequence $3,0,1$, and the entries of the branch label to the right of $b_3$ will list the lengths of attachments from top to bottom, for $P_1$ there is only one horizontal attachment of length 2. Thus the branch label of $P_1$ is $(3,0,1,0,2)$.

Next we consider $P_2=(7^2,2^3)$ illustrated in Figure~\ref{simplebranchlabel} in the middle. Here the diagonal lengths are given by $T_2=(1,2,3,4^2,3,2,1)$. We have $d=4$ and $k=2$. We also see that in the Ferrers diagram of $P_2$ is obtained from $\Delta_4$ by attaching four branches of lengths $1,2,3,4$. We also note that the (only) gap corresponds to $x^2y^2$. Thus in the branch label of $P_2$, which will be a $5$-tuple $(b_0, \dots, b_4)$, we have $b_2=0$. Listing the lengths of the vertical branches from left to right, we see that the vertical portion of the label consists of $1,2$, while the horizontal part, listed from top to bottom will be $3,4$. Thus the branch label of $P_2$ is $\mathfrak{b}=(1,2,0,3,4)$. 

Finally, consider $P_3=(9^2, 2^5)$ illustrated in Figure~\ref{simplebranchlabel} on the right. The diagonal lengths of $P_3$ are given by $T_3=(1,2,3,4^4,3,2,1)$. In this case, the actual lengths of the branch attachments to $\Delta_4$ are $3,4,5,6$. However, using the notation introduced in the definition \ref{branchdef}, here we have $s=\max\{0,4-2\}=2$ and therefore we label the partition $P_3$ with $0, 3-2,4-2,5-2,6-2$. Following the same rules as before, the branch label for $P_3$ is $(1,2,0,3,4)$. 
\end{example}
\begin{example}
Conversely, given a branch label and a sequence satisfying Equation \ref{CITeq} as the diagonal lengths, we can uniquely determine the corresponding partition. For example, the unique partition $Q$ with diagonal lengths $(1, 2,3,4^7,3,2,1)$ and branch label $(2,0, 3, 4, 1)$ is obtained by attaching branches to the basic triangle $\Delta_4$ as follows. The entry 0 corresponds to a gap at the bottom of the second column of $\Delta_4$. Each non-zero entry in the branch label corresponds to a branch attachment to $\Delta_4$. Since in the diagonal lengths sequence we have $s=\max\{0, 7-2\}=5$, the actual length of each attached branch is the corresponding entry in the label plus 5. Thus the partition $Q$ is obtained from $\Delta_4$ by attaching a vertical branch of length 7 to the first column and attaching horizontal branches of  lengths 8, 9 and 6 to rows one, two and three, respectively. Thus $Q=(4+8, 3+9, 2+6, 1, 1^{7})=(12^2, 8, 1^8)$. See Figure~\ref{labeltopartsimplefig}.

\begin{figure}
    $$\begin{array}{|c|c|c|c|}
    \hline
        \mbox{Diagonal Lengths}
        &\mbox{Branch Label}
        &
        \mbox{Attachment Lengths}
        &
        \mbox{Partition}\\
        \hline
        (1,2,3,4^7,3,2,1)  &
         (2, 0, 3,4,1)&  \begin{array}{lllll}
&&&&\textcolor{gray}{5\,+\, }\, 3\\
&&&\textcolor{gray}{5\,+\, }\, 4\\
&&\textcolor{gray}{5\,+\, }\, 1\\
&\boxed{\boxed{0}}\\
\textcolor{gray}{5}\\
\textcolor{gray}{+}\\
2
\end{array}
&
\scalebox{0.5}{
\begin{ytableau}
\,&&&&*(black! 20)&*(black! 20)&*(black! 20)&*(black! 20)&*(black! 20)&*(black! 50)&*(black! 50)&*(black! 50)\cr
\,&&&*(black! 20)&*(black! 20)&*(black! 20)&*(black! 20)&*(black! 20)&*(black! 50)&*(black! 50)&*(black! 50)&*(black! 50)\cr
\,&&*(black! 20)&*(black! 20)&*(black! 20)&*(black! 20)&*(black! 20)&*(black! 50)\cr
\cr
*(black! 20)\cr
*(black! 20)\cr
*(black! 20)\cr
*(black! 20)\cr
*(black! 20)\cr
*(black! 50)\cr
*(black! 50)\cr
\end{ytableau}}\\
\hline

    \end{array}
    $$
    \caption{Constructing the partition with diagonal lengths $T=(1,2,3,4^7,3,2,1)$ and branch label $(2,0,3,4,1)$. See Example~\ref{labeldefnexample}.}
    \label{labeltopartsimplefig}
\end{figure}
\begin{comment}
Consider the partition $P=(19^2, 11^3, 5^3, 3^8)$ whose Ferrers diagram is illustrated on the left in Figure \ref{labeldefnillust}. It has diagonal lengths $T_P=(1, 2, \ldots, 9, 10^2, 9, \ldots, 2, 1)$ and by Definition \ref{branchdef} its branch label is $\mathfrak{b}_P=\big(6,7,8,1,2,0,9,10,3,4,5\big)$. The entries of $\mathfrak{b}_Q$ count the lengths of branches attached to the basic triangle $\Delta_{10}$. The 0 entry in the label corresponds to the gap in the position corresponding to the monomial $x^5y^5$ (indicated by $E$), while the entries of label appearing before the 0 are the lengths of vertical attachments to columns below the gap, and the entries of the label after the 0 entry are the lengths of the horizontal branch attachments in the rows above the gap.
% excluded by LK - next diagram/example
\end{comment}
\end{example}
\begin{comment}
\begin{figure}
\hspace{-.4in}
\includegraphics[scale=.35]{labeldefnillust.pdf}
 \hspace{-.4in} \includegraphics[scale=.35]{labeldefnillust2.pdf}
\caption{Ferrers diagrams of partitions $P=(19^2, 11^3, 5^3, 3^8)$ on the left, and $Q=(19, 12^3, 6^3, 3^7,2^3)$ on the right with diagonal lengths $T_P=(1, 2, \ldots, 9, 10^2, 9, \ldots, 2, 1)$ and $T_Q=(1, 2, \ldots, 9, 10, 9, \ldots, 2, 1)$, respectively. The branch label of $P$ is $\mathfrak{b}_P=\big(6,7,8,1,2,0,9,10,3,4,5\big)$ and the branch label of $Q$ is $\mathfrak{b}_Q=\big(7,8,6, 0,1,2,0,9,3,4,5\big)$. See Example \ref{labeldefnexample}.}\label{labeldefnillust}
\end{figure}
\end{comment}
\par For an illustration of all partitions of diagonal lengths $T=(1, 2^k,1)$ for $k=4$ and $k=1$, and their corresponding labels for branches attached to $\Delta_2$, see Figure \ref{altbranch2kdiagram}.

\begin{figure}
\begin{center}
\includegraphics[scale=.475]{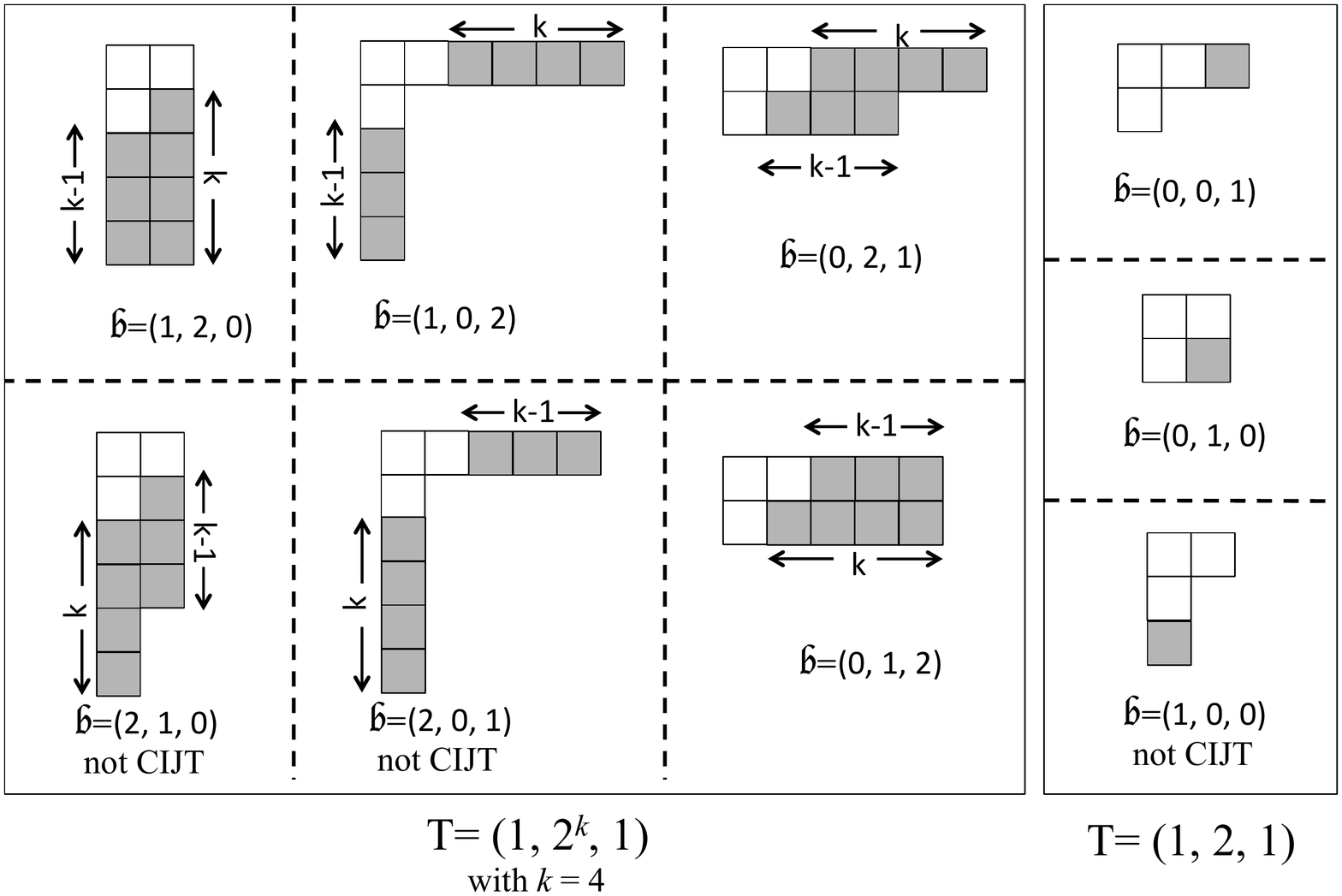}\vspace{-0.5 in}
\end{center}
\caption{Jordan Types and their associated branch labels for $T=(1,2^k,1)$, with $k\in \{1,4\}$.}\label{altbranch2kdiagram}
\end{figure}
The following lemma provides a characterization of branch labels associated with partitions of diagonal lengths $T$ when $k>1$.
%\end{comment}
\begin{lemma}[Branch labels when $k>1$]\label{partlabel}
Assume that $T=(1, 2, \ldots, d^k, \ldots, 2,1)$, satisfying Equation \eqref{CITeq} with $k>1$. If $P$ is a partition having diagonal lengths $T$, then it can be labeled by a sequence $\mathfrak{b}$ 
of the form $\mathfrak{b}=\left(\mathfrak{v}_0 \ldots, \mathfrak{v}_\epsilon, 0, \mathfrak{h}_{\epsilon+1} \ldots, \mathfrak{h}_c\right)$ where the $\mathfrak{v}_i$'s and $\mathfrak{h}_i$'s are distinct subintervals of $\{1, \ldots, d\}$ such that 

\begin{equation}
\begin{array}{l}
\displaystyle{\left(\cup_{0}^\epsilon\mathfrak{v}_i\right)} \cup \displaystyle{\left(\cup_{\epsilon+1}^c\mathfrak{h}_i\right)}=\{1, \ldots, d\},\\ 
\min(\mathfrak{v}_0)>\min(\mathfrak{v}_1)>\cdots >\min(\mathfrak{v}_\epsilon) \mbox{ and } \\ 
\max(\mathfrak{h}_{\epsilon+1})>\max(\mathfrak{h}_{\epsilon+2})>\cdots >\max(\mathfrak{h}_c).
\end{array}\label{BCITTeqn}
\end{equation}

Conversely, the Ferrers diagram associated as in \eqref{Beqn} to a branch label of the above form \eqref{BCITTeqn} represents a partition $P_{\mathfrak {b}}$ of diagonal lengths $T$.

\end{lemma}
\par\noindent
{\bf Note.} Because the intervals are non-overlapping, we may replace $\max$ by $\min$ or vice versa in either chain of inequalities in \eqref{BCITTeqn}.

\begin{proof}
Let $\mathfrak{b}=(b_0, \ldots, b_d)$ be a permutation of $\{0, \ldots, d\}$ and let $P$ be the sequence of numbers counting the lengths of rows of the Ferrers diagram associated with $\mathfrak{b}$ as constructed in \eqref{Beqn} above. Then $P$ is a partition if and only if going from left to right no column is followed by a longer column and going from top to bottom no row is followed by a longer row. By construction of the diagram, we only need to check the first $e$ columns and the first $d-e$ rows of the Ferrers diagram, where $e$ is such that $\mathfrak{b}_e=0$. Thus $P$ is a partition if and only if for $0\leq i <e$, $$b_{i+1}+d-(i+1)+k-2\leq b_{i}+d-i+k-2,$$ and for $e<i\leq d$, $$b_{i+1}+d-(e-i)+k-2\leq b_{i}+d-(e-i-1)+k-2.$$ 

Thus $P$ is a partition if and only if for all $i \neq e$ 
\begin{equation}\label{partineq}
b_{i+1}\leq b_i+1.
\end{equation}

This in particular means that in $\mathfrak{b}$, going from each entry to the next, avoiding $b_e=0$, the value either goes up by exactly 1 or it drops. In fact if $P$ is a partition then the corresponding $\mathfrak{b}$ is the concatenation of intervals of the form $\Big(b, \, b+1,\, \ldots, \,  b+x\, \Big)$. 
\vskip 0.2cm
To prove the statement of the Lemma, we first assume that $P$ is a partition of diagonal lengths $T$ and that $\mathfrak{b}$ is its corresponding branch label. We will show that $\mathfrak{b}$ has the form described in the Lemma. Let $e=\mathfrak{b}^{-1}(0)$.

We define $D_v=\{i\, |\,  1\leq i < e \mbox{ and } b_{i}<b_{i-1}\}.$ The elements of $D_v$ correspond to the positions in ``the vertical part" of sequence $\mathfrak{b}$ where there is a descent -- the entries drop. In other words, if $a$ and $a'$ are two consecutive elements in $D_v$, and if $a<i<a'$, we have $i\not \in D_v$, and therefore by \eqref{partineq}, $b_i=b_{i-1}+1$. Thus $$(b_a, \ldots, b_{a'-1})=(b_a, b_a+1, \ldots, b_a+[a'-a-1]).$$

We also note that by the definition of $D_v$, the integer $b_{a'}$ is strictly smaller than its previous entry, which, as seen above, is equal to $b_a+a'-a-1$. Since the entries of $\mathfrak{b}$ are distinct and $b_a, \ldots,  b_a+(a'-a-1)$ are already in $\mathfrak{b}$, we can conclude that when $a<a'$ in $D_v$, we have $b_a'<b_a$.

 \medskip

Next, we make a similar analysis of  the ``horizontal part" of $\mathfrak{b}$. Let $D_h=\{i\, |\,  e  < i< d \mbox{ and } b_{i+1}<b_{i}\}.$ Assume that $a$ and $a'$ are two consecutive elements in $D_h$. If $a<i<a'$ then $i\not \in D_h$, and by \eqref{partineq}, $b_{i+1}=b_{i}+1$. Thus $$(b_{a+1}, \ldots, b_{a'})=(b_{a'}-(a'-a-1), \ldots, b_{a'}-1, b_{a'}).$$

We also note that by the definition of $D_h$, $b_{a}$ is strictly greater than the next entry in $\mathfrak{b}$, namely $b_{a'}-(a'-a-1)$. Again, using the fact that the entries of $\mathfrak{b}$ are distinct and that $b_{a'}-(a'-a-1), \ldots,  b_{a'}$ are already in $\mathfrak{b}$, we conclude that if $a<a'$ in $D_h$, then $b_a'<b_a$.

\medskip

Therefore, if $P$ is a partition having diagonal lengths $T$, then it can be labeled by a sequence $\mathfrak{b}$ satisfying the conditions \eqref{BCITTeqn} of the Lemma, as desired. 
\bigskip

To prove the converse, assume that $\mathfrak{b}$ satisfies Equation \eqref{BCITTeqn}. We will show that the corresponding sequence $P_{\mathfrak {b}}$ formed by the lengths of rows of the Ferrers diagram defined in \eqref{PfromBeqn} is a partition having diagonal lengths $T$. 

Let $\mathfrak{b}=\left(\mathfrak{v}_0 \ldots, \mathfrak{v}_\epsilon, 0, \mathfrak{h}_{\epsilon+1} \ldots, \mathfrak{h}_c\right)$ and assume that for $0\leq i \leq \epsilon$, $\min(\mathfrak{v}_i)=m_i$ and $|\mathfrak{v}_i|=a_i$, and for $\epsilon < i \leq c$, $\max(\mathfrak{h}_i)=M_i$ and $|\mathfrak{h}_i|=a_i$. Then by assumption 
\begin{equation}\label{BcondCIT}
\begin{array}{l}

\mathfrak{v}_i=\{m_i, m_i+1, \ldots, m_i+a_i-1\},\\ 
m_0>\cdots>m_\epsilon,\\ 
 \mathfrak{h}_i=\{M_i-a_i+1, \ldots, M_i\}, \\
 M_{\epsilon+1}>\cdots >M_c, \mbox{ and } \\ 
 \mathfrak{b}(e)=0 \mbox{ for } e=a_0+\cdots+a_\epsilon.
\end{array}
\end{equation}

Now assume that $i\neq e$ and compare $\mathfrak{b}(i)=b_i$ and $\mathfrak{b}(i+1)=b_{i+1}$. If there exists $j$ such that both $b_i$ and $b_{i+1}$ belong to $\mathfrak{v}_j$ or $\mathfrak{h}_j$, then we obviously have $b_{i+1}=b_i+1$. If $b_i$ and $b_{i+1}$ are in different subintervals of $\mathfrak{b}$, then by the assumption about $\mathfrak{b}$, one of the following occurs:
\begin{itemize}
\item[(a)] There exists $0\leq j < \epsilon$ such that $b_i=\max(\mathfrak{v}_j)$ and $b_{i+1}=\min(\mathfrak{v}_{j+1})$;
\item[(b)] $b_i=\max(\mathfrak{v}_\epsilon)$ and $b_{i+1}=0$;
\item[(c)] There exists $\epsilon+1\leq j < c$ such that $b_i=\max(\mathfrak{h}_j)$ and $b_{i+1}=\min(\mathfrak{h}_{j+1})$.
\end{itemize}

Using the assumptions \eqref{BcondCIT} about $\mathfrak{b}$, we have that for $0\leq j < \epsilon$, 

$$\max(\mathfrak{v}_j)=m_j+a_j-1\geq m_j>m_{j+1}=\min(\mathfrak{v}_{j+1}),$$

and for $\epsilon+1\leq j < c$

$$\max(\mathfrak{h}_j)=M_j\geq M_{j+1}\geq M_{j+1}-a_j+1=\min(\mathfrak{h}_{j+1}).$$

Therefore, the inequality $b_{i+1}\leq b_i+1$ holds for all $i\neq e$. Thus by \eqref{partineq} $P$ is a partition of diagonal lengths $T$, as claimed.
\end{proof}

\begin{example}
We revisit the partition $P=(19^2, 11^3, 5^3, 3^8)$ with diagonal lengths $T_P=(1, 2, \ldots, 9, 10^2, 9, \ldots, 2, 1)$ and branch label $\mathfrak{b}_P=\big(6,7,8,1,2,0,9,10,3,4,5\big)$ from Example \ref {labeldefnexample}. In light of Lemma \ref{partlabel}, we can partition the label of $P$ into increasing vertical and horizontal subintervals as follows. $$\mathfrak{b}_P=\big(\{6,7,8\},\{1,2\},0,\{9,10\},\{3,4,5\}\big).$$ Here $e=5$, the vertical subintervals are $\mathfrak{v}_0=\{6,7,8\}$, $\mathfrak{v}_1=\{1,2\}$,  and the horizontal subintervals are $\mathfrak{h_2}=\{9,10\}$ and $\mathfrak{h_3}=\{3,4,5\}$. We also note that. as indicated in the inequalities \eqref{BCITTeqn}, the minima (equivalently maxima) of the vertical subintervals are decreasing as we move from one subinterval to the next. Similar inequalities also hold in the horizontal part of the label, while
there is no such requirement on how the minima of a vertical subinterval and of a horizontal subinterval compare.\par 

\end{example}

Next we prove a similar Lemma for the case $T=(1, \ldots, d-1, d^k, d-1, \ldots, 1)$ with $k=1$. 

We note that in this case, the Ferrers diagram of a partition $P$ having diagonal lengths $T$ is obtained from $\Delta_d$ by attaching $d-1$ branches of lengths $1, \ldots, d-1$ and leaving two gaps. We also note that since $P$ is a partition, the space between the two gaps in its Ferrers diagram must be ``filled up". In other words, If $P$ has gaps at positions corresponding to the monomials $x^vy^{d-v}$ and $x^hy^{d-h}$ with $v<h$, then for $1 \leq i \leq h-v-1$, a vertical branch of length $i$ is attached to column $v+i+1$  (equivalently,  a horizontal branch of length $i$ is attached to row $(d-h+i+1)$). Now each partition $P$ of diagonal lengths $T$ is obtained by attaching branches of lengths $h-v, \ldots, d-1$ to the remaining attachment places where the attachments to columns $1, \ldots, v$ are all vertical and attachments to rows $1, \ldots, d-h$ are all horizontal. In fact, each such $P$ is labeled by a $(d+1)$-tuple $\mathfrak{b}$ defined as follows.
\medskip
\begin{equation}\label{label2eqn}
\begin{array}{l}
\mathfrak{b}(i)=\left\{
\begin{array}{lcl} 0 && i=v, \, h\\ 
v+i+1&& v<i <h\\ 
\mbox{[Length of the }(i+1)\mbox{-st column of } P] -(d-i)&&i<v\\
\mbox{[Length of the }(i-h)\mbox{-th row of }  P] -(d-i+h+1)&&i>h\\
 \end{array}\right.
 \end{array}
 \end{equation}

Using the same arguments as in the proof of Lemma \ref{partlabel} we can show the following Lemma for this case, that the multiplicity of $d$ in $T$ is one.

\begin{lemma}[Branch labels when $k=1$]\label{partlabel2}
Assume that $T=(1, 2, \ldots, d-1,d,d-1, \ldots, 2,1)$. If $P$ is a partition having diagonal lengths $T$, then it can be labeled by a sequence $\mathfrak{b}$ 
of the form $\mathfrak{b}=\left(\mathfrak{v}_0 \ldots, \mathfrak{v}_\epsilon, 0, 1, \ldots, g-1, 0,  \mathfrak{h}_{\epsilon+1} \ldots, \mathfrak{h}_c\right)$
where $0<g\leq d$, and the $\mathfrak{v}_i$'s and $\mathfrak{h}_i$'s are distinct non-overlapping subintervals of $\{g, \ldots, d-1\}$ such that 
\begin{equation}\label{branchmult12beqn}
\begin{array}{l}
\displaystyle{\left(\cup_{0}^\epsilon\mathfrak{v}_i\right)} \cup \displaystyle{\left(\cup_{\epsilon+1}^c\mathfrak{h}_i\right)}=\{g, \ldots, d-1\},\\ 
\min(\mathfrak{v}_0)>\min(\mathfrak{v}_1)>\cdots> \min(\mathfrak{v}_\epsilon) \mbox{ and } \\ 
\max(\mathfrak{h}_{\epsilon+1})>\max(\mathfrak{h}_{\epsilon+2})>\cdots> \max(\mathfrak{h}_c).
\end{array}
\end{equation}
Conversely, the Ferrers diagram associated to a branch label of the form described above represents a partition of diagonal lengths $T$.

\end{lemma}

\begin{example}
Applying Lemma \ref{partlabel2} to a partition $Q=(19^2, 11^3, 5^3, 3^8)$ with diagonal lengths $T_Q=(1, 2, \ldots, 9, 10, 9, \ldots, 2, 1)$ and branch label $\mathfrak{b}_Q=\big(7,8, 6, 0, 1,2,0,9,3,4,5\big)$, we can write the branch label as $$\mathfrak{b}_Q=\big(\{7,8\}, \{6\}, 0, 1,2 ,0,\{9\},\{3,4,5\}\big).$$

Here $g=3$, the vertical part of the label consists of two subintervals $\mathfrak{v}_0=\{7,8\}$ and $\mathfrak{v}_1=\{6\}$, the horizontal part consists of subintervals $\mathfrak{h}_2=\{9\}$ and $\mathfrak{h}_3=\{3,4,5\}$. 

\end{example}
\subsubsection{Counting partitions having diagonal lengths $T=(1, \ldots, d-1,d^k ,d-1,\ldots, 1)$.}
We give a direct proof of the special case of \cite[Theorem 3.30, Equation 3.35]{IY} for sequences $T$ satisfying Equation \eqref{CITeq}.
\begin{corollary}\label{PdiagTcor} Let $T$ satisfy Equation \eqref{CITeq}.  Then when $k>1$ there are $2\cdot 3^{d-1}$ partitions having diagonal lengths $T$; when $k=1$ there are $3^{d-1}$ such partitions.  These are exactly the partitions that can occur as the Jordan types of linear forms in an algebra having Hilbert function $T$.
\end{corollary}
\begin{proof}
In light of Lemma \ref{partlabel}, to count the number of partitions having diagonal lengths $T$ with $k>1$, it is enough to count the number of branch labels of the form given in the Lemma. Each such label is uniquely determined by first, partitioning the set $\{1, \ldots, d\}$ into subintervals, then breaking that set up into two subsets, one subset for the vertical part of the label and one subset for the horizontal part of the label. When this designation is made, there is a unique way to arrange the subintervals for each part of the label, with decreasing minima for the vertical part and decreasing maxima for the horizontal part.  For each $x$ between 1 and $d$ , there are ${d-1 \choose x-1}$ ways to divide the interval $\{1, \ldots, d\}$ into $x$ subintervals, we simply need to choose $x-1$ ``cutting positions'' from the $d-1$ spaces between elements of $\{1, \ldots, d\}$. Once we divide up the interval into $x$ subintervals, there are $2^x$ ways of designating the vertical and horizontal roles to them. So the total number of valid labels for a partition having diagonal lengths $T$, so the total number of partitions of diagonal lengths $T$ is  
\begin{equation}\label{countdiatT1eq}
\sum_{x=1}^{d}{d-1 \choose x-1} \cdot 2^x=2 \cdot 3^{d-1}.
\end{equation}
\smallskip
Similarly, for $k=1$, using Lemma \ref{partlabel2} to generate labels that correspond to partitions of diagonal lengths $T$, we first need to choose a non-negative integer $g\in \{1, \ldots, d\}$ to represent the distance between the two gaps. For $g=d$ there is only one label, namely $(0, 1, \ldots, d-1, 0)$. For $0\leq g <d$, making a valid label is in fact equivalent to partitioning $\{g, \ldots, d-1\}$ into subintervals, and then dividing up the subintervals into two groups, one for the vertical part and one for the horizontal part of the label. The order in which these intervals appear is forced by the conditions on their maxima and minima. Thus, for  $0\leq g <d$, we can produce $$\sum_{x=1}^{d-g} {d-g-1 \choose x-1} \cdot 2^x=2 \cdot 3^{g-1}$$
distinct branch labels. Therefore, the total number of partitions having diagonal lengths $T$ when $k=1$ is $$\sum_{g=1}^{d}2 \cdot  3^{g-1} +1=3^{d-1}.$$
\end{proof}

\subsection{Partitions having complete intersection Jordan type.}\label{goodsec}
In this section we present some of our main results. Using the branch labels defined and studied in Section \ref{attachbranch1sec}, we characterize CIJT partitions having given diagonal lengths. Recall that 
we say a partition $P$ of diagonal lengths $T$ satisfying Equation~\eqref{CITeq} has CIJT (complete intersection Jordan type) if it can occur as the Jordan type of a linear form $\ell\in A_1$ for some graded CI algebra of Hilbert function $T$.

 \begin{theorem}[Branches of CIJT Partitions, $k> 1$]\label{CIpartsthm}
Let $T=(1, 2, \ldots, d^k, \ldots, 2,1)$ as in Equation \eqref{CIT2eq} with $d\geq 2$ and $k>1$. A partition $P$ of diagonal lengths $T$ has CIJT if and only if there exist an integer $0\leq e\leq d$, and an increasing sequence  $0=a_0<a_1< \cdots<a_c=d-e$ such that the branch label of $P$ satisfies $
\mathfrak{b}=\Big(\mathfrak{v}, 0, \mathfrak{h}_1, \ldots \mathfrak{h}_c\Big)$,
where $\mathfrak{v}$ is the (possibly empty) ordered interval $\{x\,|\, 0\leq x <e\}$, and for $i=1, \ldots c$, $\mathfrak{h}_i$ is the ordered interval $\{x\,|\, d-a_i< x \leq d-a_{i-1}\}$.  
\end{theorem}
\medskip

\begin{proof}

``$\Leftarrow"$

First assume that $P$ is a partition with a branch label $\mathfrak{b}$ as described above. We will argue that $P$ is a CIJT partition. Given the form of $\mathfrak{b}$, the associated partition $P$ is 
\begin{equation}\label{goodpart}
P=\left(\bigcup_{i=1}^c\Big(2d-(a_{i-1}+a_i)+k-1\Big)^{a_i-a_{i-1}},  \big(e\big)^{e+k-1}\right).
\end{equation}
For $1<i<c$, 
$$\Big(2d-(a_{i-2}+a_{i-1})+k-1\Big)=(a_{i-1}-a_{i-2})+(a_i-a_{i-1})+\Big(2d-(a_{i-1}+a_i)+k-1\Big).$$
This shows that the ``horizontal" part of $P$ satisfies the criterion \eqref{good2eq} of Lemma~\ref{criterioncor} (Fig. \ref{thmfig}).
\medskip

On the other hand,  since by assumption $a_c=d-e$, we also have $$2d-(a_{c-1}+a_c)+k-1=(e+k-1)+(a_c-a_{c-1})+e.$$
Thus, by Lemma~\ref{criterioncor}, the partition $P$ is indeed a CIJT partition, as desired.
\bigskip\par
``$\Rightarrow$" Now assume that $P$ is a CIJT partition and let $\mathfrak{b}$ be its corresponding branch label in the form given by Lemma~\ref{partlabel}. 
\medskip

We first show that the ``vertical part" of $\mathfrak{b}$ is either empty or is the single interval $\{1, \ldots, e\}$ for a positive integer $1\leq e \leq d $.
\medskip

By way of contradiction, first assume that the vertical part of $\mathfrak{b}$ consists of at least two distinct subintervals. This in particular implies that there exists an integer $j$, $1\leq j<e$ such that $\mathfrak{b}(j)< \mathfrak{b}(j-1)$ (in going from one subinterval to the next the first entry drops). Since columns $j-1$ and $j$ of the partition have, respectively, the lengths $\mathfrak{b}(j-1)+d-j+1+k-2$ and $\mathfrak{b}(j)+d-j+k-2$, the assumption $\mathfrak{b}(j)< \mathfrak{b}(j-1)$ implies that there is a drop of at least 2 from column $j-1$ to column $j$ in $P$. Therefore, in this case $P$ fails the criterion from Lemma \ref{criterionlem} and is not a CIJT partition. Thus the assumption that $P$ has CIJT, implies that the ``vertical" part of $\mathfrak{b}$ consists of at most one interval. There is nothing to prove if $e=0$, when the vertical part of $\mathfrak{b}$ is empty, or if $e=d$, when the vertical part of $\mathfrak{b}$ is the whole set $\{1, \ldots,d\}$. Now assume that $0<e<d$: we will show that in this case the vertical part of $\mathfrak{b}$ is $\{1, \ldots, e\}$. We showed that the vertical part of $\mathfrak{b}$ is a single interval, say of the form $\mathfrak{v}=\{x\,|\, m\leq x<m+e\}$, for an integer $m \in \{1, \ldots, d\}$. We will next show that $m=1$. 

By way of contradiction assume that $m>1$. Then $1$ is not in $\mathfrak{v}$ and therefore it has to be in the horizontal part of $\mathfrak{b}$. Let $i=\mathfrak{b}^{-1}(1)$. Then by \eqref{partineq}, starting in row $i-e$ and going down through the horizontal part of $\mathfrak{b}$, since the entries can not go down they have to go up by one. Indeed, for $i\leq j\leq d$, $\mathfrak{b}(j)=j-i+1$, and these entries all correspond to rows of length $1+d-i+e+k-1$ in $P$. Furthermore, since $m\in \mathfrak{v}$, $d-i+1<m$. Thus  $1+d-i+e+k-1<m+e+k-1$. 

On the other hand, using the assumption $\mathfrak{v}=\{m, \ldots, m+e-1\}$ again, we see that the smallest part of $P$, which is generated by the vertical part of $\mathfrak{b}$, has size $e$ and multiplicity $m+e+k-2$, which add up to $m+2e+k-2$. Since by assumption $e\geq 1$, $m+2e+k-2\geq m+e+k-1$. Thus in this case $P$ fails the ``criterion" which contradicts the assumption that $P$ is a CIJT partition. Thus in this case $m=1$, as desired. This completes the proof of the claim that the vertical part of $\mathfrak{b}$ is in fact of the form $\mathfrak{v}=\{x\,|\,1\leq x\leq e\}$ for an integer $0\leq e \leq d $. 
\par
Using Lemma \ref{partlabel} we write $\mathfrak{b}=\big(\mathfrak{v}, 0, \mathfrak{h}_1, \ldots, \mathfrak{h}_c\big),$
where $\mathfrak{h}_i=\{ M_i-x_i+1, \ldots, M_i\}$ for a decreasing sequence $M_1>\cdots>M_c$. 
Since the intervals $\mathfrak{h}_i$ partition the interval $\{x\, |\, e<x\leq d\}$, and their maxima are arranged in a decreasing order, setting $$a_i=\displaystyle{\sum_{1\leq j\leq i}x_i},$$ we see that for $1\leq i\leq c$, $\mathfrak{h}_i's$ have the desired form. This completes the proof of the theorem. 
\end{proof}

\begin{figure}
\begin{center}
\includegraphics[scale=.4]{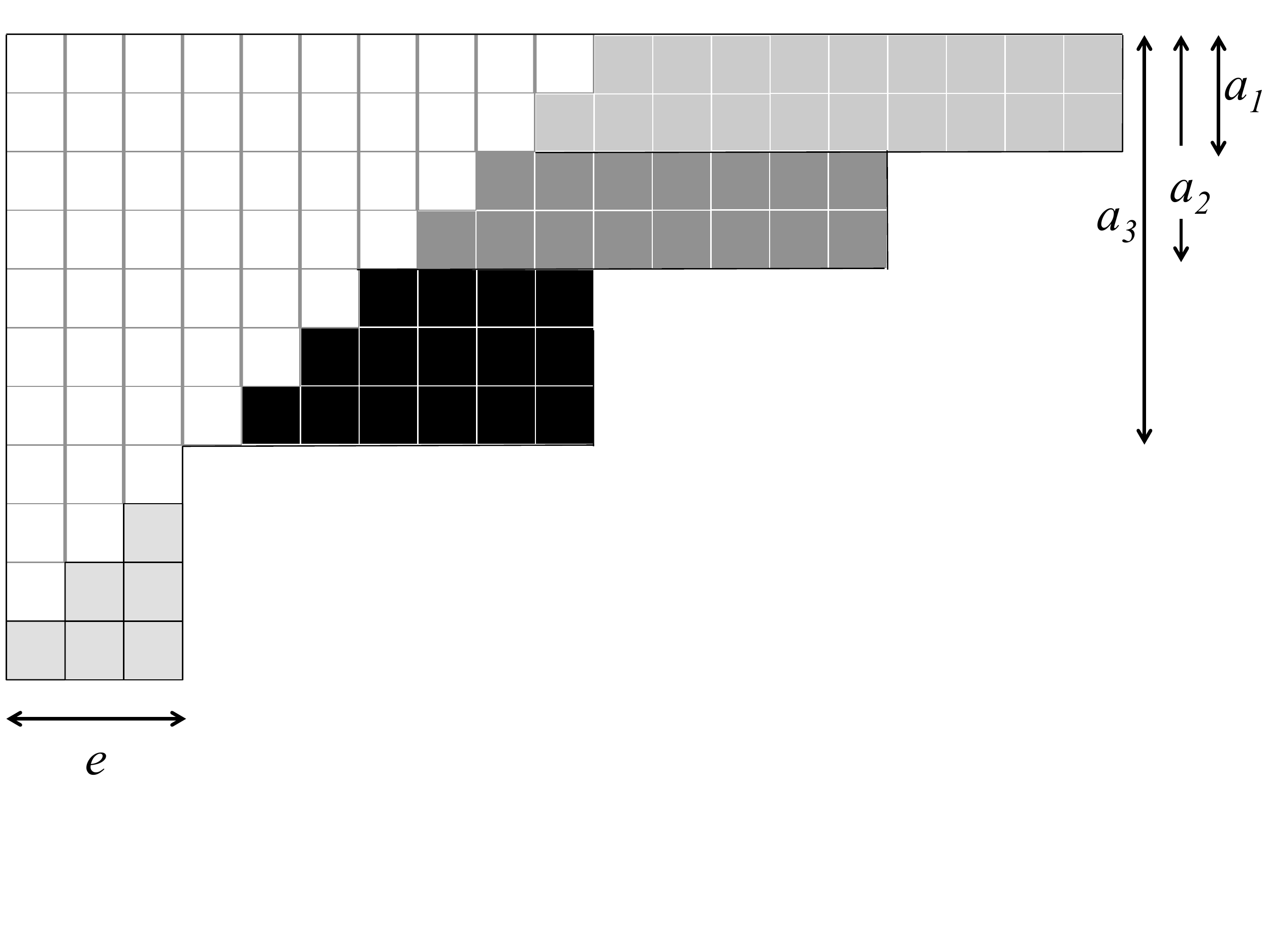}
\end{center}
\vspace{-0.7 in}

\caption{Illustration for Theorem \ref{CIpartsthm} of a CIJT partition with diagonal lengths $T=(1, 2, \ldots, 9, 10^2, 9, \ldots, 2, 1)$ and branch label $\mathfrak{b}=(\{1, 2, 3\}, 0, \{9,10\}, \{7, 8\},\{4, 5, 6\}).$}\label{thmfig}

\end{figure}

We note that Theorem \ref{CIpartsthm} also establishes a one-to-one correspondence between the set of CIJT partitions and the set of all increasing sequences $0=a_0<\cdots<a_c\leq d$, for $c=0, \ldots, d$. (We always have $e=d-a_c$.) Each such increasing sequence $(a_i)$ can also be uniquely determined by its differences $(n_i=a_i-a_{i-1})_i$ which is an ordered partition of $n$ (a partition of $n$ in which the order of parts matter) for $n\in \{0, 1, \ldots, d\}$. Thus by Theorem \ref{CIpartsthm} and  \eqref{goodpart} we get the following corollary which characterizes all CIJT partitions for a given Hilbert function. 
\bigskip

\begin{corollary}[CIJT partitions, $k>1$]\label{goodpartcor}
Let $T=(1, 2, \ldots, d^k, \ldots, 2,1)$ with $d\geq 2$ and $k>1$. A partition $P$ can occur as the Jordan type of a linear form for some complete intersection algebra of Hilbert function $T$ satisfying \eqref{CITeq} if and only if there exists an integer $n\in [0,d]$ and an ordered partition $n=n_1+\cdots+n_c$ (empty partition when $n=0$) such that 
 \begin{equation}\label{goodpartcorpart}P=\big(p_1^{n_1}, \ldots, p_c^{n_c}, (d-n)^{d-n+k-1}\big),\end{equation} where $p_i=k-1+ 2d-n_i-2\displaystyle{\sum_{j<i} n_j}$, for $1\leq i \leq c$.
\end{corollary}

We note that in Corollary \ref{goodpartcor}, if $n=d$ then the partition in \eqref{goodpartcorpart} is in fact $\big(p_1^{n_1}, \ldots, p_c^{n_c}\big)$, which has $d$ parts.
\medskip

Using arguments similar to the ones used in the proof of Theorem \ref{CIpartsthm}, we show the following theorem, which characterizes branch labels associated with CIJT partitions of diagonal lengths $T=(1, \ldots, d-1,d, d-1,\ldots, 1)$ (Fig. \ref{CIpartsthm2fig}). An entirely similar argument to that used for Corollary~\ref{goodpartcor} also yields Corollary~\ref{goodpartcor2}.

\begin{theorem}[Branches of CIJT partitions, $k=1$]\label{CIpartsthm2}
Assume that $T=(1, 2, \ldots,d-1,$ $ d, d-1,\ldots, 2,1)$ with $d\geq 2$. A partition $P$ of diagonal lengths $T$ has CIJT if and only if there exists an integer $1\leq v\leq d-1$, and an increasing sequence  \par $0=a_0<a_1< \cdots<a_c=d-v$ such that $\mathfrak{b}=\big(0,  \mathfrak{v},  0, \mathfrak{h}_1, \ldots \mathfrak{h}_c\big),$ where $\mathfrak{v}$ is the (possibly empty) interval $\{x\, |\, 1\leq x <v\}$ and for $i=1, \ldots c$, each $\mathfrak{h}_i$ is the ordered interval $\{x\,|\, d-a_i\leq x < d-a_{i-1}\}$.

\end{theorem}

\begin{figure}
\vspace{-.4 in}

\begin{center}
\includegraphics[scale=.45]{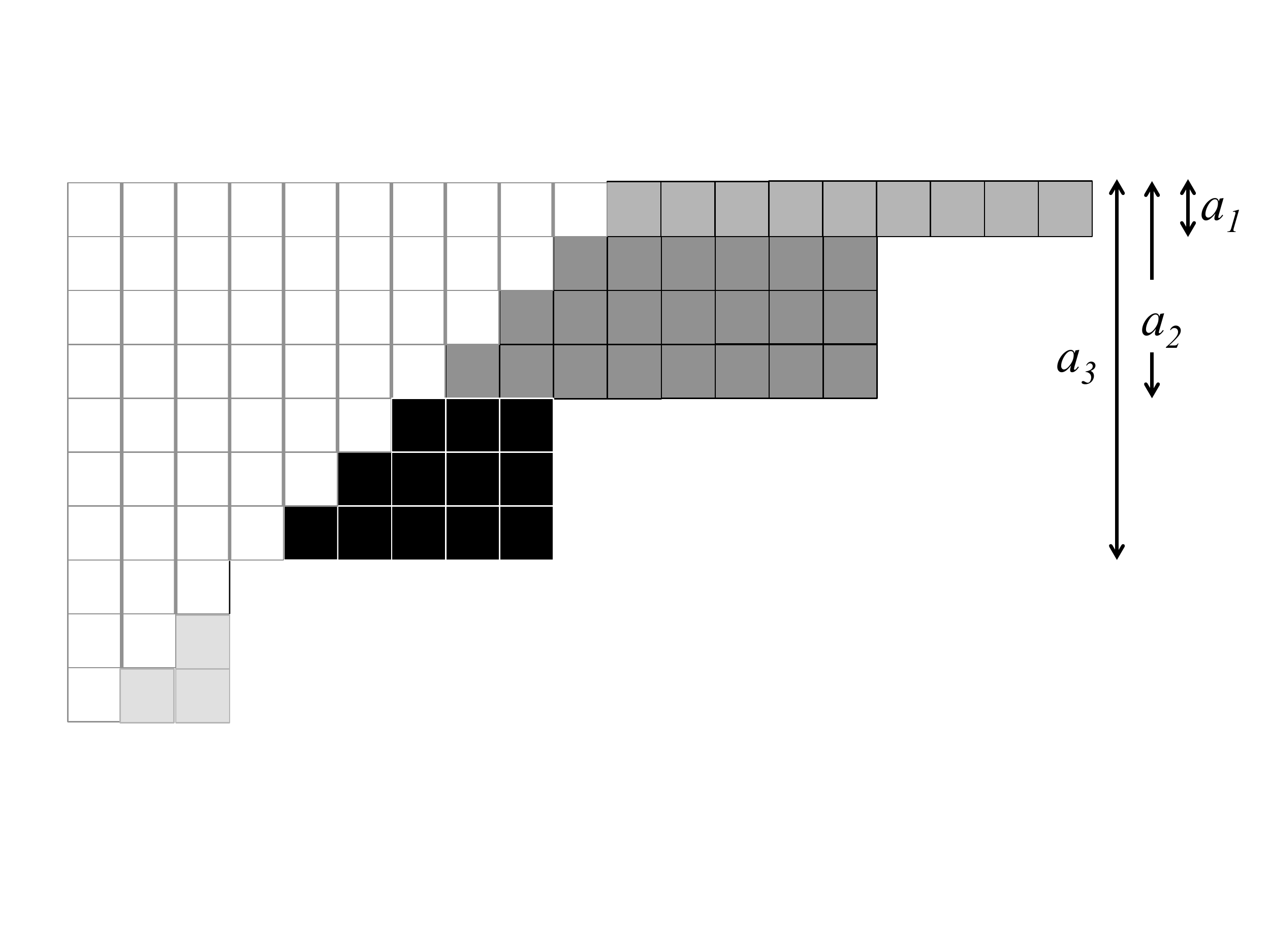}\vspace{-.8 in}
\end{center}
\caption{Illustration of Theorem \ref{CIpartsthm2} for a CIJT partition with diagonal lengths $T=(1, \ldots, 9,10,9, \ldots, 1)$ and branch label $\mathfrak{b}=\big(0, \{1,2\}, 0, \{9\}, \{6,7,8\},\{3,4,5\}\big)$. }\label{CIpartsthm2fig}
\end{figure}

\begin{corollary}[CIJT partitions, $k=1$]\label{goodpartcor2}
Let $T=(1, 2, \ldots,d-1,  d, d-1, \ldots, 2,1)$ with $d\geq 2$. A partition $P$ can occur as the Jordan type of a linear form for some complete intersection algebra of Hilbert function $T$ if and only if there exists an integer $n\in [0,d-1]$ and an ordered partition $n=n_1+\cdots+n_c$ (empty partition when $n=0$) such that  $P$ has $d$ parts of the form\begin{equation}\label{goodpartcor2part}P=\big(p_1^{n_1}, \ldots, p_c^{n_c}, (d-n)^{d-n}\big),\end{equation} where $p_i=2d-n_i-2\displaystyle{\sum_{j<i} n_j}$, for $1\leq i \leq c$.

\end{corollary}
We now show that a CIJT partition must satisfy equality in Equation \eqref{good2eq}.
\begin{theorem}[Combinatorial criterion for CIJT]\label{equalityCIthm}
Let $T=(1, 2, \ldots, d^k, \ldots, 2,1)$ with $d\geq 2$. A partition $P$ of diagonal lengths $T$ can occur as the Jordan type of a linear form for some complete intersection algebra of Hilbert function $T$ if and only if $P=(p_1^{n_1}, \ldots, p_t^{n_t})$ such that for each $i\in[2,t]$, 
\begin{equation}\label{equalityCIeq}
p_{i-1}=n_{i-1}+n_i+p_i.
\end{equation}
\end{theorem}

\begin{proof}
The statement is an immediate consequence of Lemma \ref{criterioncor} and Corollaries \ref{goodpartcor} and \ref{goodpartcor2}.
\end{proof}

\bigskip

\subsubsection{Counting CIJT partitions having diagonal lengths $T=(1, \ldots, d^k, \ldots, 1)$.}
\begin{corollary}\label{goodcountcor}
Assume that the sequence $T$ satisfies Equation \eqref{CIT2eq}. Then there are $2^d$ CIJT partitions having diagonal lengths $T$  if $k>1$, and $2^{d-1}$ if $k=1$.
\end{corollary}
\begin{proof}
First assume that $k>1$. Then by Corollary \ref{goodpartcor}, the number of CIJT partitions having diagonal lengths $T$ is the total number of ordered partitions of $n$ for $0\leq n \leq d$. If $1\leq n \leq d$,  then there are $2^{d-1}$ ordered partitions of $n$, and for $n=0$, there is the empty partition by our convention. So the total number is 
\begin{equation}\label{goodcount1eq}
\sum_{n=1}^{d}2^{n-1}+1=2^{d}.
\end{equation}\noindent

Here is a direct way of counting: Each CIJT branch label is uniquely determined by first choosing an integer $e\in [0,d]$, the size of the vertical part, and then choosing an ordered partition of $d-e$ whose partial sums  give the increasing sequence $a_1<\cdots<a_c$ (the sequence is empty when $e=d$). For each $e\in [0,d)$, the interval  $\{x\,|\, e<x\leq d\}=\{e+1, \ldots, d\}$ can be divided into subintervals by choosing the cutting points from the $d-e-1$ spaces between the elements of the set. Thus for each such $e$, we get $2^{d-e-1}$ labels. There is also one more label for the case $e=d$ (all branches attached vertically). So the total number is $\sum_{e=0}^{d-1}2^{d-e-1}+1=2^{d}.$\par
In Theorem \ref{2Hessthm} we will give a 1-1 correspondence between CIJT partitions of diagonal lengths $T$ and subsets of the non-vanishing Hessians from the $d$ active Hessians. This will give another way to verify the count  of $2^d$ CIJT partitions, which is the number of such subsets.
\medskip

Now assume that $k=1$. A similar argument, using Corollary \ref{goodpartcor2} in place of Corollary \ref{goodpartcor}, implies that the total number of CIJT partitions having diagonal lengths $T$ is the same as the total number of ordered partitions of $n$ for $0\leq n \leq d-1$, which is $2^{d-1}$. 
\end{proof}

\begin{theorem}\label{simplegoodprop}
Let $T=(1, 2, \ldots, d^k, \ldots, 2,1)$ with $d\geq 2$. Assume that $P$ is a partition having diagonal lengths $T$. Then $P$ can occur as the Jordan type of a linear form for some complete intersection algebra of Hilbert function $T$ if and only if $P$ has either $d$ parts (is weak Lefschetz) or $P$ has $d+k-1$ parts. 
\end{theorem}

We note that if $k=1$ then $d+k-1=d$. Thus for the case of $k=1$, a partition $P$ having diagonal lengths $T$ has CIJT if and only of it has $d$ parts (is weak Lefschetz).

\begin{proof}

``$\Rightarrow$" This is immediate from Corollaries \ref{goodpartcor} and \ref{goodpartcor2}.
\bigskip

``$\Leftarrow$" Assume that $P$ is a partition having diagonal lengths $T$ with either $d$ or $d+k-1$ parts. 

If $k=1$ then, since $P$ is a partition having diagonal lengths $T$, by Lemma \ref{partlabel2} its corresponding branch label has the form $\mathfrak{b}=
\left(\mathfrak{v}_0, \ldots, \mathfrak{v}_\epsilon, 0, 1, \ldots, g-1, 0, \mathfrak{h}_{\epsilon+1}, \ldots, \mathfrak{h}_c\right)$ for a an integer $0<g\leq d$. By assumption, $P$ has exactly $d$ parts. This in particular implies that no vertical branch is attached to the first column of $P$. Thus the branch label of $P$ has no vertical part and by Theorem \ref{CIpartsthm2}, $P$ is a CIJT partition. 

Now assume that $k>1$. Then by Lemma \ref{partlabel}, the branch label of $P$ has the following form. 
$$\mathfrak{b}=
\left(\mathfrak{v}_0, \ldots, \mathfrak{v}_\epsilon, 0, \mathfrak{h}_{\epsilon+1}, \ldots, \mathfrak{h}_c\right)
$$

where $\mathfrak{v}_i$s and $\mathfrak{h}_i$s are distinct subintervals of $\{1, \ldots, d\}$ such that 

$$
\begin{array}{l}
\displaystyle{\left(\cup_{0}^\epsilon\mathfrak{v}_i\right)} \cup \displaystyle{\left(\cup_{\epsilon+1}^c\mathfrak{h}_i\right)}=\{1, \ldots, d\},\\ 
\min(\mathfrak{v}_0)>\min(\mathfrak{v}_1)>\cdots >\min(\mathfrak{v}_\epsilon) \mbox{ and } \\ 
\max(\mathfrak{h}_{\epsilon+1})>\max(\mathfrak{h}_{\epsilon+2})>\cdots >\max(\mathfrak{h}_c).
\end{array}
$$

If $P$ has $d$ parts then there is no vertical branch attachment in $P$, which implies that the vertical part of $\mathfrak{b}$ is empty. On the other hand, if $P$ has $d+k-1$ parts then in $P$ a branch of length $k-1$ is vertically attached at the bottom of the first column of $\Delta_d$. This simply means $\min(\mathfrak{v}_0)=1$. Since the minima of $\mathfrak{v}_i$'s  form a decreasing sequence, the vertical part of $\mathfrak{b}$ includes only one vertical subinterval of the form $\{1, \ldots, e\}$. As we saw in the proof of Theorem \ref{CIpartsthm} when the horizontal part of $\mathfrak{b}$, with the given condition on the maxima, partitions $\{e+1, \ldots, d\}$, then $\mathfrak b$ has the form described in Theorem \ref{CIpartsthm}, hence, by the  Theorem, $P$ is a CIJT partition.
\end{proof}\par
When $k\ge 2$ we describe the relation between the CIJT partitions with $d $ parts and with $d+k-1$ parts precisely in Theorem \ref{tablethm}.

\section{Vanishing of Hessians and CIJT Partitions.}\label{corrsec}\par
In this section, we find all possible Jordan types which occur for Artinian CI algebra in $R={\sf k}[x,y]$, using the vanishing of Hessians.  After defining Hessians we first report in Section \ref{genericFsec} a special case, those Jordan types where only one Hessian vanishes (Theorem \ref{genericthm} of N. Altafi and M. Boij). In Section \ref{arbitraryFsec} we show our main results, Theorem~\ref{sumOfPartitions} and Theorem~\ref{2Hessthm} characterizing the CIJT partitions in terms of the vanishing of Hessians. In Section \ref{latticesec} we show that dominance of CIJT partitions arises from the subsets of Hessians that vanish (Theorem \ref{posetprop}) and that the closures of the cells $\mathbb V(E_P)$ satisfy a frontier property (Theorem \ref{mainHessthm}).\par
In Section \ref{patternsec} we give a further combinatorial description of the CIJT partitions, and we show a one-to-one correspondence between those with $d$ parts and those with $d+k-1$ parts when $k\ge 2$ (Theorem \ref{tablethm}).\par
Assume that $A$ is a graded Artinian CI algebra with the Hilbert function $T= (1,2,\dots , d^k , \dots ,2,1_j)$, with $k\geq 1$, socle degree $j$, and Macaulay dual generator $F\in \mathcal E_j={\sf k}[X,Y]_j$. Recall that $R$ acts on $\mathcal E$ by differentiation (Equation \eqref{dualeq}). Throughout Section \ref{corrsec}, we will assume $\cha{\sf k}=0$ for statements involving the Hessian; however, if these are rephrased using the rank of multiplication maps, they are valid also in characteristic $p>j$, the socle degree of $A$. For example Theorem \ref{sumOfPartitions} may be so rephrased, as there we use the combinatorial properties of Jordan type, and then using Lemma \ref{Hesslem}, rephrase in terms of the vanishing of Hessians.
\begin{definition}[$i$-th Hessian]\label{hessdef}\cite{MW},\cite[Def. 3.75]{H-W}.
Let $A=R/\Ann(F)$, where $F\in \mathcal{E}_j$ be a standard graded Artinian Gorenstein ${\sf k}$-algebra, and suppose $\cha {\sf k}=0$. Let $\mathcal{B}_{i}=\left(\alpha_1,\dots ,\alpha_r\right)$ be a ${\sf k}$-linear basis of $A_{i}$. Let $i\in [0,d-1]$. Then $t_i= (i+1)$. The $t_i\times t_i $ matrix 
\begin{equation}\label{Hesskeqn}
\hess^i(F):=\left[\alpha^{(i)}_u\alpha^{(i)}_v\circ F\right]
\end{equation}
is called the  $i$-th Hessian matrix of $F$ with respect to the basis $\mathcal{B}_i$.
We denote by $h^i(F)$ the Hessian determinant $h^i(F)=\det\left(\hess^i(F)\right)$, which is a bihomogeneous form in the coefficients of $F$ and in $X,Y$, respectively having bidegree $\left(t_i, t_i\cdot (j-2t_i)\right)$. Up to a non-zero constant multiple $h^i(F)$ is independent of the basis $\mathcal{B}_{i}$: thus we may regard it as an element of the projective space $\mathbb P^j\times \mathbb P^{(t_i)(j-2t_i)-1}$.
An ``active Hessian'' of $A$ from Hilbert function $T$ of Equation \eqref{CIT2eq} is one of  $h^0(F),h^1(F),\ldots, h^{d-2}(F)$, as well as $h^{d-1}(F)$ if $k\geq 2$.
\end{definition}

We note that when $i=1$ the form $\hess^1(F)$ coincides with the usual Hessian. See Example \ref{3.4ex}.
For $\ell = ax+by$ we denote by $h^i_{\ell}=h^i_{(a,b)}$ the Hessian evaluated at the point ${\sf p}_\ell=(a,b)$.
\begin{lemma}\label{Hesslem}\cite[Theorem 3.2]{MW}, \cite[Proof of Theorem 3.36]{H-W}. Let $A=R/\Ann F$ be an Artinian Gorenstein quotient of $R$, let $\ell=ax+by$ be a linear form, pick bases $\mathcal{B}_{i}$ and $\mathcal{B}_{j-i}$ for $A_i, A_{j-i}$, respectively.  Consider the linear multiplication map, $m_{\ell^{j-2i}}$
\begin{equation}\label{hess2eqn}
m_{\ell^{j-2i}}: A_i\to A_{j-i}, \quad  h\to \ell^{j-2i}\cdot h.
\end{equation}
There is a non-zero constant $c_{i,A}$ such that the determinant of the multiplication map satisfies \begin{equation}\label{hesseqn}
\det (m_{\ell^{j-2i}})=c_{i,A}h^i(F)_{(a,b)}.
\end{equation}
Furthermore, for a fixed $T$, for a general enough dual generator $F$ we have that each $h^i(F)$ has $(i+1)(j-2i)$ distinct roots (no multiple roots).
\end{lemma}
\begin{proof} The first statement is straightforward.  The simplicity of the roots for a general enough $F$ may be concluded from Theorem \ref{genericthm} below, or as follows.  The Equation~\ref{hesseqn} leads to an equality between the Hessian $h^i(F)$ for $i<d$ and a certain Wronskian determinant $W(I_{j-2i})$ where $I=\Ann(F)$.\footnote{To define the Wronskian determinant for a vector space of homogeneous forms goes beyond the scope of the paper, but see \cite[\S 2]{IY}.} By Proposition \cite[4.9]{IY}, an argument involving Schubert calculus, and in a more general setting, we may conclude that $W(I_{j-2i})$ has distinct roots for a general enough $F$.
\end{proof}\par

Since we are primarily interested in the vanishing or non-vanishing of $h^i_\ell$, the value of the nonzero constant in Equation \eqref{hess2eqn} is not important for us.

\subsection{CI Jordan type for a ``general enough'' dual generator $F$.}\label{genericFsec}
The following previously found result provides the list of all Jordan types and their loci for linear forms of an Artinian CI algebra, $A=R/\Ann F$, for a general enough degree $j$ form, $F\in {\sf k}[X,Y]$, the polynomial ring. Given $(A,\ell\in A_1)$ we denote by $f_1^\vee= h^{d-1}_\ell$.
\begin{theorem}[N. Altafi and M. Boij]\label{genericthm}
Let $R={\sf k}[x,y]$ over a field $\sf k$ of characteristic zero, and $A=R/\Ann(F)$ of Hilbert function $H(A)=(1,2,\dots , d^k , \dots ,2,1)$, be a CI algebra where $F\in {\sf k}[X,Y]$, is a sufficiently general homogeneous polynomial for an integer $j\geq 2$, in the following sense: we assume that $F$ is outside the union of all sets of forms $F\in \mathcal E_j$ such that for some linear form $\ell\in R$, two or more of the active Hessians 
\begin{equation}\label{Hesslisteq}
h_\ell^i(F), 0\le i\le d-2, \text { and also $f_1^\vee$ if $k\geq 2$ },
\end{equation}
 have simultaneous roots $p_\ell$. We also assume that $F$ is general enough so that the Hessian $h^i(F)$ that has zeroes has no multiple roots.\footnote{The Hessian matrices have as entries forms in the coefficients of $F$, so each resultant of two of them or discriminant of one is a homogeneous polynomial in the coefficients of $F$: each such polynomial determines a codimension one subvariety of zeroes in $\mathbb P^j=\mathbb P(\mathcal E_j)$, so the ``general'' $F$ is any $F$ in the complement of a union of these resultant subvarieties of codimension one in $\mathbb P^j$; thus, a ``general'' $F$ here is one belonging to a specific open Zariski-dense subset of $\mathbb P^j$.} 
 Then, 
 \begin{itemize}
 \item[$(a)$] If $k=1$, there are exactly $d$ different Jordan types for linear forms of $A$: here $d-1$ of them correspond to each choice of  $i$ satisfying $0\leq i\leq d-2$ where $h^i_\ell(F)=0$ (see Equation \eqref{genJtypeeq}),  and there is one Jordan type for a strong Lefschetz element.
 \item[$(b)$] If $k\geq 2$, there are exactly $d+1$ different Jordan types for linear forms of $A$: here $d-1$ of them correspond to each choice of  $i$ satisfying $0\leq i\leq d-2$ where $h^i_\ell(F)=0$ (see Equation \eqref{genJtypeeq}),  there is one Jordan type for the roots of $f^{\vee}_1$ (see Equation \eqref{genJtype2eq}), and there is one Jordan type for a strong Lefschetz element.
 \end{itemize}
 
 The Jordan type $P_\ell$ of a linear form $\ell$ where, for some integer $i\in [0,d-2]$ the Hessian $h^i_\ell(F)=0$  has an order one zero, and no other active Hessian is zero, is the maximum consecutive subsequence (with $d$ elements) of 
\begin{equation}\label{genJtypeeq}
(\dots, j-2i+7,j-2i+5,j-2i+3,j-2i,j-2i,j-2i-3, j-2i-5, j-2i-7,\dots),
\end{equation}
 for which every entry is greater than or equal to $k$ and less than or equal to $j+1$.
Moreover, for each $i\in [0,d-2]$ there are exactly $(i+1)(j-2i)$ distinct linear forms, corresponding to the roots $p_\ell$ of $h^i(F)=0$, which all have the same Jordan type given in Equation \eqref{genJtypeeq}.
\par
If $k\geq 2$, the Jordan type of linear forms where $f_1^{\vee}=0$, and all other Hessians are non-zero, is 
\begin{equation}\label{genJtype2eq}
 (j+1,j-1, j-3,\dots ,j+1-2(d-2),1^k),
 \end{equation}
 with $d+k-1$ parts.
Moreover, there are exactly $d$ distinct linear forms, corresponding to the roots of  $f_1^{\vee}=0$, with the same Jordan type of Equation \eqref{genJtype2eq}.
All the other linear forms have the strong Lefschetz Jordan type, $H(A)^\vee$.\par
\end{theorem}
\begin{remark}\label{wLpLoci}
We can use Theorem \ref{genericthm} to determine the weak Lefschetz loci for an Artinian CI algebra $A=R/\Ann F$, where $F\in \mathcal{E}_j$ is a sufficiently general form (in the sense of Theorem \ref{genericthm}). Assume that $H(A)=\left(1,2,\dots , d-1,d^k,d-1,\dots ,2,1\right)$ is the Hilbert function of $A$. If $k=1$, every linear form is a weak Lefschetz element for $A$, and corresponds to the Jordan type partition of Equation \eqref{genJtypeeq}. But if $k\geq 2$, the linear form $\ell =ax+by$ is a weak Lefschetz element for $A$ (for a general $F$) if and only if $f_1^{\vee}(p_\ell)\neq 0$,
% :equivalently, $\ell$ is  WL for $A$ iff $\ell$ is not a root of $f_1$,
and this is equivalent under the hypothesis of $F$ sufficiently general, to $P_\ell$ not being the Jordan type partition of Equation \eqref{genJtype2eq}.
\end{remark}

\begin{example}\label{3.4ex}
\begin{itemize}
\item[(i)] Let $A=R/{\Ann F}$ be a complete intersection algebra where $F = (X+Y)^4+(X-Y)^4+(X+2Y)^4\in \mathcal{E}_4$. Here $F$ is a sufficiently general form according to the assumption of Theorem \ref{genericthm}. The Hilbert function $H(A)=(1,2,3,2,1)$, $d=3,k=1$ and $h^0_\ell(F), h^1_\ell(F)$ are the only  active Hessians. By the Theorem there are exactly $3$ different Jordan types for $\ell\in A_1$: $H(A)^{\vee}=\left(5,3,1\right)$, $\left(5,2,2\right)$ (for linear forms $\ell$, where $h^1_\ell(F)=0$) and $(4,4,1)$ (for linear forms $\ell$, where  $h^0_\ell(F)=0$);  they each have $3$ parts so they are weak Lefschetz Jordan types.
\item[(ii)] Let $A=R/{\Ann F}$ be a complete intersection algebra where $F = (X+Y)^5+(X-Y)^5+(X+2Y)^5\in \mathcal{E}_5$; this is a  sufficiently general form. The Hilbert function $H(A)=(1,2,3,3,2,1), d=3,k=2$ and $h^0_\ell(F), h^1_\ell(F)$ and $h^2_\ell(F)$ are the active Hessians. Therefore, there are exactly four different Jordan types: $H(A)^{\vee}=\left(6,4,2\right)$ (for the general linear form), $\left(6,3,3\right)$, (for linear forms $\ell$ satisfying $h^1_\ell(F)=0$), $(5,5,2)$, for linear forms $\ell$ satisfying $h^0_\ell(F)=0$), and $(6,4,1,1)$, (for linear forms satisfying $h^2_\ell(F)=0$). The first three Jordan types are weak Lefschetz Jordan types but the last one, which corresponds to the roots of $h^2(F)$ (or $f_1^{\vee}$) does not have weak Lefschetz Jordan type.
\item[(iii)] Let $F=X^2Y^3$. Then $\Ann F=(x^3,y^4), H(R/\Ann F)=(1,2,3,3,2,1)$. from Figure \ref{Hessfig} using standard monomial bases $\mathcal{B}_{i}=(x,y), \mathcal{B}_{2}=(x^2,xy,y^2)$ we have the Hessian determinants $h^0(F)=F, h^1(F)=-24X^2Y^4, h^2(F)=-12^3Y^3$, with common root $p_\ell=(1,0), \ell=x$. Thus $F$ is not sufficiently general in the sense of Theorem \ref{genericthm}. But for $\ell\not=x,y$ up to scalar, $P_\ell=(6,4,2)$ is strong-Lefschetz. It is readily seen that $P_{x,A}=(3^4)$, and $P_{y,A}=(4^3)$.
\end{itemize}
\end{example}
\begin{figure}\small
$\qquad\qquad\qquad Hess^1(F)=\left(\begin{array}{cc}x^2&xy\\xy&y^2\end{array}\right)\circ F=\left(\begin{array}{cc}2Y^3&6XY^2\\6XY^2&6X^2Y\end{array}\right)$\par
$\qquad\qquad\qquad Hess^2(F)=\left(\begin{array}{ccc}x^4&x3y&x^2y^2\\x^3y&x^2y^2&xy^3\\x^2y^2&xy^3&y^4\end{array}\right)\circ F=\left(\begin{array}{ccc}0&0&12Y\\0&12Y&12X\\12Y&12X&0\end{array}\right).$
\caption{Hessian matrices for $F=X^2Y^3$, see Example \ref{3.4ex}(iii).}\label{Hessfig}
\end{figure}
\normalsize

\subsection{CI Jordan types for an arbitrary dual generator $F$.}\label{arbitraryFsec}
This section contains our main results showing that the CIJT partitions $P=P_\ell$ of diagonal lengths $T$  correspond 1-1 to the sets of Hessians $h_\ell^i(F)$ that can vanish for a linear form $\ell $ in a complete intersection $A$ of Hilbert function $T$.\par
For an Artinian CI algebra $A=R/\Ann F$ where the form $F\in \mathcal{E}_j$ is not general enough in the above sense of Theorem \ref{genericthm}, then several different active Hessians may have simultaneous roots $\ell$, so there are more possible Jordan types. 
We first determine the set of CIJT partitions having a particular non-vanishing Hessian (Theorem~\ref{sumOfPartitions}). We apply this in Theorem \ref{2Hessthm} to show the 1-1 correspondence between CIJT partitions and the $2^d$ (when $k\ge 2$), or $2^{d-1}$ (when $k=1$) subsets of the active Hessians: this count agrees with the number of different complete intersection Jordan types we showed in Corollary \ref{goodcountcor}.  In Proposition \ref{rkseq} we show that the rank of all the multiplication maps $\Hess_\ell^i(F)$ by different powers of $\ell$ when $P_\ell$ is a CIJT partition are determined by which Hessians vanish.
\par
Let $T=\left(1,2,\dots ,d^k,\dots ,2,1\right)$ as in Equation~\ref{CIT2eq} and consider the conjugate partition $T^\vee=\left(2d+k-2,2d+k-4,\dots ,k+2,k\right)$, the strong Lefschetz partition. For a CI Artinian algebra $A$ that is strong Lefschetz \footnote{Every quotient of $R={\sf k}[x,y]$ when $\cha {\sf k}=0$ or $\cha {\sf k}> j$ is strong Lefschetz by \cite{Bri}.} a generic linear form $\ell$ has Jordan type $T^\vee$. Then the higher Hessians are all non-zero at the point $p_\ell$. We now show that the $i$-th Hessian $h^{i}_{\ell}=h^i(F)_{p_\ell}$ is non-zero for an integer $i\in [0,d-1]$ {if and only if} the sum of the first $i+1$ parts of the Jordan type partition $P_\ell$ is equal to the sum of the first $i+1$ parts of $T^\vee$, which is the sum in Equation \eqref{corrformula}. We will write $P_\ell=(p_1,p_2,\ldots)$ with $p_1\ge p_2\ge\cdots$. Recall that the socle degree $j=(2d+k-3)$ and $|P|=|T|=\sum T_i=d(j+2-d)$.
\begin{theorem}[When is a Hessian non-zero?]\label{sumOfPartitions}
Let $P_\ell=(p_1,p_2,\dots ,p_{d},\dots )$ be the Jordan type partition for a linear form $\ell$ of an Artinian CI algebra $A=R/\Ann(F)$ of Hilbert function $H(A)=(1,2,\dots ,d^k,\dots ,2,1)$, for an integer $k\geq 1$.
Then for each $i\in \left[0,d-1\right]$ 
 we have
\begin{equation}\label{corrformula}
h_\ell^{i}(F)\neq 0 \hspace*{0.1 cm} \Longleftrightarrow \hspace*{0.1 cm}\sum^{i+1}_{j=1}p_j=(i+1)(2d+k-i-2)=(i+1)(j+1-i).
\end{equation}
In particular, $P_\ell$ has $d$ parts unless $k\ge 2$ and $h^{d-1}_\ell=0$.
\end{theorem}
\begin{proof}
$``\Rightarrow"$ Suppose that for some integer $i\in [0,d-1]$ we have $h_\ell^i(F)\neq 0$; then the multiplication map
% $\ell^{2d-2i+k-3}=
$m_{\ell^{j-2i}}:A_i\to A_{j-i}$ is an isomorphism and has the maximal rank, that is $i+1$. Therefore, $p_n\geq  j+1-2i=2d-2i+k-2$, for every $n\in\left[0,i\right]$, which correspond to all the basis elements of $A_m$'s, for $m\in \left[i,2d+k-i-3\right]$. Since the multiplication map $m_{\ell^{i-r}}:A_r\rightarrow A_i$, for every $r\in \left[0,i-1\right]$, has trivially maximal rank, the first $i+1$ parts of $P_\ell$ contain $\frac{i(i+1)}{2}$ boxes (i.e. sum to $\frac{i(i+1)}{2}$), corresponding to the basis elements of $A_r$, for every $r\in \left[0,i-1\right]$.

 We claim that there are $\frac{i(i+1)}{2}$ more boxes in the first $i+1$ parts  $P_\ell$.  Since, $\dim_k(A_i)=\dim_k(A_{j-i})=i+1$, and $h_\ell^i(F)\neq 0$, all $i+1$ boxes corresponding to basis elements of $A_{j-i}$ are contained in the first $i+1$ parts in $P_\ell$. Therefore, the $\frac{i(i+1)}{2}$ boxes corresponding to basis elements of $A_r$'s, for $r\in\left[j+1-i,j\right]$, are also contained in the first $i+1$ parts of $P_\ell$. Summing the number of boxes, we have
$$\sum^{i+1}_{j=1}p_j=\frac{i(i+1)}{2}+(i+1)(2d+k-2i-2)+\frac{i(i+1)}{2}=(i+1)(j+1-i).$$
$``\Leftarrow"$ The multiplication map $m_{\ell^{i-r}}:A_r\rightarrow A_i$, for every $r\in \left[0,i-1\right]$, has trivially maximal rank, so the first $i+1$ parts of $P_\ell$ contain $\frac{i(i+1)}{2}$ boxes, corresponding to basis elements of $A_r$, for every $r\in \left[0,i-1\right]$.
Now assume by way of contradiction that for some $i\in [0, d-1]$, $h_\ell^i(F)= 0$; then the  rank of the multiplication map $m_{\ell^{j-2i}}:A_i\rightarrow A_{j-i}$ is at most $i$, for simplicity we may assume it is exactly $i$. This implies that there are $j+1-2i$ more boxes in $i$ parts among $p_0,p_1,\dots ,p_i$, and there are at most $j-2i$ more boxes  in the remaining part among $p_0,p_1,\dots ,p_i$. On the other hand, similarly to the previous case,  there are at most $\frac{i(i+1)}{2}$ more boxes in the first $i+1$ parts of $P_\ell$, corresponding to the basis elements of $A_r$'s, for $r\in\left[j+1-i,j\right]$. Therefore, 
\begin{align*}
\sum^{i+1}_{j=1}p_j&<\frac{i(i+1)}{2}+(i)(2d+k-2i-2)+(2d+k-2i-3)+\frac{i(i+1)}{2}\\
&<(i+1)(j+1-i),
\end{align*}
contradicting our assumption. We have shown $\Leftarrow$.
\end{proof}\par

The following example illustrates Theorem \ref{sumOfPartitions}.
\begin{example}\label{sumOfPartitionex}
Let $A=R/\Ann(F)$ be an Artinian CI algebra with the Hilbert function {$H(A)=(1, 2, \ldots, 9,10^2,9, \ldots, 2, 1)$} where $d=10,k=2,j=19$. Let the partition $P_\ell=({19}^2,{15}^2,{10}^3,3^4)$ be the Jordan type for a linear form $\ell$, see the Ferrers diagram in Figure \ref{vanishHess}. Each box of the Ferrers diagram represents a basis element of $A_i$, for $i\in\left[0,19\right]$ and we label the boxes by the degree of the elements. The boxes with labels $5,6,13$ and $14$ are indicated in the Figure \ref{vanishHess}. Integers $0,1,2,\dots ,9$, written in the left of the rows indicate the order $i$ of the Hessian $h^i_{\ell}(F)$ for that row.

Using the equivalence in Equation \eqref{corrformula}, we can determine the set of vanishing Hessians for $A$ and $\ell$ from the partition $P_\ell$.

First, since $p_1=19$, Equation \eqref{corrformula} implies that $h^0_\ell(F)=0$. In fact, $p_1$ represents the power of $\ell$ which is zero: in other words, we have $\ell^{19}=0$ which implies that $m_{\ell^{19}}:A_0\rightarrow A_{19}$ is not an isomorphism and therefore $h^0_\ell(F)=0$.

We see that $h_\ell^6(F)\neq 0$ by showing that $m_{\ell^7}:A_6\rightarrow A_{13}$ has maximum rank, that is, $7$. By looking at the partition and boxes with labels $6$ and $13$ we see that they are all in the first $7$ rows of the partition. Moreover, every box with label less than $6$ has to be on the left of the boxes with label $6$, and every box with label greater than $13$ has to be to the right of the boxes with label $13$. Therefore the number of boxes in the first $7$ rows of $P_\ell$ is exactly $21+56+21=98=7(14)$ as in Equation \ref{corrformula} for $i=6$.\par

Now we look at the boxes labeled with $5$ and $14$ to determine if the map $m_{\ell^9}:A_5\rightarrow A_{14}$ has maximal rank. We see in the Ferrers diagram that one box with label $14$ is moved to the rows below the those labelled with $5$. This shows that  $\ell^9$ has rank $5$, that is, one less than the maximal rank, and therefore $h_\ell^5(F)=0$. The number of boxes in the first $6$ rows is equal to $88$ which is less than $90=6(15)$, as shown by Theorem \ref{sumOfPartitions}.

Similarly, we see that  $h_\ell^1(F)\neq 0$ and $h_\ell^3(F)\neq 0$ and that all the other Hessians are zero.
\end{example}

\begin{figure}
\begin{center}
\includegraphics[scale=1.025]{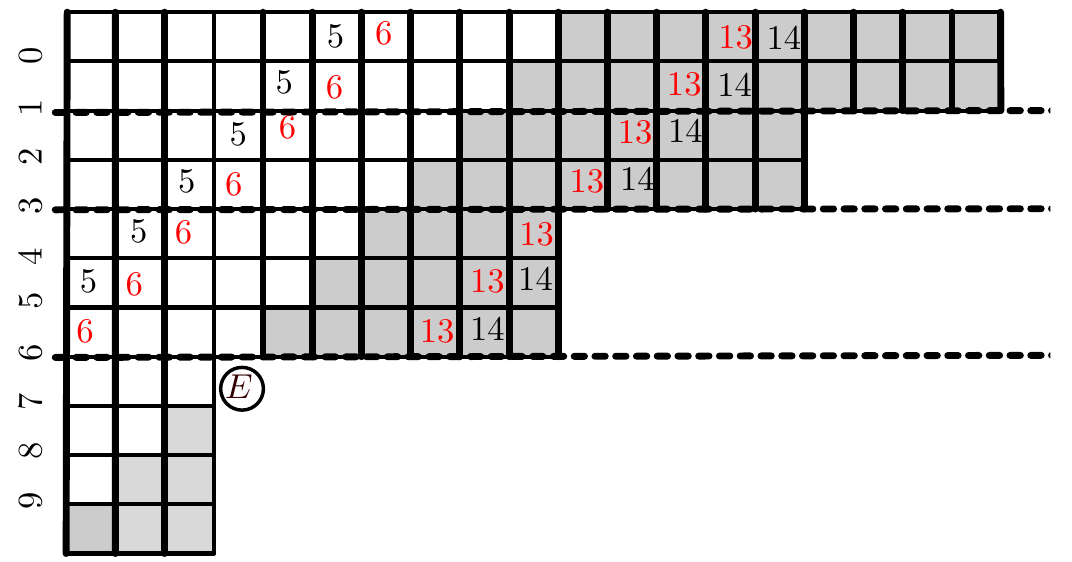}
\end{center}
\vspace{-.2 in}

\caption{Vanishing Hessians for CIJT partition $P=({19}^2,{15}^2,{10}^3,3^4)$ of diagonal lengths $T=(1, 2, \ldots, 10^2, \ldots, 2, 1)$. See Example \ref{sumOfPartitionex}.}\label{vanishHess}

\end{figure}

In Section \ref{JTCIsec}, we provided the list of possible complete intersection Jordan types having diagonal lengths $T=(1,2,\dots , d^k , \dots ,2,1)$, for $d\geq 2$ and $k\geq 1$. In Corollaries~\ref{goodpartcor} and ~\ref{goodpartcor2}, we specified all such partitions explicitly. Now, using Theorem \ref{sumOfPartitions}, we determine the set of Hessians which vanish for each possible CIJT partition. An ordered partition of zero is the empty partition.
\begin{theorem}[Hessians and partitions]\label{2Hessthm}
Assume that $T=\left( 1,2,3, \dots , d^k,\dots ,3,2,1\right)$, satisfies Equation \eqref{CITeq} for $d\geq 2$ and $k\geq 2$ $(k=1$, respectively), and assume that $\cha {\sf k}=0$. Then there is a 1-1 correspondence between the CIJT partitions $P_\ell$ of diagonal lengths $T$, and the $2^d$ (when $k>1$), or $2^{d-1}$ (when $k=1$) subsets of the active Hessians for $T$ that vanish at $\ell$ in $R_1$. \par
In particular, let $P$ be a partition of diagonal lengths $T$. The following are equivalent.
\begin{enumerate}[i.]
\item $P=P_{\ell,A}$ for a linear form $\ell\in R$ and an Artinian complete intersection algebra $A=R/\Ann F$, and there is an ordered partition $n=n_1+\cdots +n_c$ of an integer $n$ satisfying $0\leq n\leq d$ (or $0\leq n\leq d-1$, respectively) such that $h^{n_1+\cdots +n_i-1}_\ell(F)\neq 0$, for each $i\in\left[1,c\right]$, and the remaining Hessians are zero;
\item $P$ satisfies
\begin{equation}\label{JTpartition}
P=\big(p_1^{n_1}, \ldots, p_c^{n_c}, (d-n)^{d-n+k-1}\big),
\end{equation}
 where 
$p_i=k-1+2d-n_i-2\displaystyle(n_1+\cdots +n_{i-1})$, for $1\leq i \leq c$.
\end{enumerate}
\end{theorem}
\begin{proof}
First we observe that the number of subsets of active Hessians for a complete intersection algebra $A$ having the Hilbert function $T$ is $2^d$ when $k\ge 2$ and $2^{d-1}$ when $k=1$, which agrees with the number of complete intersection Jordan types having diagonal lengths $T$ from Corollary \ref{goodcountcor}. Thus, to prove the statement we show that for each partition $P$ in (ii) above, the set of Hessians which vanish is given in (i).

Suppose that for an integer $n\in [0,d]$ and for an ordered partition $n=n_1+\cdots +n_c$, we have a partition $P$ satisfying Equation \eqref{JTpartition}. By Corollary~\ref{goodpartcor} or Corollary ~\ref{goodpartcor2} the partition $P$ occurs as a CIJT partition. For every $i\in\left[1,c\right]$, we have that
\small
\begin{align*}
\sum^{n_1+\cdots +n_i}_{j=1}p_j =&\sum^{n_1}_{j=1}p_j+\sum^{n_1+n_2}_{j=n_1+1}p_j+\cdots +\sum^{n_1+\cdots n_i}_{j=n_1+\cdots n_{i-1}+1}p_j\\
=&n_1(2d+k-1-n_1)+\\
&n_2(2d+k-1-n_1-2n_1)+\\
\vdots &\qquad\qquad\qquad\qquad+\\
&n_i\left(2d+k-1-n_i-2(n_1+\cdots +n_{i-1})\right)\\
=& (n_1+\cdots +n_{i})(2d+k-1)-n_1^2-n_2(n_2+2n_1)-\cdots -n_i(n_i+2(n_1+\cdots n_{i-1}))\\
=&(n_1+\cdots +n_{i})(2d+k-1-(n_1+n_2+\cdots +n_i)).
\end{align*}
\normalsize
Using Theorem \ref{sumOfPartitions}, we conclude that $h^{n_1+\cdots n_i-1}_\ell(F)\neq 0$.
Now we show that for every integer $1\leq t\leq n_i-1$ and for every $i\in[1,c]$, we have $h^{n_1+\cdots +n_{i-1}-1+t}_\ell(F)=0$ (here we set $n_0:=0$). By Theorem \ref{sumOfPartitions}, we may write
\begin{align*}
\sum^{n_1+\cdots +n_{i-1}+t}_{j=1}p_j=&\sum^{n_1+\cdots +n_{i-1}}_{j=1}p_j+\sum^{n_1+\cdots +n_{i-1}+t}_{j=n_1+\cdots +n_{i-1}+1}p_j\\
=&(n_1+\cdots +n_{i-1})(2d+k-1-(n_1+n_2+\cdots +n_{i-1}))\\
+&t(2d+k-1-n_i-2(n_1+\cdots +n_{i-1}))\\
<&(n_1+\cdots +n_{i-1}+t)(2d+k-1-(n_1+\cdots +n_{i-1}+t));
\end{align*}
the last inequality holds since $1\leq t\leq n_i-1$. Thus $P$ satisfies (i). This completes the proof.
\end{proof}
\begin{remark}
Let $d,k\geq 2$ and set $T=\left(1,2,\dots ,d^k,\dots , 2,1\right)$. Consider a branch label $\mathfrak{b}=\big(\mathfrak{v}, 0, \mathfrak{h}_1, \ldots ,\mathfrak{h}_c\big),$ where $\mathfrak{v}$ is the (possibly empty) ordered interval $\{x\,|\, 0\leq x <d-e-1\}$, for some $e\in\left[-1, d-1\right]$, and
for an increasing sequence $a_0<a_1<\cdots <a_c=e$ and each $i\in\left[0,c\right]$, the interval $\mathfrak{h}_i$ is the ordered interval $\{x\,|\, d-a_i\leq x < d-a_{i-1}\}$, where $a_0:=-1$. Suppose that $P$ is the Jordan type partition with diagonal lengths $T$ and branch label $\mathfrak{b}$, for linear form $\ell\in R$  of an Artinian complete intersection algebra $A=R/\Ann F$. Theorem \ref{2Hessthm}, implies that for every $i\in \left[1,c\right]$, the Hessian $h^{a_i}_\ell\neq 0$ and the remaining Hessians are zero.\par In fact, each $n_i$, for $i\in \left[1,c\right]$ in the ordered partition of $0\leq n\leq d$ in Theorem \ref{2Hessthm} is equal to $a_i-a_{i-1}$. \par A similar correspondence holds for $k=1$.
\end{remark}

In the following example, we use Theorem \ref{2Hessthm} to show how to determine all Jordan type partitions possible for a linear form of an Artinian CI algebra having a certain given Hilbert function and having given sets of vanishing Hessians.

\begin{example}
Let  $T=\left( 1,2,3,3,2,1\right)$. The active Hessians for an Artinian CI algebra $A=R/\Ann F$ with the Hilbert function $T$ are $h^0, h^1$ and $h^2$. Thus, there are $2^3=8$ complete intersection Jordan types with diagonal lengths $T$. As discussed in Remark~\ref{wLpLoci}, such CIJT partitions have either three or four parts depending on whether the top Hessian vanishes or not. For example for $\ell$ and $F$ where we have $h^0_\ell(F)\neq 0$, $h^1_\ell(F)\neq 0$ and $h^2_\ell(F)\neq 0$, we have $n_1=n_2=n_3=1$, and therefore $P_\ell=(6,4,2)$, is the strong Lefschetz Jordan type. But for $\ell$ and $F$ where we have  $h^0_\ell(F)\neq 0$, $h^1_\ell(F) \neq  0$ and $h^2_\ell(F)= 0$, we have $n_1=n_2=1$, and the partition is $P=(6,4,1,1)$. A complete list of CIJT partitions with diagonal lengths $(1,2,3,3,2,1)$ and the ranks of the corresponding Hessians is included in Figure~\ref{fig9}. For each CIJT partition, the active Hessians that are zero are indicated in bold with $\ast$.
\end{example}

R. Gondim and G. Zappal\`{a} in \cite{GoZ} introduced mixed Hessians, which can be used to compute the ranks of multiplication maps by powers of linear forms.

\begin{definition}
Let $\mathcal{B}_k=\left(\alpha_1,\dots ,\alpha_r\right)$ be a ${\sf k}$-linear basis of $A_k$ and $\mathcal{B}_l=\left(\beta_1,\dots ,\beta_s\right)$ be a ${\sf k}$-linear basis of $A_l$, and assume $\cha {\sf k}=0$. The matrix 
$$
\hess^{(k,l)}(F):=\left[\alpha^{(k)}_i\beta^{(l)}_j\circ F\right]
$$
is the mixed Hessian matrix of $F$ of mixed order $(k,l)$ with respect to the bases $\mathcal{B}_k$ and $\mathcal{B}_l$. Denote by $h^{(k,l)}(F)$ the determinant $ h^{(k,l)}(F):=\det\left(\hess^{(k,l)}(F)\right)$, which we term the ``mixed Hessian'' of order $(k,l)$.
\end{definition}
We observe that the mixed Hessian of $F$ of order $(k,k)$, $\hess^{(k,k)}(F)$, coincides with the $k$-th Hessian of $F$, $\hess^k(F)$.
\par 
We notice that an immediate consequence of Theorem \ref{2Hessthm} is that for an Artinian complete intersection algebra with a given Hilbert function and a linear form for which a set of higher Hessians vanish there is exactly one CIJT partition. This implies -- as we will show -- that there is only one collection of ranks of higher and mixed Hessians for vanishing Hessians that occurs for a pair $(\ell,A)$ where $A$ is any Artinian CI algebra, but the Jordan type $P_\ell$ is fixed.

\begin{proposition}[Ranks of mixed and higher Hessians]\label{rkseq}
Assume that $A=R/\Ann F$ is an Artinian complete intersection $\sf k$-algebra of $\cha{\sf k}=0$ with the Hilbert function \linebreak$T=\left(1,2,\dots ,d^k,\dots ,2,1\right)$, for $d\geq 2$ and $k\geq 1$. Denote the socle degree of $A$ by $j=2d+k-3$.
Assume further that for a linear form $\ell\in R$, and a non-negative integer $m$, we have 
$$
h^m_{\ell}(F)=h^{m+1}_{\ell}(F)=\cdots =h^{m+n}_{\ell}(F)=0, \hspace*{0.1cm}\text{and if}\hspace*{0.2cm}m\neq 0, \hspace*{0.1cm}\text{then}\hspace*{0.2cm}\hspace*{0.2cm} h^{m-1}_{\ell}(F)\neq 0.
$$
Then
\begin{itemize}
\item[$(i)$]  If $m+n\leq d-2$ and $h^{m+n+1}_{\ell}(F)\neq 0$ (recall $h_\ell^{d-1}\ne 0$) then for every $i\in[0,n]$ we have 
\begin{align*}
\rk\hess^{m+i,s}_{\ell}(F)=\left\{
                \begin{array}{ll}
                     \max\{j+i-(n+s),m\} \hspace*{0.35cm} \text{if}\hspace*{0.2 cm} s\in \left[j-(m+n),j-(m+i)\right],\\
                    m+i+1\hspace*{3cm}\text{if}\hspace*{0.2 cm} s\in \left[d,j-(m+n+1)\right],\\
                \end{array}
              \right.           
\end{align*}
In particular
\begin{equation}\label{HessRkSeq1eq}
\rk\hess^{m+i}_{\ell}(F)=m, \rk\hess^{m+n-i}_{\ell}(F)=m+n-2i, \hspace*{0.2 cm} \text{for}\hspace*{0.15 cm} \text{every}\hspace*{0.2 cm} i\in\left[0,\big\lfloor\frac{n}{2}\big\rfloor\right].
\end{equation}

\item[$(ii)$] If $k\geq2$, and $m+n=d-1$, then for every $i\in\left[0,n+\big\lfloor \frac{k}{2}\big\rfloor-1\right]$ and every $s\in \left[d,j-(m+i)\right]$ we have 
\begin{equation*}
\rk\hess^{m+i,s}_{\ell}(F)=\max\{2m+n+i+1-s,m\}.
\end{equation*}
\end{itemize}
\end{proposition}
\begin{proof}
First we show $(i)$. Note that the assumption in $(i)$ implies by Theorem \ref{sumOfPartitions} that the Jordan type partition $P_\ell=\left(p_1,p_2,\dots, p_{d}\right)$ has exactly $d$ parts ($\ell$ is weak Lefschetz). By Theorem \ref{2Hessthm}, we have that 
\begin{equation}\label{equalparts}
p_{m+1}=p_{m+2}=\cdots =p_{m+n+2}=j-(2m+n),
\end{equation}

For each $i\in\left[0,n\right]$ and $s\in\left[j-(m+n),j-(m+i)\right]$, the rank of $\hess^{m+i,s}_\ell(F)$ is equal to the number of rows in the Ferrers diagram of $P_\ell$ between $p_1$ and $p_{m+i+1}$ containing all the boxes between the diagonals of degree $m+i$ and degree $s$. Suppose $m>0$ then $h^{m-1}_\ell(F)\neq 0$ implies that the first $m$ parts of $P_\ell$ contain all the boxes between the diagonals of degree $m-1$ and degree  $j-(m-1)$ and so they contain all the boxes between the diagonals of degree $m+i$ and degree  $s$. On the other hand, the number of rows of $P_\ell$ between  $p_{m+1}$ and $p_{m+i+1}$ containing all the boxes between the diagonals of degree $m+i$ and degree  $s$ is equal to the number of those rows with size larger than $s-(m+i)$. By Equation (\ref{equalparts}), it is equal to
$$ \max\{j-(2m+n)-(s-(m+i)),0\}=\max\{j+i-(m+n+s),0\}.$$ We conclude that for $i\in \left[0,n\right]$ and $s\in\left[j-(m+n),j-(m+i)\right]$,
$$\rk\hess^{m+i,s}_{\ell}(F)=m+\max\{j+i-(m+n+s),0\}.$$
If $m=0$ then the number of rows between $p_1$ to $p_{i+1}$ with size larger than $s-i$ is equal to $\max\{j-n-(s-i),0\}$ and so the $\rk\hess^{m+i,s}_{\ell}(F)=\max\{j+i-(n+s),0\}$.

Now the assumption $h^{m+n+1}_\ell(F) = h^{m+n+1,j-(m+n+1)}_\ell(F)\neq 0$ implies that \linebreak$ m_{\ell^{j-2(m+n+1)}}:A_{m+n+1}\rightarrow A_{j-(m+n+1)}$ has maximal rank. Notice that since $m+n+1\leq d-1$ , the multiplication map $ m_{\ell^{n+1-i}}:A_{m+i}\rightarrow A_{m+n+1}$ trivially has maximal rank for all $i\in \left[0,n\right]$. Therefore, $m_{\ell^{j-(n+i+1)}}:A_{m+i}\rightarrow A_{j-(m+n+1)}$, for $i\in \left[0,n\right]$ has maximal rank which means that $h^{m+i,j-(m+n+1)}\neq 0$. Consequently,  $h^{m+i,s}\neq 0$, for all $i\in \left[0,n\right]$ and $s\in\left[d,j-(m+n+1)\right]$, which means that $\rk\hess^{m+i,s}_{\ell}(F)=m+i+1$.

To show $(ii)$ we first notice that the assumption in $(ii)$ implies that for all $i\in\left[0,n+\big\lfloor \frac{k}{2}\big\rfloor-1\right]$, $h^{m+i}_\ell(F)=0$. Using Theorem \ref{2Hessthm} we get that 
$$
p_{m+1}=p_{m+2}=\dots =p_{d+k-1}=n+1.
$$
Similar to the proof of $(i)$, for $i\in\left[0,n+\big\lfloor \frac{k}{2}\big\rfloor-1\right]$ and $s\in \left[d,j-(d-r+i)\right]$ the rank of $\hess^{m+i,s}_\ell(F)$ is equal to $m$ plus the number of rows between $p_{m+1}$ and $p_{d+k-1}$ with size larger than $s-(m+i)$. In another words,
$$\rk\hess^{m+i,s}_{\ell}(F)=m+\max\{n+1-(s-(m+i)),0\}=\max\{2m+n+i+1-s,m\}.$$

\end{proof}

\begin{remark}[Uniqueness of ranks of Hessian matrices]\label{rankremark}
Lemma \ref{criterioncor}, Theorem \ref{2Hessthm}, and Proposition \ref{rkseq} determine the sets of possible ranks of higher and mixed Hessian matrices for a partition $P$ having CIJT: Lemma \ref{criterioncor} is the criterion for $P$ to have CIJT, Theorem  \ref{2Hessthm} determines which Hessians vanish for such $P$.
%, and Proposition  \ref{rkseq} together with Corollary~\ref{rkMixed} then completely determine the ranks of the higher and mixed Hessian matrices. 
Thus, there is no Artinian CI algebra $A$ such that the ranks of the active Hessian matrices are different from the ones given in Proposition \ref{rkseq}.
\end{remark}

\subsection{Lattice structure on the CI Jordan types, and dominance of partitions.}\label{latticesec}

Let $T=(1, \ldots, d^k, \ldots ,1)$ and $P$ be a CIJT partition of diagonal lengths $T$. Given a linear form $\ell\in R_1$, of Jordan type $P$ in an Artinian CI algebra $A=R/\Ann (F)$ we define  the set ${\mathcal H}_{P,\ell}={\mathcal H}_{P.\ell}(F)$ to be the set of all integers $i$ such that $h^i_{\ell}(F)\neq 0$.
\begin{definition}\label{dominancedef}
Let $P=(p_1,\ldots, p_t), p_1\ge \cdots \ge p_t$ and $Q=(q_1,q_2,\ldots,q_{t^\prime}), q_1\ge \cdots \ge q_{t^\prime}$ be partitions of the integer $n$. The dominance partial order is 
\begin{equation}\label{dominanceeq}
Q\le P \Leftrightarrow \sum_{j=0}^iq_j\leq \sum_{j=0}^ip_j \text { for all } i\le \min\{t,t'\}.
\end{equation}
\end{definition}
\begin{theorem}\label{posetprop}[Dominance and closure]
Let $T=(1, \ldots, d^k, \ldots ,1)$ satisfy Equation~\eqref{CIT2eq} and asuume that $P=P_\ell=(p_0, p_1, \ldots)$ and $Q=Q_{\ell'}=(q_0, q_1,\ldots)$ are CIJT partitions having diagonal lengths $T$. Then $Q\leq P_\ell$ in the dominance order if and only if ${\mathcal H}_{Q,{\ell'}}\subseteq {\mathcal H}_{P,\ell}$. 
\end{theorem}

\begin{proof}
First assume that $Q\leq P$ in the dominance order. If $i\not \in {\mathcal H}_P$, then by Theorem~\ref{sumOfPartitions}, $$\sum_{j=0}^ip_j<(i+1)(j+1-i).$$ Since $Q\leq P$, we also also have $\sum_{j=0}^iq_j<\sum_{j=0}^ip_j.$
Therefore $$\sum_{j=0}^iq_j<(i+1)(j+1-i).$$ Thus by Theorem \ref{sumOfPartitions}, $i\not \in {\mathcal H}_Q$. This shows that ${\mathcal H}_P\subseteq {\mathcal H}_Q$.
\bigskip

Conversely, to prove that ${\mathcal H}_P\subseteq {\mathcal H}_Q$ implies $Q\leq P$, it is enough to prove that if ${\mathcal H}_P\setminus {\mathcal H}_Q$ has only one element, say $\alpha$, then $Q\leq P$. We write the elements of ${\mathcal H}_Q$ in increasing order as $a_1<\cdots<a_c$, and we assume that $t\in[1,c+1]$ is such that $a_{t-1}<\alpha<a_t$ (here $a_{0}=-1$ and $a_{c+1}=d$). 

By Theorem \ref{sumOfPartitions} for each $a_i$, since $a_i$ is in both ${\mathcal H}_P$ and ${\mathcal H}_Q$, we get $$\displaystyle{\sum^{a_i}_{j=0}q_j} = \displaystyle{\sum^{a_i}_{j=0}p_j}=(n_i+1)(j+1-n_i).$$ Thus, in order to prove $P\leq Q$, it is enough to prove that for all $a_{t-1}<i<a_t$, 
$$\displaystyle{\sum^{i}_{j=a_t+1}q_j} \leq  \displaystyle{\sum^{i}_{j=a_t+1}p_j}.$$

\noindent{\bf Case 1.} Assume that $t\leq c$. Then $\alpha<a_c$ and therefore in this case for all $a_{t-1}<i<a_t$ we are in the horizontal part of $P$ and $Q$. In fact, the branch labels of $P$ and $Q$ are the same except for one of the horizontal intervals of $Q$ that is now divided into two subintervals for $P$ through the introduction of the new ``cut'', corresponding to the new added element $\alpha$. For simplicity, assume that the horizontal interval in $Q$ consists of $\{a, a+1, \ldots, a+h\}$ that are attached to the rows of $\Delta_d$ with lengths $u, (u-1), \ldots, (u-h)$. Thus the corresponding subpartition of $Q$ has the form $\big((a+u)^{h+1}\big)$.

On the other hand, the addition of the new element $\alpha$ to ${\mathcal H}_Q$ to obtain ${\mathcal H}_P$ is equivalent to breaking up the horizontal interval $[a,a+h]$ in the branch label of $Q$  into two horizontal subintervals, say $[a,a+\bar{h}]$ and $[a+\bar{h}+1, a+h]$ for an integer $\bar{h}$ such that $0\leq \bar{h}< h$. 
Then in $P$, branches of lengths $a+\bar{h}+1, \ldots, a+h$ are added to rows of lengths $u, \ldots, u-h+\bar{h}+1$ of $\Delta_d$, and branches of lengths $a, \ldots, a+\bar{h}$ are added to rows of lengths $u-h+\bar{h}, \ldots, u-h$ of $\Delta_d$.  The corresponding subpartition of $P$ has the form $\big((a+u+\bar{h}+1)^{h-\bar{h}}, (a+u-h+\bar{h})^{\bar{h}+1}\big)$. 

Since $\bar{h}\geq 0$, $a+u+\bar{h}\geq a+u$. Additionally, if $1 \leq i \leq \bar{h}+1$ then
\begin{small}
$$\begin{array}{ll}
(a+u+\bar{h}+1)(h-\bar{h})+(a+u-h+\bar{h})i&=(a+u)(h-\bar{h}+i)+(\bar{h}+1)(h-\bar{h})-(h-\bar{h})i\\ 
&=(a+u)(h-\bar{h}+i)+(h-\bar{h})(\bar{h}+1-i)\\ 
&>(a+u)(h-\bar{h}+i).
\end{array}$$
\end{small} 
Thus $\big((a+u)^{h+1}\big)\leq \big((a+u+\bar{h})^{h-\bar{h}}, (a+u+\bar{h}-1)^{\bar{h}}\big).$ This shows that in this case, $Q\leq P$, as desired.
\medskip

\noindent{\bf Case 2.} Now assume that $a_c<\alpha$. This in particular implies that $a_c< d-1$ and therefore the branch label of $Q$ has a vertical part, say of the form $[1,v]$ for a positive integer $v$. We note that $v=d-a_c-2$. The branch label of $P$ is obtained from the branch label of $Q$ by keeping all the horizontal parts of $Q$ and breaking up its vertical part into two subintervals, say $[1,\bar{v}-1]$ and $[\bar{v},v]$ for some $\bar{v}\in[1,v-1]$, where the first subinterval (which may be empty) will be the vertical part of the branch label of $P$ and the second subinterval will be added to the last horizontal part of the label as a new last interval.

For $0\leq i\leq a_c$, $p_i=q_i$. Thus, in order to show that $Q\leq P$, it is enough to show that the desired inequalities for the partial sums of $P$ and $Q$ hold beyond $a_c$.
\medskip

{\bf Case 2.1.} First assume that $k=1$. Then in order to compare $P$ and $Q$, it is enough to compare the subpartition $\big((v+1)^{v+1}\big)$ of $Q$ with the subpartition\par\noindent $\big((d-a_c-1+\bar{v})^{v-\bar{v}+1},(\bar{v})^{\bar{v}}\big)=\big((v+\bar{v}+1)^{v-\bar{v}+1},(\bar{v})^{\bar{v}}\big)$ of $P$.
\par
Since $\bar{v}\geq 1$, $ v+\bar{v}+1> v+1$. Additionally, if $1 \leq i \leq \bar{v}$ then 
$$\begin{array}{rl}
(v+\bar{v}+1)(v-\bar{v}+1)+\bar{v} \, i=&(v+1)(v-\bar{v}+1)+\bar{v}(v-\bar{v}+1)\\ &+(v+1)i-(v-\bar{v}+1)i\\ 
=&(v+1)(v-\bar{v}+1+i)+(\bar{v}-i)(v-\bar{v}+1)\\ 
 \geq& (v+1)(v-\bar{v}+1+i).
\end{array}$$
Thus $\big((v+1)^{v+1}\big)\leq \big((v+\bar{v}+1)^{v-\bar{v}+1},(\bar{v})^{\bar{v}}\big).$ This shows that in this case, $Q\leq P$, as desired.
\medskip

{\bf Case 2.2.} Now we assume that $k\geq 2$. In this case, we need to compare the subpartition $\big(v^{v+k-1}\big)$ of $Q$ with the following subpartition of $P$. $$\big((d-a_c-1+\bar{v}+k-2)^{v-\bar{v}+1},(\bar{v}-1)^{\bar{v}+k-1}\big)=\big((v+\bar{v}+k-1)^{v-\bar{v}+1},(\bar{v}-1)^{\bar{v}+k-1}\big)$$

We obviously have $ v+\bar{v}+k-1> v+1$. Additionally, if $1 \leq i \leq \bar{v}+k-1$ then 
$$\begin{array}{rl}
(v+\bar{v}+k-1)(v-\bar{v}+1)+(\bar{v}-1) i=&v(v-\bar{v}+1)+(\bar{v}+k-1)(v-\bar{v}+1)\\
&+v\, i-(v-\bar{v}+1)i\\ 
=&v(v-\bar{v}+1+i)+(\bar{v}+k-1-i)(v-\bar{v}+1)\\ 
 \geq& v(v-\bar{v}+1+i).
\end{array}$$
Thus $\big(v^{v+k-1}\big)\leq \big((v+\bar{v}+k-1)^{v-\bar{v}+1},(\bar{v}-1)^{\bar{v}+k-1}\big).$
This completes the proof of the Proposition.
\end{proof}

\begin{example}\label{posetex}
In the first row of Figure \ref{poset}, we start with a complete intersection Jordan type partition $Q=(17^2, 10^5,4, 1^2)$ having diagonal lengths $T=(1, \ldots, 8, 9^2,8, \ldots,1)$. By Theorem \ref{sumOfPartitions}, the set of non-zero Hessians associated with $Q$ is ${\mathcal H}_Q=\{1,6,7\}$. We then form ${\mathcal H}_P=\{1,3,6,7\}$ by adding an extra non-vanishing condition for the Hessians. This, as illustrated in the figure, leads to the CIJT partition $P=(17^2, 12^2, 8^3, 4, 1^2)$, which clearly dominates $Q$. In the second row of Figure \ref{poset} we illustrate a similar relation between the CIJT partitions $Q=(14^2,6^6)$ and $P=(14^2, 10^2, 4^4)$ of diagonal lengths $T=(1, \ldots, 7, 8, 7, \ldots, 1).$
\end{example}

\begin{figure}
\resizebox{.98 \textwidth}{!} 
{
    $\begin{array}{ll}

\includegraphics[scale=.3]{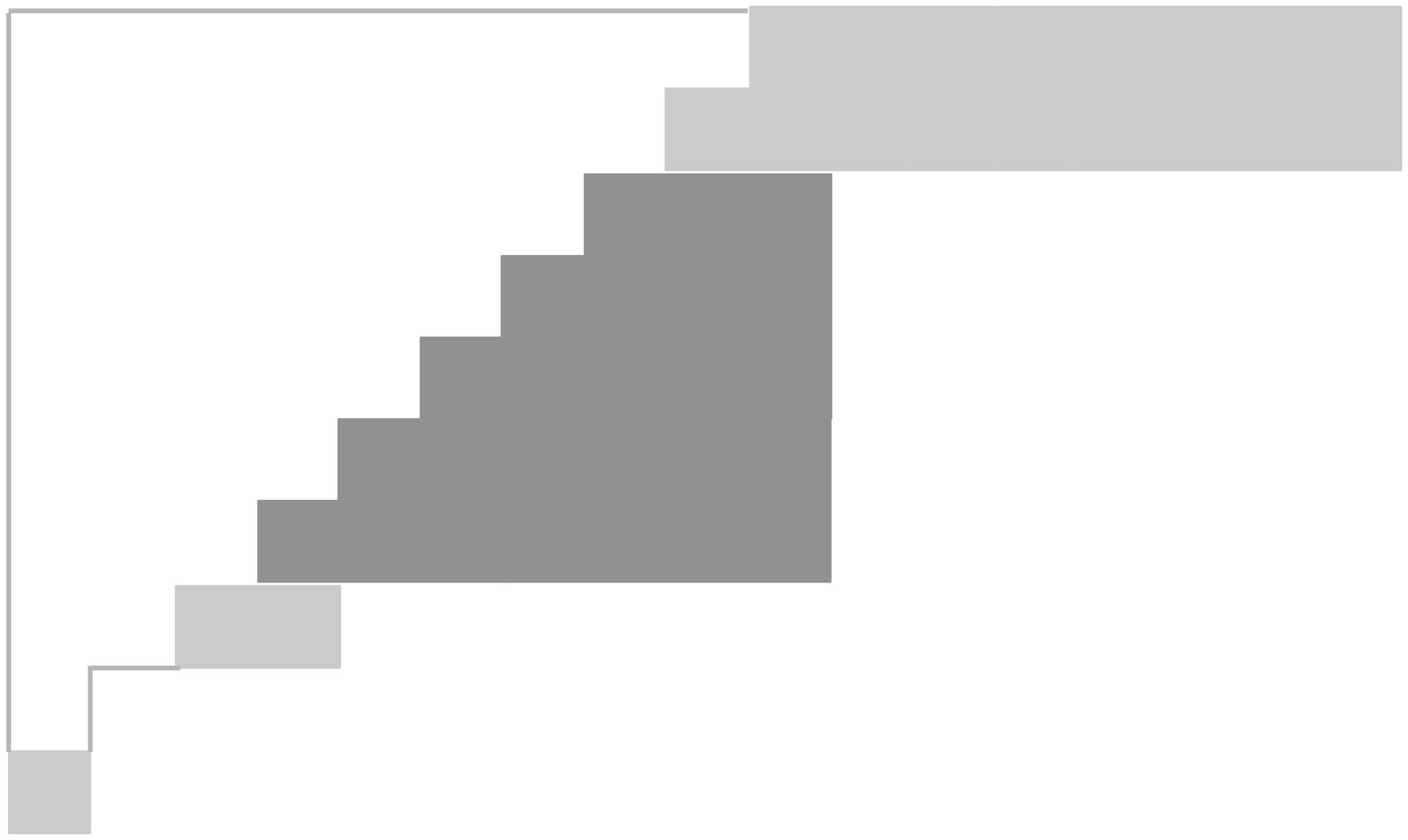}\vspace{-0.3in}&\includegraphics[scale=.3]{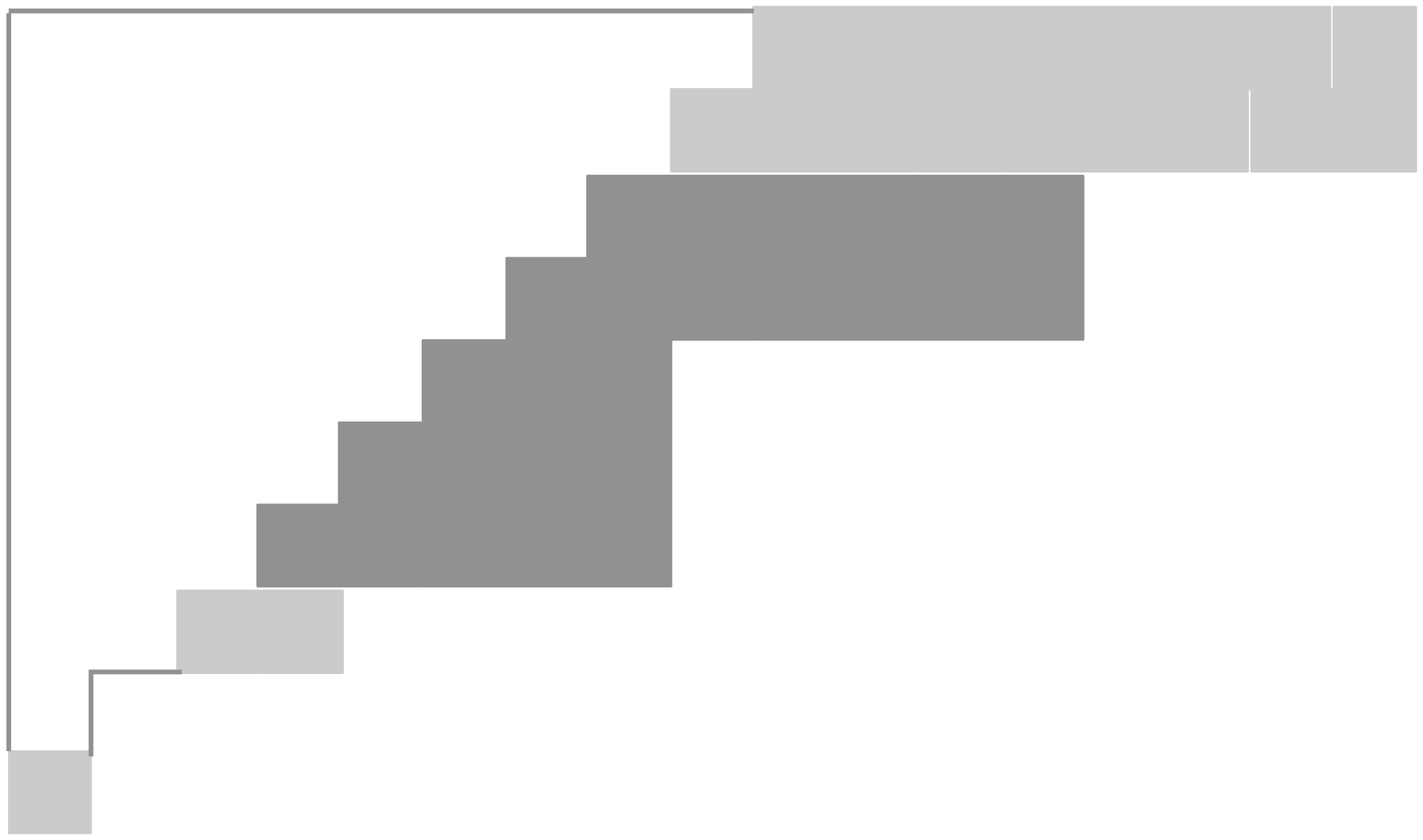}\vspace{-0.3in}\\ 
Q=(17^2, {\bf 10^5}, 4, 1^2)&P=(17^2, {\bf 12^2, 8^3}, 4, 1^2)\\ 
\mathfrak{b}_Q=\big(\{1\}, 0, \{8,9\}, {\bf\{3,4,5,6,7\}}, \{2\}\big)&\mathfrak{b}_P=\big(\{1\}, 0, \{8,9\}, {\bf \{6,7\},\{3,4,5\}}, \{2\}\big)\\ 
\mathcal{H}_Q=\{1,6,7\}&\mathcal{H}_P=\{1, {\bf 3}, 6,7\} \\
\hline
\hline
\includegraphics[scale=.3]{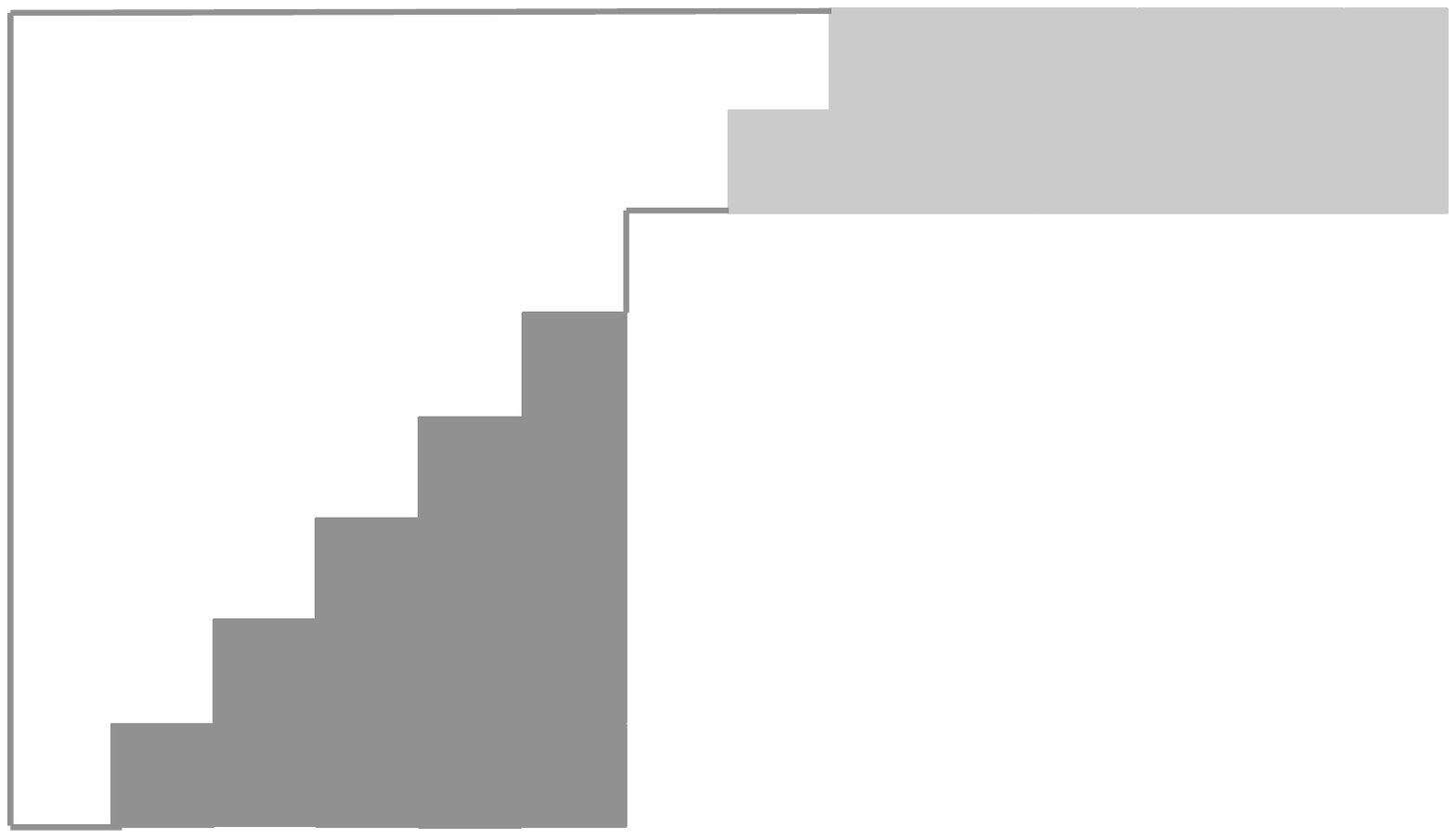}\vspace{-0.4in}&\includegraphics[scale=.3]{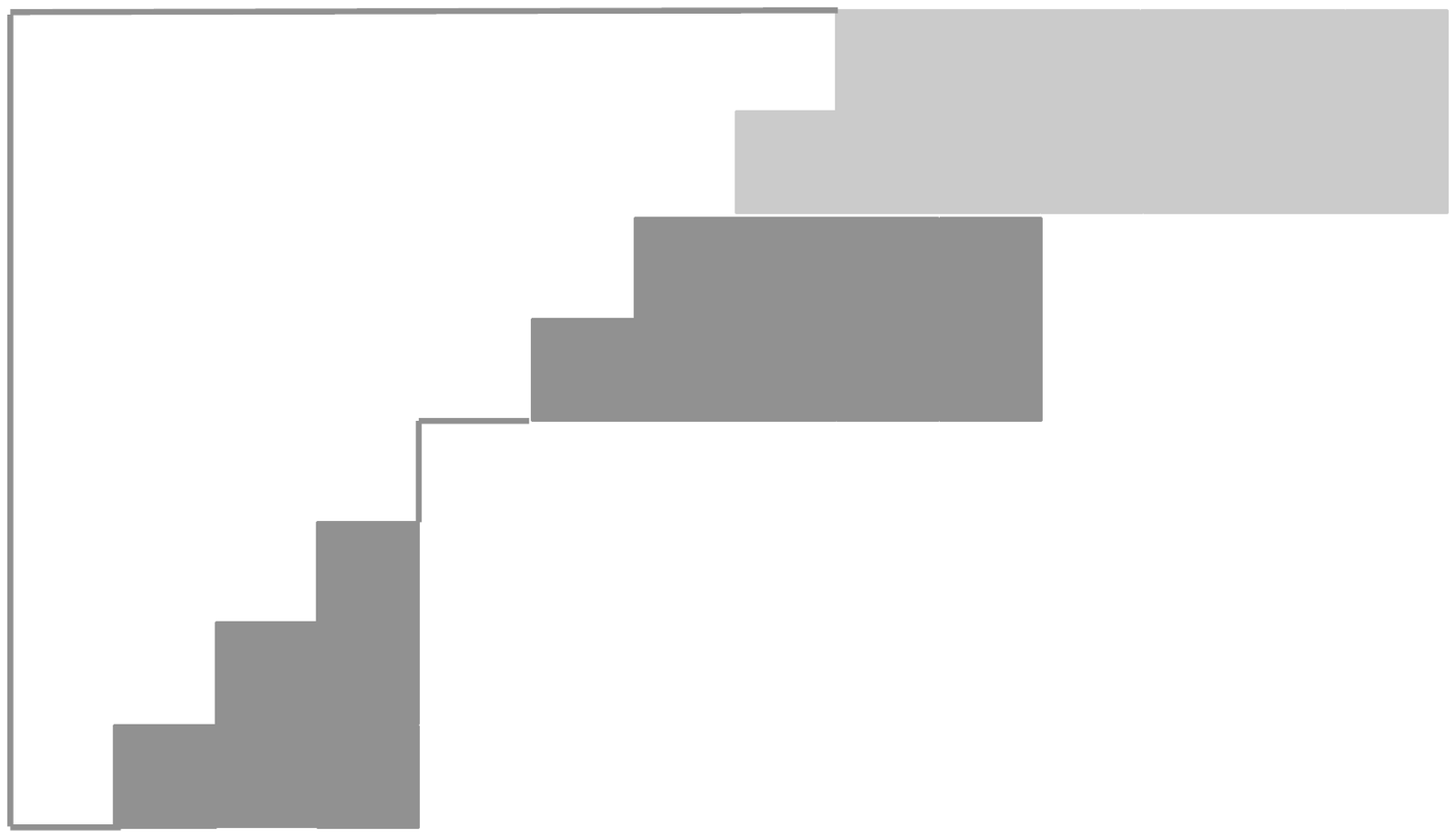}\vspace{-0.4in}\\ 
{Q=(14^2, {\bf 6^6})}&P=(14^2, {\bf 10^2, 4^4})\\ 
\mathfrak{b}_Q=\big( 0,{\bf \{1, 2, 3,4,5\}}, 0, \{6,7\}\big)&\mathfrak{b}_P=\big( 0,{\bf \{1, 2, 3\}}, 0, \{6,7\}, {\bf \{4,5\}}\big)\\ 
\mathcal{H}_Q=\{1,7\}&\mathcal{H}_P=\{1, {\bf 3}, 7\}\\

\end{array}$
}
\caption{An illustration for Example \ref{posetex} showing the effect of adding an extra non-vanishing condition for Hessians on the partition and on its branch label.}\label{poset}
\end{figure}

\par\noindent
\subsubsection{Geometric consequence.}
Recall that for a Hilbert function $T$ that occurs for an Artinian quotient of $R={\sf k}[x,y]$ the projective variety $G_T$ parametrizes graded algebra quotients $A=R/I$ of $R$ having Hilbert function $T$: it is smooth of known dimension \cite{I0,IY}. For $T$ satisfying Equation~\eqref{CIT2eq} this dimension is $(1+2(d-1))$ when $k\ge 2$ and $2(d-1)$ when $k=1$.
 Recall that, given a partition $P$ of $n$ having diagonal lengths $T$, we denote by $\mathbb V(E_{P,\ell})$ the affine cell of $G_T$ parametrizing algebras $A=R/I$ such that $I$ has initial monomial ideal $E_{P,\ell}$ in the direction $\ell$ (Definition \ref{basiccelldef}). For simplicity we may write  $\mathbb V(E_{P})$ for $\mathbb V(E_{P,\ell})$.
\begin{corollary}\label{geomcor}[Proper intersection of CIJT cells of $G_T$]  Let $T$ satisfy Equation \eqref{CIT2eq}, and let  $P,Q$ be CIJT partitions of diagonal lengths $T$, and $P\cap Q$ their intersection in  the poset of partitions.
Then, fixing $\ell$, we have
\begin{align}
\overline{\mathbb V(E_P)}\cap \overline {\mathbb V(E_Q)}&=\overline{ \mathbb V(E_{P\cap Q})} \text { and }\notag\\
\overline{\mathbb V(E_P)}&= \bigcup_ {P'\le P} \mathbb V(E_{P'}).
\end{align}
Furthermore, the codimension of the cell $\mathbb V(E_{P,\ell})$ in $G_T$ is the number of Hessians that vanish at $p_\ell$.\footnote{Later in Corollary \ref{goodhookcor}C we will see that this codimension is the difference between the number of difference-one hooks in the conjugate $T^\vee$ and the number of difference-one hooks of $P$.} The cells $ \overline {\mathbb V(E_P)} $ and $ \overline {V(E_Q)}$ intersect properly.
\end{corollary}
\begin{proof} The second part of Lemma  \ref{criterioncor} shows that we may replace ``$P_\ell=P$'' by ``the initial ideal of $I$ in the $(y,\ell)$ direction is $E_P$'': that is, the decomposition of $G_T$ into affine cells corresponding to the initial ideals $E_P$ is the same as that according to the Jordan types $P=P_\ell$.  The decomposition of $G_T$ into affine cells is a result of \cite{I0,IY}, which for our CI Jordan types we show in Lemma \ref{criterioncor}. The rest follows from Theorem~\ref{posetprop}.
\end{proof}
\begin{remark}\label{contrastremark}
The conclusion of Corollary \ref{geomcor} is in contrast to an example of J.~Yam\'{e}ogo where the intersection is not dimensionally proper for the two non-CIJT partitions $P=(5,2,1,1)$ and $Q=(4,2,1,1,1)$ having diagonal lengths $T=(1,2,3,2,1)$ (see \cite[Example 4.1]{Y2} and \cite[Example 1.24]{IY2}).
\end{remark}
We will denote by $CI_T$ the open dense subvariety of $G_T$ parametrizing complete intersections. We now show that the cells of $CI_T$ satisfy a frontier property, the closure of a cell is the union of cells.
\begin{theorem}\label{mainHessthm}[Closure of $\mathbb V(E_P)$] Assume that $T$ satisfies Equation\eqref{CIT2eq}. Then the complete intersection Jordan types (CIJT's) having diagonal lengths $T$ correspond one to one with the subsets of the active Hessians, according to the maps given in Theorem~\ref{2Hessthm}.  The Zariski closure in $CI_T$ of the locus $\mathbb V(E_P)$ of Artinian algebras whose Jordan type is a CIJT partition $P\in \mathcal P(T)$, is the union $\overline{\mathbb V(E_P)}=\bigcup^\prime_{P^\prime\ge P} \mathbb V(E_{P^\prime})$ where the union is over complete intersection partitions $P^\prime$.
\end{theorem}
\begin{proof} Immediate from Theorem \ref{2Hessthm} and Corollary \ref{geomcor}.
\end{proof}\par
\begin{remark}\label{frontierrem} The analogous frontier property is not shared by cells corresponding to non-CIJT partitions, even when $T$ satisfies Equation \eqref{CIT2eq}: J. Yam\'{e}ogo showed  that the cell $C$ corresponding to the non-CIJT partition $P(C)= (5,2,1,1)$ of diagonal lengths $T=(1,2,3,2,1)$ has closure that is not the union of cells (\cite{Y1}, see also \cite[Example~3.28, and \S 3F(E)]{IY}).
\end{remark}

\subsection{Pattern of CIJT partitions.}\label{patternsec}

\begin{figure}
\begin{center}
%\resizebox{.97 \textwidth}{!} 
%{
   $\begin{array}{c||c}
P \,\text { with 3 parts}& \iota(P) \,\text { with $k+2$ parts }\\
\hline\hline
(4+k,2+k,k)&(4+k,2+k,1^k)\\
\hline
(4+k,1+k,1+k)&(4+k,2^{k+1})\\
\hline\hline
(3+k,3+k,k)&(3+k,3+k,1^k)\\
\hline
(k+2,k+2,k+2)&(3^{k+2})\\
\hline
\end{array}$
%}
\caption{CIJT partitions for $T=(1,2,3^k,2,1), k\ge 1$. When $k=1$ the two columns are the same.
}\label{fig10.5}
\end{center}
\end{figure}
\begin{figure}
\begin{center}
%\resizebox{.97 \textwidth}{!} 
%{
   $\begin{array}{c||c}
P \,\text { with 4 parts}& \iota(P) \,\text { with $k+3$ parts }\\
\hline\hline

(6+k,4+k,2+k,k)&(6+k,4+k,2+k,1^k)\\
\hline
(6+k,4+k,1+k,1+k)&(6+k,4+k,2^{k+1})\\
\hline
(6+k,3+k,3+k,k)&(6+k,3+k,3+k,1^k)\\
\hline
(6+k,2+k,2+k,2+k)&(6+k,3^{k+2})\\
\hline\hline
(5+k,5+k,2+k,k)&(5+k,5+k,2+k,1^k)\\
\hline
(5+k,5+k,1+k,1+k)&(5+k,5+k,2^{k+1})\\
\hline
(4+k, 4+k,4+k, k)&(4+k, 4+k,4+k,1^k)\\
\hline
(3+k, 3+k,3+k, 3+k)&(4^{k+3})\\
\hline
\end{array}$
%}
\caption{CIJT partitions for $T=(1,2,3,4^k,3,2,1), k\ge 1$.  When $k=1$ the two columns are the same.
The map from $P$ to $\iota(P)$ is by conjugating the rectangular subpartition determined by the smallest part (Theorem \ref{tablethm})}\label{fig11}
\end{center}
\end{figure}
\begin{figure}
\begin{center}
%\resizebox{.97 \textwidth}{!} 
%{
   $\begin{array}{c||c}
P \,\text { with $d=5$ parts}& \iota (P) \,\text { with $k+4$ parts }\\
\hline\hline

(8+k,6+k,4+k,2+k,k)&(8+k,6+k,4+k,2+k,1^k)\\
\hline
(8+k,6+k,4+k,1+k,1+k)&(8+k,6+k,4+k,2^{k+1})\\
\hline
(8+k,6+k,3+k,3+k,k)&(8+k,6+k,3+k,3+k,1^k)\\
\hline
(8+k,6+k,2+k,2+k,2+k)&(8+k,6+k,3^{k+2})\\
\hline
(8+k,5+k,5+k,2+k,k)&(8+k,5+k,5+k,2+k,1^k)\\
\hline
(8+k,5+k,5+k,1+k,1+k)&(8+k,5+k,5+k,1+k,2^{k+1})\\
\hline
(8+k,4+k,4+k,4+k,k)&(8+k,4+k,4+k,4+k,1^k)\\
\hline
(8+k,3+k,3+k,3+k,3+k)&(8+k,4^{k+3})\\
\hline\hline
(7+k,7+k,4+k,2+k,k)&(7+k,7+k,4+k,2+k,1^k)\\
\hline
(7+k,7+k,4+k,1+k,1+k)&(7+k,7+k,4+k,2^{k+1})\\
\hline
(7+k,7+k,3+k,3+k,k)&(7+k,7+k,3+k,3+k,1^k)\\
\hline
(7+k,7+k,2+k,2+k,2+k)&(7+k,7+k,3^{k+2})\\
\hline
(6+k,6+k,6+k,2+k,k)&(6+k,6+k,6+k,2+k,1^k)\\
\hline
(6+k,6+k,6+k,1+k,1+k)&(6+k,6+k,6+k,2^{k+1})\\
\hline
(5+k,5+k,5+k,5+k,k)&(5+k,5+k,5+k,5+k,1^k)\\
\hline
(4+k, 4+k,4+k, 4+k,4+k)&(5^{k+4})\\
\hline
\end{array}$
%}
\caption{CIJT partitions for $T=(1,2,3,4,5^k,4,3,2,1), k\ge 1$.  When $k=1$ $\iota(P)=P$}\label{fig12}
\end{center}
\end{figure}
Recall that we denote by $\mathcal P(T)$ the set of partitions having diagonal lengths $T$. We first state a symmetry condition satisfied by the strings for Jordan decompositions of graded Gorenstein Artinian algebras, then we state and prove a result about the tables of CIJT partitions having a given diagonal lengths $T$.

\subsubsection{Symmetry of Jordan strings for an Artinian Gorenstein algebra.}
T. Harima and J. Watanabe \cite{HW} defined ``central simple modules'' of a Gorenstein Artinian algebra $A$.  From their paper, or from an alternative approach using symmetric decomposition of the algebra with respect to powers of a principle ideal $(\ell)$ of a linear element, one can show that the strings of the Jordan type of $A$ with respect to $\ell$ satisfy the symmetry condition below in Lemma \ref{symlem}. This statement can be found also in B.~Costa and R. Gondim
\cite[Lemma~4.6]{CGo} who use the result of T. Harima and J.~Watanabe (see also \cite[\S 2.6]{IMM}.)\par Recall from Definition~\ref{basiccelldef} that $A$ can be written as the direct sum of simple ${\sf k}[x]$ modules for the ${\sf k}[x]$ action on $A$ defined by the multiplication map $m_\ell$: we term these the \emph{strings} of $m_\ell$. We denote by $V_{i,s}$ the vector space span of generators for the simple modules of length $s$ whose generators lie in degree $i$, and by $W_{i,s}$ the space $\ell^{s-1}V_{i,s}$. Let $j$ be the socle degree of $A$.  We have, as $A$ is a finite-length module over the principle ideal domain ${\sf k}[x]$, 
\begin{lemma} Let $A$ be a graded Artinian algebra, let $\ell\in A_1$. Then, as a module over ${\sf k}[x]$ with $x$ acting as $m_\ell$, we have
\begin{equation}\label{Adecompeq} A=\oplus _{i,s}{\sf k}[x]/(x^s)\cdot  V_{i,s}.
\end{equation}
\end{lemma}
The following symmetry result is a consequence of \cite{HW}; see \cite[Lemma 4.6]{CGo} and \cite[Theorem 2.34]{IMM} for proofs. In \cite[Definition 2.23]{IMM} and in \cite[Definition 4.7]{CGo} the property is that $\ell$ has \emph{symmetric Jordan degree type}.
\begin{lemma}[Symmetry of strings for a Gorenstein Artinian algebra $A$]\label{symlem} Assume that $A=S/I$ is a standard-graded Gorenstein Artinian algebra of socle degree $j$ that is a ${\sf k}[x]$-module under the map multiplication by $\ell\in R_1$. Let $\phi: A_j\to {\sf k}$ and denote by $\langle \cdot,\cdot\rangle_\phi$ the exact pairing $A\times A\to \sf k$, $\langle a,b\rangle_\phi=\phi(ab)$. Then 
\begin{enumerate}[i.]
\item for each pair $(i,s)$ we have an $(i',s) $ such that $W_{i',s}=V_{i,s}^\vee,$ the dual of $V(i,s)$ under $\langle \cdot,\cdot\rangle_\phi$.
\item For each pair $(i,s)$ we have $\dim_{\sf k} V_{i,s}=\dim_{\sf k} V_{j+1-s-i,s}$.
\end{enumerate}
\end{lemma}
We will say that a partition $P\in \mathcal P(T)$, for $T$ satisfying Equation \eqref{CITeq} is \emph{symmetric} if it has symmetric Jordan degree type in the sense above: the parts can be arranged symmetrically about $j/2$ so as to sum to $T$.
\subsubsection{List of height two CIJT partitions.}
  Assume that $T$ satisfies equation \ref{CIT2eq} with the maximum value $d$ occuring $k$ times. We give in Figures \ref{fig10.5}, \ref{fig11}, and \ref{fig12} tables of CIJT partitions for $T$ of heights $d=3,4,5$, respectively.  To tabulate these we have used Theorem \ref{simplegoodprop}: if the sequence $T$ is a possible Hilbert function for a CI partition, so satisfies equation \eqref{CIT2eq} then the CIJT partitions of diagonal lengths $T$  are exactly those in $\mathcal P(T)$ (having diagonal lengths $T$) that have either $d$ or $d+k-1$ parts.  
 Let $P$ be a CIJT partition of diagonal lengths $T$ and having $d$ parts. Then we denote by $\iota (P)$ the partition obtained by flipping the smallest width rectangle in the Ferrers diagram of $P$. For example, for $P=(6,3,3)$ in $\mathcal P(T), T=(1,2,3,3,2,1)$ we have $\iota(P)=(6,2,2,2)$ (see Figure \ref{fig10.5}).

  \begin{theorem}\label{tablethm} Assume $T$ satisfies Equation \eqref{CIT2eq}, of height $d$ that occurs  $k$ times. Then the set of CIJT partitions of diagonal lengths $T$ satisfy:

\begin{enumerate}[a.]
\item Let $P$ be a CIJT partition having diagonal lengths $T$ and $d$ parts. Then the rectangular block of smallest parts has the form $(a+k)^{a+1}$ for some $a\in [0,d-1]$, and $\iota(P)$ is a CIJT partition having diagonal lengths $T$ and $d+k-1$ parts.  The map $P\to \iota (P)$ is 1-1 onto from the set of CIJT partitions of diagonal lengths $T$ and having $d$ parts, to those having $d+k-1$ parts.\par
Suppose $k\ge 2$ and that $P$ is a CIJT partition with $d$ parts.  Then the vanishing Hessians for $\iota(P)$ are $\mathcal H_{\iota(P)}=\mathcal H_P\cup h^{d-1}$.
\item There are $2^{d-1}$ CIJT partitions $P$ with $d$ parts and having diagonal  lengths $T$, and as well $2^{d-1}$ CIJT partitions having $d+k-1$ parts and diagonal lengths $T$.
\item Let $d\ge 2$. The number of CIJT partitions $P$ of diagonal lengths $T$, having $d$ parts and smallest part 
$a+k$ for $0\le a\le d-2$ is $2^{d-2-a}$; there is a single partition whose smallest part is $d+k-1$, with multiplicity $d$.
(See Figure \ref{smallpartfig}).
\item  For $P$ a CIJT partition of diagonal lengths $T$ and with $d$ parts $P=(p_1,p_2,\ldots, p_d)$ the $i$-th string, of length $p_i$ begins in degree $i-1$: that is the space $V_{i,s}$ of Lemma~\ref{symlem} is $0$ unless $s=p_{i+1}$, and $\dim_{\sf k}V_{i-1,p_i}=1$. \par
The analogous statement is true for a CIJT partition $P$ with $d+k-1$ parts.
\end{enumerate}
\end{theorem}
\begin{proof}[Proof of (a)] When $k=1$ the map $\iota$ is the identity map, and, by Corollary \ref{goodpartcor2} the smallest block of $P$ is $(d-n)^{d-n}$: letting $a=d-n-1$, we have the smallest 
block of $P$ is the square block $(a+1)^{a+1}$ as stated in part (a). Here the integer $a\in [0, d-1]$.\par
When $k>1$ and $n=d$ in Corollary \ref{goodpartcor}, then $P=(p_1^{n_1},\ldots,p_c^{n_c})$ in Equation \eqref{goodpartcorpart} where $p_c= k-1+n_c$: letting $a=n_c-1$ we have 
$$P=(P',(a+k)^{a+1}) \text { where } P'=({p_1}^{n_1},\ldots,{p_{c-1}}^{n_{c-1}}).$$ 
Letting  $n' =\sum_{i=1}^{c-1}n_i$ we have $a=d-n'-1$.  Then, from Equation \eqref{goodpartcorpart} again, the last rectangle is $(d-n')^{d-n'+k-1}$, we have\par
$$\iota (P)= (P', (a+1)^{a+k}).$$
This shows the formula of part (a), and evidently from Corollary \ref{goodpartcor} the map $\iota$ is 1-1 onto. Note that $P'$, comprised of the largest parts of $P$, is also a CIJT partition. \vskip 0.2cm\noindent
{\it Proof of (b)}. There are by Theorem 2 of the introduction (Theorem \ref{2Hessthm}) when $k>1$ exactly $2^d$ partitions of diagonal lengths $T$, and when $k=1$ there are $2^{d-1}$. The statement follows.
\vskip 0.2cm\noindent
{\it Proof of (c)}. We show this by complete induction on $d$. If $d=1$ there is a single partition $P=(d)$. 
When $d=2$, then the two CIJT partitions with 2 parts are $(k+2,k)$ (here $a=0$) and $(k+1)^2$ (here $a=1$), 
satisfying  the formula of (b) (see Figure \ref{fig3a}). Now suppose that the count in part (b) is known for
heights of $T$ less than $d$, let $T$ have height $d$ with multiplicity $k$.  We know that the last part 
$p_c^{n_c}$ of $P$ has the form $(a+k)^{a+1}$. Removing this rectangle, we have a partition $P'$ of diagonal lengths $T'$ obtained by removing the top $(a+1)$ rows from the bar graph of $T$. Here $T'$ has height $h=d-(a+1)$, and the multiplicity of $h$ in $T'$ is $k+2(a+1)$.  The number of such partitions $P'$ is $2^{h-1}$ (half the total number of CIJT partitions, by (b) above). Thus there are at most 
$$\sum_{a=0}^{d-1}2^{d-(a+2)}=2^{d-1}$$
partitions $P$ having $d$ parts of diagonal lengths $T$, if we count all potential $P'$. By (b) this is the number of partitions with $d$ parts and diagonal lengths $T$, so all potential $P'$ occur. This proves (c).\vskip 0.2cm\noindent
{\it Proof of (d)}. We show the first claim for CIJT partitions with $d$ parts, also by complete induction on $d$, parallel to the proof of (c). When $d=1$ there is a single such CIJT partition of diagonal lengths $T$, and when $d=2$ the two partitions $(k+2,k)$ and $(k+1,k+1)$, each satisfying the statement about strings.
Suppose the assertion is correct for heights less than $d$. Adding to the strings of $P'$ the strings corresponding to $(a+k)^{a+1}$, can be done in a symmetric way in the sense of Lemma \ref{symlem} only by following the prescription of (d).  Similarly, adding to the $P'$  strings a block $(a+1)^{a+k}$ in a symmetric way follows the prescription.
\end{proof}
 
\begin{figure}
\begin{center}
   $\begin{array}{c|cccccc}
d&k&k+1&k+2&k+3&k+4&k+5\\
\hline
3&2&1&1&&&\\
\hline
4&4&2&1&1&&\\
\hline
5&8&4&2&1&1&\\
\hline
6&16&8&4&2&1&1\\
\hline
\end{array}$
%}
\caption{Number of CIJT partitions in $\mathcal P(T)$ with $d$ parts, and given smallest part. See Theorem~\ref{tablethm} (c.)
}\label{smallpartfig}
\end{center}
\end{figure}

\section{CI Jordan types and their hook codes.}\label{CIJTsolsec}
Hook codes for partitions $P$ of arbitrary diagonal lengths $T$ were introduced in \cite{IY}; they naturally give the
dimension of cells $\mathbb V(E_P)$ of the variety $G_T$ parametrizing Artinian algebras of Hilbert function $H(A)=T$.
We first in Section \ref{hookcodesec} define the hook codes for partitions of arbitrary diagonal lengths $T$ satisfying Equation \eqref{HFeq}.
We then in Section~\ref{CIJThookcodesec} restrict to Hilbert functions satisfying  Equation \eqref{CIT2eq} and determine the possible Jordan types. In Section \ref{vanishHess&hookcodesec} we connect the hook codes for CIJT partitions with the vanishing of Hessians. 
\subsection{Hooks and the affine cells $\mathbb V(E_P)$ of $G_T$.}\label{hookcodesec}
We first describe the difference-one hooks of a partition $P$ - related to standard bases for graded ideals in $R={\sf k}[x,y]$. This is relevant for us as the variety $G_T$ parametrizing graded Artinian quotients of $R$ having Hilbert function $H(A)=T$ has for each linear form $\ell\in A_1$ a decomposition into affine cells  $\mathbb V(E_{Q,\ell})$ where $Q$ runs through the partitions of diagonal lengths $T$. Recall from Section \ref{cell1sec} that the cell ${\mathbb V}(E_{P,\ell})$ for the pair $(y,\ell)$ is the set of Artinian algebras whose generic linear form $\ell$ has Jordan type $P$, it is also determined by $A=R/I$ where the ideal $I$ has initial forms $E_Q$ in the $\ell$ direction.
\subsubsection{Difference-one hooks and the cell ${\mathbb V}(E_P).$}\label{diff1hooksec} For simplicity we take $\ell=x$ in describing initial form and monomials.\par 
Recall from Definition \ref{basiccelldef} that, given a partition $P=(p_1,p_2,\ldots, p_t)$ of $n=\sum p_i$ where $p_1\ge p_2\ge\cdots \ge p_t$, of diagonal lengths $T$ we 
 let $C_P$ be the set of $n$  monomials that fill the Ferrers diagram of $P$ as follows: for $i\in [1,t]$ the, $i-$th row counting from the the top is filled by the monomials $y^{i-1},y^{i-1}x,\ldots ,y^{i-1}x^{p_i-1}$. We let $E_P$ be the complementary set of monomials to $C(P)$ and denote by $(E_P)$ the ideal they generate.
\begin{definition}\label{hookdef}
A \emph{hook} of a partition $P$ is a subset of $C_P$ consisting of a corner monomial $c$, an \emph{arm} 
 $(c,xc,\ldots ,\nu=x^{u-1}c)$ and a \emph{leg} $(c,yc,\ldots , \mu= y^{v-1}c)$, such that $x\nu\in E_P$ and $y\mu\in E_P$ 
 (Figure \ref{hook fig}). The \emph{arm length} is $u$ and the \emph{leg length} is $v$; the hook has arm-leg \emph{difference} 
 $u-v$. We term the monomial $\nu\in C_P$ the \emph{hand}, and the monomial $\mu\in C_P$ the \emph{foot} of the hook.
 \end{definition}
\begin{figure}[!ht]
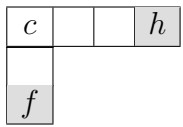

\begin{center}
$\begin{array}{|c|c|c|c|}
\hline
c&\phantom{c}&\phantom{c}&\cellcolor{gray1}h\\
\hline
\cline{1-1}\\
\cline{1-1}
\cellcolor{gray1}f\\
\cline{1-1}
\end{array}$
\caption{Difference-one hook with hand $h$, foot $f$, corner $c$.}\label{hook fig}
\end{center}
\end{figure}

\begin{example}\label{431ex} Let $P=(4,4,1)$. The hook with corner $x$ in the Ferrers diagram $C_P$  has arm length 3, foot length 2, hand $x^3$, foot $yx$, so has (arm $-$ leg)  difference one (Figure \ref{hook1fig}). Here $T(P)=(1,2,3,2,1),\, \Delta(P)=\Delta _3$ and the degree-$3$-diagonal of $C_P$ has the two spaces corresponding to the monomials $y^2x
$ and $ y^3$.\par
\end{example}

\begin{figure}[!ht]
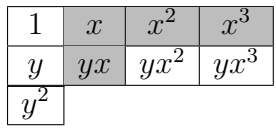

\begin{center}
$\begin{array}{|c|c|c|c|}
\hline
1&\cellcolor{gray2}x&\cellcolor{gray2}x^2&\cellcolor{gray2}x^3\\
\hline
y&\cellcolor{gray2}yx&yx^2&yx^3\\
\hline
\cline{1-1}
y^2\\
\cline{1-1}
\end{array}$
	\caption{Difference-one hook in the Ferrers diagram of partition $(4,4,1)$.}\label{hook1fig}
\end{center}
\end{figure}\noindent
For the following result, see \cite[\S 3-B,Theorem 3.12, and \S 3-F]{IY}. Recall that we denote by $\mathcal P(T)$ the set of all partitions of $n=|T|$ having diagonal lengths $T$. 
\begin{theorem}\label{celldimthminapp}
The cell $\mathbb V(E_{P,\ell})$ is an affine space of dimension the total number of difference-one hooks in $C_P$.\par
Fix $\ell\in R_1$. The variety $\G_T$ parametrizing all graded quotients $A=R/I$ of Hilbert function $T$ is a projective variety with a finite decomposition into affine cells,
\begin{equation}
\G_T=\bigcup_{P\in \mathcal P(T)}\mathbb V(E_{P,\ell}).
\end{equation}

\end{theorem}\noindent

\subsection{The hook codes for partitions having diagonal lengths $T$.}\label{CIJThookcodesec}
We first define hook codes for partitions $P$ of diagonal lengths $T$ satisfying Equation~\eqref{CIT2eq} (that can occur for a graded complete intersection) and  specify our notation for them. We then determine all the hook codes that can occur for partitions of such diagonal lengths $T$ (Proposition \ref{HCprop}), and those that can occur for CIJT partitions (Corollary \ref{goodhookcor}).\vskip 0.2cm\noindent%\subsubsection{Notation for the hook code.}
{\bf Notations for the hook code}.  Our notation $\mathfrak h(P)$ is based on the branch label $\mathfrak b$ for $P$ of Definition \ref{branchdef}, which specifies the lengths of the branches. Thus, our notation $\mathfrak h$ is different than that of \cite{IY2}, which we will denote here by $\mathfrak H(P)$.
\begin{definition}[Hook code]\label{newhookcode} In this section, working with branch labels we replace a $0$ entry by $E$ to indicate it is an omitted attachment point. The entries of a branch label $\mathfrak b$ are a permutation of $\{E,1,\ldots, d\}$. Two such $E$'s occur if the height of $T$ occurs exactly once $(k=1)$, otherwise there is only one $E$. We write the hook code $\mathfrak h$ as $\mathfrak b$ subscripted: for the entry $i_a$ with $i>0$ the subscript $a$ is the number of difference-one hooks of $P$ having as hand the endpoint of the branch of length $i+k-2$ in degree $d+i+k-3$. It is an integer between $0$ and $2$, except that for $i=1$ the highest possible subscript value is $1$. An entry $E$ does not have a subscript, as $E$ is not a hand of a hook. \par
The traditional hook code of \cite[Definition 3.26]{IY}  for $T$ satisfying \eqref{CIT2eq}, is $\mathfrak H(P)=(\mathfrak H(P)_{d},\ldots ,\mathfrak H(P)_j)$ where $\mathfrak H(P)_i\in [0,2]$ is the number of difference-one hooks having the unique possible hand in degree $i$. (When $k\ge 2$ then we have also $\mathfrak H(P)_{d+k-2}\in [0,1]$.)

\end{definition}\noindent
{\bf Examples of hook code.}\label{hookcodeexs}
Consider $T=(1,2,3,3,2,1)$, $P=(6,4,2)$ and $P'=(5,5,1,1)$. Corners of difference-one hooks are indicated in Figure \ref{hookcodefig} by $c$. Here $2_4$ in the hook code $\mathfrak H(P)=(1_3,2_4,2_5)$ indicates that $P$ has  $2$  hooks with a degree $4$ hand.\par
\begin{lemma} Let $P$ be the partition $T^\vee$ for $T$ a CI Hilbert function satisfying Equation \eqref{CIT2eq}, and denote by $\mathfrak H(P)=(a_d,\ldots, a_k)$ its hook code. Then each sequence $a'= (a'_d,a'_{d+1},\ldots ,a'_k)$ satisfying $0\le a'_i\le a_i$ occurs
as the traditional hook code $\mathfrak H({P_{a'}})$ of difference-one hooks for a unique partition $P_{a' }$ having diagonal lengths $T$. The correspondence between partitions and hook codes is 1-1 and takes the conjugate partition $P_{a'}^\vee$ to the complementary hook code $\mathfrak H^\vee({a'})=a-a' $ with respect to $\mathfrak H_P=a$.
\end{lemma}
Thus, the partition $(4,2,2,2)=(5,5,1,1)^\vee$ of diagonal lengths $T=(1,2,3,3,2,1)$ has hook code
$\mathfrak H=(1,0,1)=(1,2,2)-(0,2,1)$ (the hook code for $H^\vee=(6,4,2)$ minus that for $(5,5,1,1)$).  \par
The complement of a CIJT partition $P$ in $T^\vee$ is often not a CIJT partition.

\begin{figure}
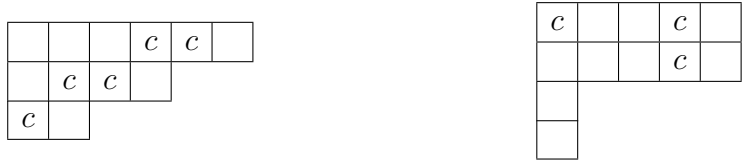

\begin{center}
$\begin{array}{lcr}
\begin{array}{|c|c|c|c|c|c|}
\hline
&&&c&c&\textcolor{white}{c}\\
\hline
&c&c&\\
\cline{1-4}
c&\\
\cline{1-2}

\end{array}

&\textcolor{white}{create some space}&

\begin{array}{|c|c|c|c|c|}
\hline
c&\textcolor{white}{c}&\textcolor{white}{c}&c&\textcolor{white}{c}\\
\hline
&&&c&\\
\hline
\\
\cline{1-1}
\\
\cline{1-1}
\end{array}
\end{array}$

\caption{Hook codes $(1_3,2_4,2_5)$ for $P=(6,4,2)$ and $(0_3,2_4,1_5)$ for $P=(5,5,1,1).$ Grading is by the degree of the hand monomial of the hook.}\label{hookcodefig}
\end{center}
\end{figure}

\begin{example}\label{hook1ex} I. We consider partitions having diagonal lengths $T=(1,2,1), (1,2,2,1)$ or $(1,2,3,2,1)$ in Figure \ref{fig2a}: we give the partition along with traditional hook code $\mathfrak H(P)$ where a subscripted integer $a_i$ means there are $a$ hooks (between zero and 2) with hand in degree $i$. We next give the ranks of the Hessian matrices, followed by Y/N according to whether the partition could correspond to a symmetric decomposition (Lemma~\ref{symlem})\footnote{ At one time we wondered if the possibility of symmetric decomposition of the partition could correspond to whether it is CIJT. These tables show that the answer is ``No''.} of $H$, then Y/N for a CIJT partition, and finally we give the (new) hook code -- the branch label subscripted by the new hook lengths, which we term $\mathfrak h$.\vskip 0.2cm
II. We consider the partitions having diagonal lengths $T=(1,2,2,2,1)$, where $d=2,k=3$. First, for the partition $P=(5,3)$ we have $\mathfrak h=(E,2_2,1_1)$, the maximum values possible. The traditional hook code for $P$ is  $\mathfrak H(P)=(1_3,2_4)$ (see Figure \ref{fig3a}): the boxes are $1\times 1$ in degree three and $1\times 2$ in degree four. Any other traditional hook code for $T$ is a pair of subpartitions of these boxes, so here corresponds to a pair of integers $(a_3,b_4)$ with $0\le a_3\le 1, 0\le b_4\le 2$. There are six such pairs. We consider the rest.
 For $P=(4,4)$, we have $\mathfrak h=(E,1_1,2_1)$ and $\mathfrak H(P)=(1_3,1_4)$; for $P=(5,1^4), \mathfrak h=(1_0,E,2_2)$ and $\mathfrak H(P)=(0_3,2_4)$; and for $P=(2,2,2,2)$, $\mathfrak{h} = (1_0,2_1,E)$ and $\mathfrak{H}(P)=(0_3,1_4)$. \par
%III.  For the partition $P=(19^2,11^3,5^3,3^8)$ of Figure \ref{labeldefnillust} we have for the subscripted hook code
%$$\mathfrak h=\big((6_0,7_1,8_1),(1_0,2_1),E,(9_2,10_1),(3_2,4_1,5_1)\big).$$
\end{example}

\begin{figure}
\begin{center}
$T=(1,2,1):\begin{array}{c||c|cc|c|c|}
P&\mathfrak H(P)&\rk\,\Hess^0&Y,N&CIJT&Label\, \mathfrak h\\
\hline\hline
(3,1)&(2_2)&1&Y&Y&(E,E,1_2)\\
\hline
(2,2)&(1_2)&{\bf 0\ast}&Y&Y&(E,1_1,E)\\
\hline
(2,1,1)&(0_2)&-&N&N&(1_0,E,E)\\
\hline
\end{array}$\vskip 0.2cm
\vskip 0.2cm
$T=(1,2,2,1):\begin{array}{c||c|ccc|c|c|}
P&\mathfrak H(P)&\rk\,\Hess^0& \rk\,\Hess^1&Y,N&CIJT&Label\, \mathfrak h\\
\hline\hline
(4,2)&(1_2,2_3)&1&2&Y&Y&(E,2_2,1_1)\\
\hline
(4,1,1)&(0_2,2_3)&1&{\bf 1\ast}&Y&Y&(1_0,E,2_2)\\
\hline
(3,3)&(1_2,1_3)&{\bf 0\ast}&2&Y&Y&(E,1_1,2_1)\\
\hline
\hline
(2,2,2)&(0_2,1_3)&{\bf 0\ast}&{\bf 1\ast}&Y&Y&(1_0,2_1,E)\\
\hline
(3,1,1,1)&(1_2,0_3)&-&-&N&N&(2_0,E,1_1)\\
\hline
(2,2,1,1)&(0_2,0_3)&-&-&Y&N&(2_0,1_0,E)\\
\hline
\end{array}$
\vskip 0.2cm
\vskip 0.4cm

\resizebox{.9 \textwidth}{!} 
{
$T=(1,2,3,2,1):\begin{array}{c||c|ccc|c|c}
P&\mathfrak H(P)&\rk\,\Hess^0& \rk\,\Hess^1&Y,N&CIJT&Label\, \mathfrak h\\
\hline\hline
(5,3,1)&(2_3,2_4)&1&2&Y&Y&(E,E,2_2,1_2)\\
\hline\hline
(4,4,1)&(2_3,1_4)&{\bf 0\ast}&2&Y&Y&(E,E,1_2,2_1)\\
\hline
(5,2,2)&(1_3,2_4)&1&{\bf 1\ast}&Y&Y&(E,1_1,E,2_2)\\
\hline
\hline
(4,2,1^3)&(2_3,0_4)&-&-&N&N&(2_0,E,E,1_2)\\
\hline
(3,3,3)&(1_3,1_4)&{\bf 0\ast}&{\bf 1\ast}&Y&Y&(E,1_1,2_1,E)\\
\hline
(5,2,1,1)&(0_3,2_4)&-&-&N&N&(1_0,E,E,2_2)\\
\hline\hline
(3,3,1^3)&(1_3,0_4)&-&-&Y&N&(2_0,E,E,1_1)\\
\hline
(3,2^3)&(0_3,1_4)&-&-&N&N&(1_0,2_1,E,E)\\
\hline\hline
(3,2,2,1^2)&(0_3,0_4)&-&-&N&N&(1_0,2_0,E,E)\\
\hline
\end{array}$}\vskip 0.2cm
\caption{Jordan type tables. See Example \ref{hook1ex}.}\label{fig2a}
\end{center}
\end{figure}
\vskip 0.2cm
The proof we give of the following Proposition depends on Lemma \ref{partlabel}.
\begin{proposition}[Hook code and branch label]\label{HCprop} A. Let $\mathfrak b$ be the branch label for a partition $P$ of diagonal lengths $T$ of Equation \eqref{CIT2eq} for which the multiplicity $k$ of the height $d$ is at least two. The hook code $\mathfrak h$ is as follows.
\begin{enumerate}[i.]
\item For the increasing subsequences of $\mathfrak b$ before $E$ (so vertical branches), the first subscript is $0$ (hooks), the subsequent subscripts are $1$ (hook).
\item For increasing subsequences of $\mathfrak b$ after (above) $E$ the first subscript is the maximum possible (so $2$ hooks unless the entry of $\mathfrak b$ is $i=1$, in which case the maximum is $1$ hook). The subsequent subscripts are $1$ (hook).
\end{enumerate}
B. Let $\mathfrak b$ be the branch label for a partition $P$ of diagonal lengths $T$ of Equation \eqref{CIT2eq} for which the multiplicity $k$ of the height $d$ is exactly one. There are two $E$'s in the branch label (zero-length branches). Then (i). above applies to increasing sequences of $\mathfrak b$ before the lower $E$ (so vertical branches), (ii) above applies to increasing sequences of $\mathfrak b$ after the higher $E$ (so horizontal branches), and the branches between the two $E$'s are a sequence $(1,\ldots, g)$ from higher to lower (the only hooks there are coming from the square portion of $P$ cut out by horizontals and verticals through the $E$'s.\footnote{Put more simply (i) applies to all increasing sequences below the higher $E$, and (ii) to all increasing sequences above the higher $E$.}
\end{proposition}
\begin{proof}[Proof of (A)] By induction on $d$.  For $d=2$, $T=(1,2^k,1)$ the six codes (see Figure \ref{fig2a} for $T=(1,2,2,1)$ and Figure \ref{fig3a} for $T=(1,2^k,1)$) are
\begin{equation}\label{h1221eqn}
\begin{array}{c||c|c|c|c|c|c|}
P&(4,2)&(4,1,1)&(3,3)&(2,2,2)&(3,1,1,1)&(2,2,1,1)\\
{\mathfrak h}&(E,2_2,1_1)&(1_0,E,2_2)&(E,1_1,2_1)&(1_0,2_1,E)&(2_0,E,1_1)&(2_0,1_0,E)
\end{array}
\end{equation}
and they satisfy the conditions (i),(ii) of the Proposition.\vskip 0.2cm
Assume that the Proposition is true for some $d=b-1$ where $ b\ge 3$: we will show it for $d=b$. For simplicity we assume $k=2$.  Let $P$ have diagonal lengths $T$ with $d=b$ and $\mathfrak b_P=\mathfrak b$, let $\mathfrak b'$ be $\mathfrak b$ with the entry $b$ removed, and let $P'=P_{\mathfrak b'}$ of diagonal lengths $T'=(1,2,\ldots, (b-1)^2,\ldots, 2,1)$ obtained by glueing the branch lengths $\mathfrak b'$ to the basic triangle $\Delta_{b-1}$. We need to add a branch of length $b$, to $\Delta_b$, and a branch of length $b$ to $\mathfrak b'$ to form $P, \mathfrak b$, respectively. We follow the style of the argument after Equation~\eqref{good2eq} in the proof of Lemma \ref{criterionlem} -- however, here the branch label is that of  Lemma \ref{partlabel}, which is more general.  There are three cases.
\begin{enumerate}[a.]
\item Assume first that the branch is added in the vertical section of $\mathfrak b'$. Then it is the largest element in an increasing interval $(a,\ldots,b)$ in the vertical section.  Vertical branches of length $k$ can only affect the hook length for horizontal branches of length $k+1$, so adding $b$ vertically changes no hook number of a horizontal (after $E$) element of $\mathfrak b$. Also, the endpoint monomial $\nu$ of a row of the vertical section of $\mathfrak b'$ has no
hook count change from the addition of $b$ -- which is too long if $b$ occurs before $\mu$ and does not affect the number of difference-one hooks if added after $\mu$. If the branch $b$ is isolated (the interval $(a,b)$ is just $(b)$) then there are evidently zero difference-one hooks and the entry of $\mathfrak h$ is $b_0$.  If $b$ is immediately preceded by $b-1$ (so $a<b$) then there is a new difference-one hook whose hand $\nu$ is the foot monomial of the branch $b$ and whose foot is  $\nu:x$, the foot monomial of the branch $b-1$. 
\item Assume that the branch $b-1$ is horizontal, that is $\mathfrak b'=(S'',E,S')$ with $b-1\in S'$, and the branch $b$ is added adjacent to $b-1$. If $b$ is added just above $b-1$,  then the endpoint $\nu$ of the branch $b$ has two difference-one hooks, the first
with foot  $\nu:x$ and the second with foot $\mu:y$ just above the next lower generator $\mu$ of $E_P$ (next inside corner). 
If $b$ is added below the branch $b-1$, then by Equation \eqref{BCITTeqn}, $\mathfrak b=(\ldots, E, [a,b],\ldots)$: the interval $[a,b]$ is the first (top) interval of the horizontal part. Inspection shows that the hook count for branches $\{a,a+1,\ldots,b-1\}$ is unchanged from $\mathfrak b^\prime$ to $\mathfrak b$. No other difference-one hooks from $P_{\mathfrak b'}$ are affected.  The  new branch $b$ will have only one hook, with hand $\mu$ and foot $\mu:x$.
\item Suppose that $b$ is added at the top, with hand $\mu$, and not just above $b-1$. First, if also $b-1$ is vertical, then there are two difference-one hooks for the top branch, one with foot $\mu:x$ and one with foot the foot monomial $\nu$ of the vertical $b-1$ labelled branch. Two hooks is the maximum possible, by theory
(see the Appendix and \cite[Theorem 1.17]{IY2}.)\footnote{This maximum of two is also straightforward to verify directly from the construction of $\mathfrak b$ and a combinatorial argument.} No other hook counts are affected when adding the longest branch horizontally at top. \par
 If instead $b-1$ is a horizontal branch, and the next branch after $b$ is $a$, then 
 $\mathfrak b=(\ldots, E, b,a,a+1,\ldots ,b,\ldots)$ as by Lemma \ref{partlabel} the next interval of $\mathfrak b$ is $[a,b-1]$; inspection shows that besides $\mu:x$ as before, there is a single further hook with foot the hand of the branch labelled $(b-1)$.  No other hook count is affected as $b$ has been added to the top.
\end{enumerate} 
This completes the induction step and the proof of (A) of the Proposition. \par
The proof of (B) is entirely similar, one makes use of the fact that between the two $E$'s, the hook lengths are $(0,1,\ldots, s)$ with no gaps, and all other branches are longer.  Those below the lower $E$ are vertical, so, since they are greater than $s$ there are no difference-one hooks with hand in the diagonal between the two $E$'s, and foot in the portion of $P$ below the lower $E$. So all hooks with hands in the middle portion come from the square cut out by the two $E$'s, yielding the statement that they are $(0,1,1,\ldots,1)$.
\end{proof}\par

\begin{figure}[h]
\begin{center}
$\begin{array}{c||c|c|c|c|c|c|}
P&\mathfrak H(P)&\rk\,\Hess^0&rk \Hess^1&Y,N&CIJT&Label\, \mathfrak h\\
\hline\hline
(k+2,k)&\left(1_k,2_{k+1}\right)&1&2&Y&Y&\left(E,2_2,1_1\right)\\
\hline
(k+2,1^{k})&\left(0_k,2_{k+1}\right)&1&{\bf 1\ast}&Y&Y&\left(1_0,E,2_2\right)\\
\hline
(k+1,k+1)&\left(1_k,1_{k+1}\right)&{\bf 0\ast}&2&Y&Y&\left(E,1_1,2_1\right)\\
\hline
\hline
(2^{k+1})&\left(0_k,1_{k+1}\right)&{\bf 0\ast}&{\bf 0\ast}&Y&Y&\left(1_0,2_1,E\right)\\
\hline
(k+1,1^{k+1})&\left(1_k,0_{k+1}\right)&-&-&N&N&\left(2_0,E,1_1\right)\\
\hline
(2^{k},1,1), k \text { even}&\left(0_k,0_{k+1}\right)&-&-& Y &N&\left(2_0,1_0,E\right)\\
\hline
(2^{k},1,1), k \text { odd}&\left(0_k,0_{k+1}\right)&-&-& N &N&\left(2_0,1_0,E\right)\\
\hline
\end{array}$
\caption{Jordan types, and ranks of Hessian matrices, $T=(1,2^k,1)$, for $k\geq 2$. See Example \ref{hook1ex} II. }\label{fig3a}
\end{center}
\end{figure}
Applying Proposition \ref{HCprop} we determine the hook codes of CIJT partitions having diagonal lengths $T$.
\begin{corollary}[Hooks and branch labels of CIJT partitions]\label{goodhookcor}
A.  Let $\mathfrak b$ be the branch label for a CIJT partition $P$ of diagonal lengths $T$ of Equation \eqref{CIT2eq} where the multiplicity $k$ of the height $d$ is at least two. The hook code $\mathfrak h$ is as follows.
\begin{enumerate}[i.]
\item If there is a vertical portion of $P$, then $\mathfrak h=(1_0,2_1,\ldots ,a_1,E, \ldots)$: here $\mathfrak b$ begins with an interval $[1,a]$: the first subscript is $0$ (hooks), the subsequent vertical subscripts are $1$.
\item For increasing subsequences of $\mathfrak b$ after (above) $E$ the first subscript is the maximum possible (so $2$ hooks unless the entry of $\mathfrak b$ is $i=1$, in which case the maximum is $1$). The subsequent subscripts are $1$ (hook).
\end{enumerate}\par
B. Let $\mathfrak b$ be the branch label for a CIJT partition $P$ of diagonal lengths $T$ of Equation \eqref{CIT2eq} where the multiplicity $k$ of the height $d$ is one. Then $\mathfrak B=(E,1_0,2_1,\ldots,g_1,E,\ldots)$: here the portion above the second $E$ follows (ii) above.\par
C. The dimension of the cell $\mathbb V(E_P)$ is the number of difference-one hooks of $P$.
\end{corollary}
\begin{proof} Parts (A) and (B) are immediate from Theorem \ref{CIpartsthm} and Proposition \ref{HCprop}. Part (C) follows from Theorem  \ref{celldimthminapp}.
\end{proof}

\subsection{Vanishing of Hessians and hook codes.}\label{vanishHess&hookcodesec}
Proposition \ref{HCprop} and Corollary \ref{goodhookcor} determine the hook code corresponding to the branch label $\mathfrak b$ of a partition $P$ having diagonal lengths $T=\left(1,2,\dots ,d^k,\dots ,2,1\right)$, for $k\geq 1$. On the other hand, Theorem \ref{2Hessthm} determines the vanishing and non-vanishing Hessians of a linear form of a CI algebra having Jordan type $P$. Conversely, we may also read the information of vanishing and non-vanishing Hessians from the hook code of a partition. For a CI algebra $A$ with the Hilbert function $T$ and a linear form $\ell$, there are $d$ active Hessians $h^i(F)$ when $k>1$, or $ d-1$, respectively, when $k=1$. These each correspond to a multiplication map $m_{\ell^{j-2i}}:A_{i}\rightarrow A_{j-i}$, for $i\in \left[0,d-1\right]$ (or $i\in\left[0,d-2\right]$ when $k=1 $); their vanishing and non-vanishing completely determine the corresponding Jordan type partition $P_\ell$.\par 
In the branch label $\mathfrak{b}$ of $P_\ell$ there are $d$ difference-one hooks, one with hand in degree $j-i$, for each $i$. For a general linear form $\ell$ with the strong Lefschetz partition $T^\vee$, all the Hessians are non-zero, and the corresponding branch label with the hook indices is $\mathfrak{h}_{SL}=\left(E,d_2,{d-1}_2,\dots ,1_1\right)$ by Corollary \ref{goodhookcor}.  The traditional hook code is $\mathfrak {H}_{SL}=(1_{d+k-2},2_{d+k-1},\ldots, 2_j)$ if $k>2$ and $\mathfrak {H}_{SL}=(2_{d+k-1},\ldots, 2_j)$ if $k=1$. That is, for $P=T^\vee$ the number of difference-one hooks with hand in degree $j-i$ is $2$ for $i\le d-2$, and, if $k>1$ is $1$ for $i=d-1$. Any  CIJT partition $P$, where the number of hooks with hands in degree $j-i$ is less than the corresponding one of $T^\vee$ is a partition $P_\ell$ for a CI algebra $A$ where $h_{\ell}^i=0$. We show this next using Theorem \ref{2Hessthm}.

We will denote the traditional hook code $\mathfrak H(P_\ell)$ for a CIJT partition $P_\ell$ by $\mathfrak{H_\ell}$; and the number of hooks with hands in degree $j-i$ by $({\mathfrak{H}_\ell})_i$, for each $i\in\left[0,d-1\right]$.

\begin{proposition}\label{HookandHess}
Let $T=\left(1,2,\dots ,d^k,\dots ,2,1\right)$, for $k\geq 1$. Assume that $P_{\ell}$ is the complete intersection Jordan type partition of a linear form $\ell$ having diagonal lengths $T$ and that the traditional hook code is $\mathfrak{H}_{\ell}$. We have the following,
\begin{itemize}
\item[$(i)$] If $k\geq 2$, then 
\[ h^{d-i}_{\ell}=0\Leftrightarrow(\mathfrak{H}_{\ell})_i < \left\{ \begin{array}{ll}
        1& \mbox{if $i=1$, }\\
      2 & \mbox{if $i\in\left[2,d\right]$}
       .\end{array} \right. \] 
 \item[$(ii)$] If $k=1$, then for each $i\in\left[1,d-1\right]$, we have that
\[ h^{d-1-i}_{\ell}=0\Leftrightarrow(\mathfrak{H}_{\ell})_i < 2 . \] 
 \end{itemize}
\end{proposition}
\begin{proof}
First, we prove the statement for $k\geq 2$. Using Proposition \ref{HCprop}, we get that the hook code for $T^\vee$ is $\mathfrak{h}_{\max}=\left(E,d_2,{d-1}_2,\dots ,1_1 \right)$, and all the Hessians are non-zero. We show that for any partition $P_\ell$ having diagonal lengths $T$, the zero Hessians are those that correspond to the hook codes strictly less than the one for $T^\vee$.
Suppose that there is a vertical part of $P_\ell$, then we have that $\mathfrak{h}_\ell = \left(1_0,2_1,\dots ,a_1, E, \dots \right)$, where $1\leq a\leq d$. By Theorem \ref{2Hessthm} we get that $h^i_{\ell}=0$, for every $i\in \left[d-a-1,d-1\right]$. On the other hand, if $a>1$ we have that  $(\mathfrak{H}_\ell)_{d-1}=0 <1$ and $(\mathfrak{H}_\ell)_{d-i}=1 <2$, for every $i\in \left[d-a-1,d-2\right]$. 

Let $\{ b, b+1,\dots ,b+m\}$ be an increasing sequence after $E$, where $b\geq 1$ and $b+m\leq d$. Corollary \ref{goodhookcor} implies that $(\mathfrak{H}_\ell)_{b}=1$,  if $b=1$ and  $(\mathfrak{H}_\ell)_{b}=2$ otherwise. By Theorem~\ref{2Hessthm} we have that $h^{d-b}_{\ell}\neq 0$, for every $b\geq 1$.

For other elements in this increasing sequence we have that $(\mathfrak{H}_\ell)_{b+t}=1<2$, for every $t\in\left[1,m\right]$. Theorem \ref{2Hessthm} implies that the Hessians corresponding to these elements are zero, in other word we have that $h^{d-b-t}_{\ell}=0$, for every $t\in \left[1,m\right]$.
This completes the proof when $k\ge 2$.\par
Now assume that $k=1$, Proposition \ref{HCprop} implies that $\mathfrak{h}_{\max}=\left(E,E,d_2,{d-1}_2,\dots ,1_2\right)$; and all the Hessians are non-zero. If there is a vertical part in $P_\ell$ then \linebreak$\mathfrak h_\ell=\left(E,1_1,2_1,\dots , a_1,E,\dots \right)$, where $1\leq a\leq d-1$ and  Theorem \ref{2Hessthm} implies that $h^{d-2}_\ell=h^{d-3}_\ell=\cdots =h^{d-a-1}_\ell=0$. If there is an increasing sequence after the second $E$, $\{b,b+1,\dots ,b+m\}$, where $1\leq b\leq d-1$ and $b+m\leq d-1$, Corollary \ref{goodhookcor}, implies that $(\mathfrak{H}_\ell)_b=2$ and $(\mathfrak{h}_\ell)_{b+t}=1<2$, for every $t\in \left[1,m\right]$. On the other hand Theorem~\ref{2Hessthm} implies that $(\mathfrak{h}_\ell)_{d-b-1}\neq 0$ and $(\mathfrak{H}_\ell)_{d-b-t-1}=0$ for every $t\in \left[1,m\right]$.
\end{proof}\par
The next example with Figure \ref{fig9} illustrates the results in Sections~\ref{JTCIsec}, \ref{corrsec} and \ref{CIJTsolsec}. We list all possible Jordan types for an Artinian algebra with the Hilbert function $\left(1,2,3,3,2,1\right)$ together with their hook codes and branch labels and subset of Hessians which vanish and the ranks of the Hessian matrices. 

Here $d$ is the maximum value of the Hilbert function from Equation \eqref{CIT2eq}.\par
\begin{example}\label{lastex}
[$T=(1,2,3,3,2,1)$] There are $18=2\cdot 3^{d-1}$ partitions having diagonal lengths $T$, (\ref{countdiatT1eq}); There are $2^d=8$ complete intersection partitions, by Theorem~\ref{CIpartsthm}. There are also $2^d=8$ different subsets of Hessians corresponding to complete intersection partitions, by Theorem \ref{2Hessthm}. The number of rank sequences possible for the Hessian triple $(\Hess^0,\Hess^1,\Hess^2)$, satisfying Equation \eqref{HessRkSeq1eq} is equal to $2^3=8$.

The maximum hook code is for the strong Lefschetz partition, $T^\vee=(6,4,2)$, and it is equal to $\mathfrak{h}_{\max}= (E,3_2,2_2,1_1)$, by Proposition \ref{HCprop}. Other complete intersection partitions with different hook codes correspond to the vanishing of some of the Hessians (Proposition \ref{HookandHess}). Some partitions, as $P=(6,2^2,1^2)$ or $P=(4,4,1^4)$ have symmetric Jordan type diagrams in the sense of Lemma \ref{symlem} (indicated by Y in next to last column of Fig.\ref{fig9}) but are not CI Jordan types.  By Corollary \ref{geomcor} the set of loci where $k$ active Hessians vanish meet properly: that is, their codimensions add.  
 \end{example}

\vskip 0.2cm

\begin{figure}
\begin{center}
\resizebox{.97 \textwidth}{!} 
{
   $\begin{array}{c||c|c|cccc|c|}
P&\mathfrak H(P) &\mathfrak{b}&\rk\,\Hess^0& \rk\,\Hess^1&\rk\,\Hess^2&Y,N&CIJT\\
\hline\hline
(6,4,2)&(1_3,2_4,2_5)&(E,3,2,1)&1&2&3&Y&Y\\
\hline\hline
(5,5,2)&(1_3,2_4,1_5)&(E,2,3,1)&{\bf 0\ast}&2&3&Y&Y\\
\hline
(6,3,3)&(1_3,1_4,2_5)&(E,3,1,2)&1&{\bf 1\ast}&3&Y&Y\\
\hline
(6,4,1,1)&(0_3,2_4,2_5)&(1,E,3,2)&1&2&{\bf 2\ast}&Y&Y\\
\hline
\hline
(4,4,4)&(1_3,1_4,1_5)&(E,1,2,3)&{\bf 0\ast}&{\bf 1\ast}&3&Y&Y\\
\hline
(5,5,1,1)&(0_3,2_4,1_5)&(1,E,2,3)&{\bf 0\ast}&2&{\bf 2\ast}&Y&Y\\\hline
(6,2,2,2)&(0_3,1_4,2_5)&(1,2,E,3)&1&{\bf 1\ast}&{\bf 2\ast}&Y&Y\\
\hline
(6,3,1,1,1)&(1_3,0_4,2_5)&\color{red}{(2,E,3,1)}&-&-&-&N&N\\
\hline
(5,3,1^4)&(1_3,2_4,0_5)&\color{red}{(3,E,2,1)}&-&-&-&N&N\\
\hline
\hline
(6,2,2,1,1)&(0_3,0_4,2_5)&\color{red}{(2,1,E,3)}&-&-&-&Y&N\\
\hline
(5,2,2,1^3)&(0_3,2_4,0_5)&\color{red}{(3,1,E,2)}&-&-&-&N&N\\
\hline
(4,4,1^4)&(1_3,1_4,0_5)&\color{red}{(3,E,1,2)}&-&-&-&Y&N\\
\hline
(4,2,2,2,2)&(1_3,0_4,1_5)&\color{red}{(2,3,E,1)}&-&-&-&Y&N\\
\hline
(3,3,3,3)&(0_3,1_4,1_5)&(1,2,3,E)&{\bf 0\ast}&{\bf 0\ast}&{\bf 2\ast}&Y&Y\\
\hline
\hline
(4,2^3,1^2)&(1_3,0_4,0_5)&\color{red}{(3,2,E,1)}&-&-&-&Y&N\\
\hline
(3,3,3,1^3)&(0_3,1_4,0_5)&\color{red}{(3,1,2,E)}&-&-&-&N&N\\
\hline
(3,3,2^3)&(0_3,0_4,1_5)&\color{red}{(2,3,1,E)}&-&-&-&Y&N\\
\hline
\hline
(3,3,2,2,1,1)&(0_3,0_4,0_5)&\color{red}{(3,2,1,E)}&-&-&-&Y&N\\
\hline
\end{array}$}
\caption{Jordan types, hook code $\mathfrak{H}(P)$, branch label $\mathfrak{b}$, and ranks of Hessian matrices for $T=(1,2,3,3,2,1)$. Also Y/N for symmetry condition, Y/N for CIJT. Conjugate partitions are located symmetrically about the center line and have complementary hook codes in $(1_3,2_4,2_5)$ as well as reverse branch labels. Note that there are 8 that are CIJT, and they each correspond to a vanishing subset of the Hessians (indicated in bold with $\ast$). The branch labels for non CIJT are in red. See Example \ref{lastex}}\label{fig9}
\end{center}
\end{figure}
\begin{example}[$T=(1,2,3,4,3,2,1)$] There are $3^3=27$ partitions of diagonal lengths $T$, and 14 rank sequences possible for the Hessian triple $(\Hess^0,\Hess^1,\Hess^2)$. The nine partitions of diagonal lengths $T$ having first part $(7)$ behave exactly like their
remainders, a partition of diagonal lengths $T'=(1,2,3,2,1)$ so there are 4 Y and 5 N for CI. For example,  $(7,5,2,1^2)$ is not symmetric in the sense of Lemma \ref{symlem}, since $(5,2,1^2)$ is not. There are 18 more partitions to consider. 
\end{example}

\begin{remark}\label{extendHessrem} The Hessians for codimension two Artinian complete intersections $A=R/I$ correspond to certain Wronskian determinants associated to the homogeneous components $I_i$ of the ideal. Thus, in codimension two, we may regard the Wronskians, which are defined for all graded ideals, as extensions of the Hessians to non-CI algebras. We can then use Wronskians to study Jordan types occurring in non-CI algebras with Hilbert functions that satisfy the more general Equation \eqref{HFeq}.\par For example, by the D. Hilbert-L. Burch theorem, the graded Artinian quotients $A=R/I$ of ${\sf k}[x,y]$ having socle dimension $t$, are defined by ideals $I$ having $t+1$ generators. The Hilbert functions $T$ corresponding to such graded Artinian algebras $A$ can have descents $t_{i-1}-t_i$ at most $t$ (well-known, see, say \cite{Mac}). There is a formula giving the minimal number of generators $\kappa(T)$ possible for $I$ given that $H(A)=T$ (see \cite[Theorem 4.3, Lemma 4.5]{I0} and \cite[\S 3.1]{AIKY}). The following question generalizes to $t>1$ that answered here for $t=1$, the complete intersection case.\vskip 0.2cm\noindent
{\bf Problem.} Let $T$ be a sequence satisfying Equation \eqref{HFeq}, and let $\kappa(T)$ be the least number of generators for a homogeneous ideal $I$ for which the Artinian algebra $A=R/I$ satisfies $H(A)=T$. Determine all partitions $P=P_{\ell,A}$ having diagonal lengths $T$, which are possible for such $A$ where $I$ has $\kappa(T)$ generators. \vskip 0.2cm
 Some of the combinatorial and geometric aspects of this problem had been studied tangentially in \cite{IY,IY2} where a connection is made between the cells $\mathbb V(E_P)$ and Wronskians. 
The goal of \cite{AIKY} is to answer this Problem.
\end{remark}

\begin{ack} This paper began in a working group on Jordan type at the conference ``Lefschetz Properties and Jordan type in Algebra, Geometry, and Combinatorics'' held at C.I.R.M. at Levico, June 24-29, 2018.  Participants in the working group were Nasrin Altafi, Rodrigo Gondim, Tony Iarrobino, Leila Khatami, Gleb Nenashev, Lisa Nicklasson,  Marti Salat Molto, Giorgio Ottaviani, and Yong-Su Shin; all contributed  to the discussion, which began with a presentation by Rodrigo Gondim of a diagram method of
Barbara Costa and himself \cite{CGo}. Subsequently, the authors resolved some of the questions presented at the conference, resulting in this paper. We thank in particular Rodrigo Gondim for lucidly presenting his work and viewpoint on Hessians at the conference, responding to our questions, and comments.  We thank also Joachim Yam\'{e}ogo for his comments, as well as his original work on hook codes, partitions, and the varieties $G_T$. We appreciate helpful comments of Junzo Watanabe. The first author would like to thank Mats Boij for useful and insightful comments and discussion.  We are grateful to  comments/suggestions from the referee. The first author was supported by the grant VR2013-4545.
  \end{ack}

E-mail addresses:  nasrinar@kth.se (N. Altafi), a.iarrobino@northeastern.edu (A. Iarrobino), khatamil@union.edu (L. Khatami).\vskip 0.3cm

\begin{thebibliography}{99}
 \small
 
\bibitem {AIKY}
N. Altafi, A. Iarrobino, L. Khatami, and J. Yam\'{e}ogo: \emph{Jordan type for graded Artinian algebras in height two}, preprint, 2020.

 \bibitem{Bri}
 J. Brian\c{c}on: \emph{Description de Hilb$^n\mathbb C\{ x,y\}$}, Invent. Math. 41 (1977), 45--89.
 
\bibitem{CGo}
B. Costa and R. Gondim: \emph{Jordan type of graded Artinian Gorenstein algebras}, Adv. in Appl. Math 111 (2019), 101941.

\bibitem{Go}
R. Gondim: \emph{On higher Hessians and the Lefschetz properties}, J. Algebra 489 (2017), 241--263.

\bibitem{GoZ}
R. Gondim and G. Zappal\`{a}: \emph{On mixed Hessians and the Lefschetz properties},  
J. Pure Appl. Algebra 223 (2019), no. 10, 4268--4282.

\bibitem{Got}
L. G\"{o}ttsche:  \emph{Betti numbers for the Hilbert function strata of the punctual Hilbert scheme in two variables}, Manuscripta Math. 66, (1990) 253--259. 

\bibitem{HW}
T. Harima and J. Watanabe: \emph{The central simple modules of Artinian Gorenstein algebras}, J. Pure Appl. Algebra 210 (2007), no. 2, 447--463.

\bibitem{H-W}
T. Harima, T. Maeno, H. Morita, Y. Numata, A. Wachi, and J. Watanabe: \emph{The Lefschetz properties}. Lecture Notes in Mathematics, vol. 2080. Springer, Heidelberg, ISBN: 978-3-642-38205-5 (2013), xx+250 pp. 

\bibitem{Hu}
C. Huneke:
\emph{Hyman Bass and ubiquity: Gorenstein rings}, in: Algebra, K-theory, groups, and education (New York, 1997), in: Contemp. Math., 243, Amer. Math. Soc., Providence, RI, 1999, pp. 55--78.

\bibitem{I0}
A. Iarrobino: \emph{Punctual Hilbert schemes}, Mem. Amer. Math. Soc., Vol. 10 (1977), 111 pp., \#188.

\bibitem{IMM}
A. Iarrobino, P. Marques, and C. McDaniel: \emph{Artinian algebras and Jordan type}, preprint, arXiv:math.AC/1802.07383 v.4 (2019).

\bibitem{IY}
A. Iarrobino and J. Yam\'{e}ogo: {\it The family $G_T$ of graded
Artinian quotients of ${\sf k}[x,y]$ of given Hilbert function}, in: Special issue in honor of Steven L. Kleiman, Comm. Algebra 31 \# 8,(2003), 3863--3916.

\bibitem{IY2}
A. Iarrobino and J. Yam\'{e}ogo: \emph{Graded Ideals in ${\sf k}[x,y]$ and partitions of diagonal lengths $T$, 1A: the hook code}, preprint, 2020.

\bibitem{Mac}
F.H.S. Macaulay: \emph{On a method of dealing with the intersections of plane curves}, Trans. Amer. Math. Soc. 5 (4) (1904) 385--410.

\bibitem{Mac2}  
F.H.S. Macaulay: \emph{The algebraic theory of modular systems}, Cambridge Mathematical Library. Cambridge  University Press, Cambridge, 1916. Reissued with an introduction by P.~Roberts, ISBN: 0-521-45562-6 (1994), xxxii+112 pp.

\bibitem{MN}
T. Maeno and Y. Numata: \emph{Sperner property and finite-dimensional Gorenstein algebras associated to matroids,} 
J. Commut. Algebra 8 no. 4 (2016), 549--570.

\bibitem{MW}
T. Maeno and J. Watanabe: \emph{Lefschetz elements of Artinian Gorenstein algebras and Hessians of homogeneous polynomials}, Illinois J. Math. 53 (2009), 593--603.

\bibitem{Se}
J. P. Serre: \emph{Sur les modules projectifs}, in S\'{e}minaire P. Dubreil, M.-L. Dubreil-Jacotin et C. Pisot, 24 14i\'{e}me ann\'{e}e: 1960/61, in: Alg\`{e}bre et th\'{e}orie des nombres, Fasc. 1 (Exp. 2), Facult\'{e} des Sciences
de Paris Secr\'{e}tariat math\'{e}matique, Paris, 1963. 

\bibitem{Y1}
J. Yam\'{e}ogo: \emph{Fibr\'{e}s en droites amples sur des familles d'id\'{e}aux homog\`{e}nes de $C[x,y]$},  Algebraic geometry (Catania, 1993/Barcelona, 1994),  
Lecture Notes in Pure and Appl. Math., 200, Dekker, New York, 1998, pp. 229--244.

\bibitem{Y2}
J. Yam\'{e}ogo: \emph{D\'{e}composition cellulaire de vari\'{e}t\'{e}s param\'{e}trant des id\'{e}aux homog\`{e}nes de C[[x,y]]. Incidence des cellules. I.} (French) [Cellular decomposition of varieties parametrizing homogeneous ideals in C[[x,y]]. Incidence of cells. I] Compositio Math. 90  no. 1, (1994), 81--98. 


\end{thebibliography}
\end{document}